\documentclass{amsart}

\usepackage{amssymb, amsmath,mathtools}
\usepackage{mathrsfs}
\usepackage{amscd}
\usepackage{verbatim}
\usepackage{stmaryrd}
\usepackage[dvipsnames]{xcolor}
\usepackage{a4wide}

\usepackage{enumerate}

\usepackage[colorlinks,linkcolor={blue},citecolor={blue},urlcolor={purple},]{hyperref}

\usepackage[colorinlistoftodos,prependcaption,textsize=tiny]{todonotes}

\usepackage{tikz-cd}
\usetikzlibrary{arrows,positioning}
\tikzset{>=latex,shift left/.style ={commutative diagrams/shift left={#1}},
  shift right/.style={commutative diagrams/shift right={#1}}}

\usepackage{tabularx}
\newcolumntype{L}{>{\arraybackslash}X}
\usepackage{multirow}

\theoremstyle{plain}
\newtheorem{theorem}{Theorem}[section]
\theoremstyle{remark}
\newtheorem{remark}[theorem]{Remark}
\newtheorem{example}[theorem]{Example}
\theoremstyle{plain}

\newtheorem{lemma}[theorem]{Lemma}
\newtheorem{proposition}[theorem]{Proposition}
\newtheorem{definition}[theorem]{Definition}

\newtheorem{assumption}[theorem]{Assumption}

\numberwithin{equation}{section}

\def\N{{\mathbb N}}

\def\R{{\mathbb R}}

\def\T{{\mathbb T}}

\newcommand{\E}{{\mathbb E}}
\renewcommand{\P}{{\mathbb P}}
\newcommand{\F}{{\mathscr F}}

\newcommand{\g}{\gamma}

\newcommand{\om}{\omega}
\renewcommand{\O}{\Omega}

\newcommand{\loc}{{\rm loc}}

\newcommand{\Tor}{\mathbb{T}}
\newcommand{\Dom}{\mathcal{O}}

\newcommand{\Ls}{\mathbb{L}}

\newcommand{\Hs}{\mathbb{H}}

\newcommand{\calL}{\mathcal{L}}

\DeclareMathOperator*{\esssup}{\textup{ess\,sup}}

\newcommand{\wt}{\widetilde}

\usepackage{stmaryrd}

\newcommand{\one}{{{\bf 1}}}
\newcommand{\embed}{\hookrightarrow}
\newcommand{\s}{\delta}

\renewcommand{\div}{\normalfont{\text{div}}}

\newcommand{\supp}{\mathrm{supp}\,}

\newcommand{\norm}[1]{{\left\vert\kern-0.25ex\left\vert\kern-0.25ex\left\vert #1
    \right\vert\kern-0.25ex\right\vert\kern-0.25ex\right\vert}}

\renewcommand{\emptyset}{\varnothing}

\newcommand{\Progress}{\mathcal{P}}

\newcommand{\btwod}{b}

\newcommand{\am}{a}
\newcommand{\bm}{b}

\newcommand{\Borel}{\mathcal{B}}

\def\XXint#1#2#3{{\setbox0=\hbox{$#1{#2#3}{\int}$ }
\vcenter{\hbox{$#2#3$ }}\kern-.6\wd0}}

\newcommand{\ellip}{\nu}

\newcommand{\of}{\overline{f}}
\newcommand{\dd}{\mathrm{d}}

\newcommand{\nn}{|\!|\!|}

\allowdisplaybreaks

\begin{document}

\author{Antonio Agresti}
\address{Institute of Science and Technology Austria (ISTA), Am Campus 1, 3400 Klosterneuburg, Austria} \email{antonio.agresti92@gmail.com}
\curraddr{Delft Institute of Applied Mathematics, Delft University of Technology, P.O.\ Box 5031, 2600 GA Delft, The Netherlands}

\author{Mark Veraar}
\address{Delft Institute of Applied Mathematics\\
Delft University of Technology \\ P.O. Box 5031\\ 2600 GA Delft\\The
Netherlands.} \email{M.C.Veraar@tudelft.nl}

\thanks{The first author has received funding from the European Research Council (ERC) under the Eu\-ropean Union’s Horizon 2020 research and innovation programme (grant agreement No 948819) \includegraphics[height=0.4cm]{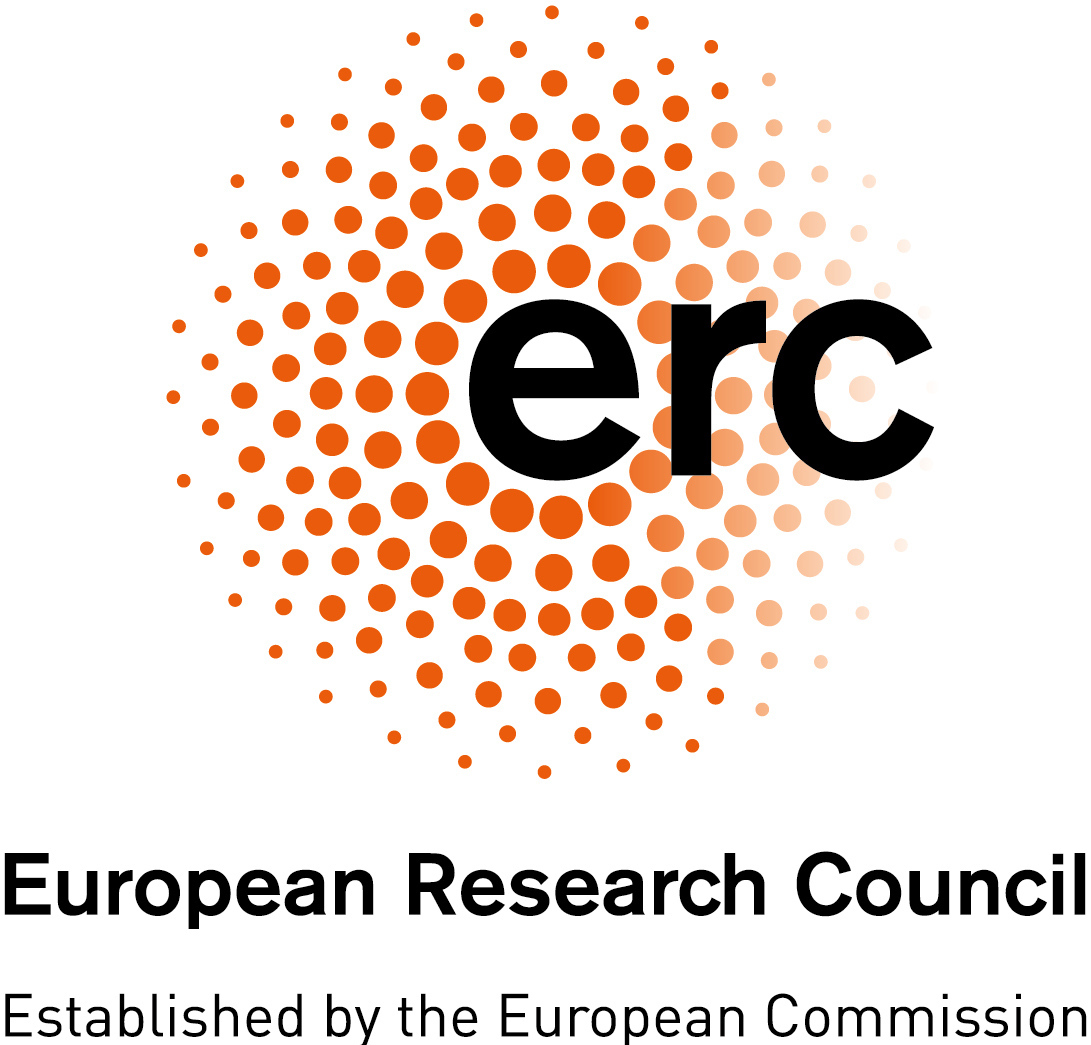}\,\includegraphics[height=0.4cm]{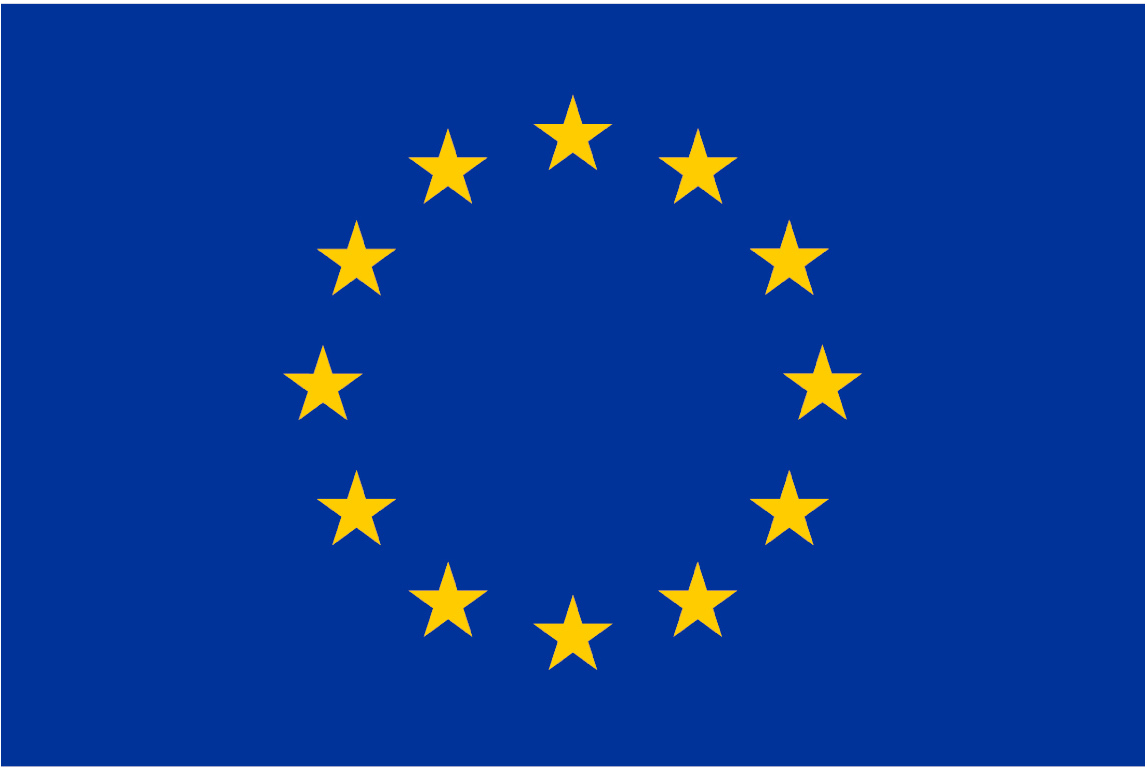}. The second author is supported by the VICI subsidy VI.C.212.027 of the Netherlands Organisation for Scientific Research (NWO)}

\date\today

\title[The critical variational setting for stochastic evolution equations]{The critical variational setting for\\ stochastic evolution equations}



\begin{abstract}
In this paper we introduce the critical variational setting for parabolic stochastic evolution equations of quasi- or semi-linear type. Our results improve many of the abstract results in the classical variational setting. In particular, we are able to replace the usual weak or local monotonicity condition by a more flexible local Lipschitz condition. Moreover, the usual growth conditions on the multiplicative noise are weakened considerably. Our new setting provides general conditions under which local and global existence and uniqueness hold. In addition, we prove continuous dependence on the initial data. We show that many classical SPDEs, which could not be covered by the classical variational setting, do fit in the critical variational setting. In particular, this is the case for the Cahn-Hilliard equation, tamed Navier-Stokes equations, and Allen-Cahn equation.
\end{abstract}

\maketitle
\addtocontents{toc}{\protect\setcounter{tocdepth}{1}}
\tableofcontents

\section{Introduction}

The variational approach to stochastic evolution equations goes back to the work of \cite{bensoussan_equations_1972, KR79, Par75}. Detailed information on the topic can also be found in the monographs \cite{LR15,Rozov}. Further progress was made in the papers \cite{brzezniak_strong_2014, Gyo82}, where the L\'evy and the semimartingale case were considered respectively. Other forms of extensions of the variational setting have been proposed in several papers, and the reader is referred to \cite{Barb19, BarRoc,Gess12, rockner_2010, MarSca18, ScaSte1, ScaSte2} and references therein.
Recently, in the papers \cite{GHV, neelima_2020} a new type of coercivity condition has been found which gives sufficient conditions for $L^p(\Omega)$-estimates as well.

As usual we assume that $(V, H, V^*)$ is a triple of spaces such that $V\hookrightarrow H\hookrightarrow V^*$, where we use the identification of the Hilbert space $H$ and its dual, and where $V^*$ is the dual with respect to the inner product of $H$. In the classical framework $V$ can be a Banach space, but since we will be dealing with  non-degenerate quasi-linear equations only, it will be natural to assume $V$ is a Hilbert space as well. The stochastic evolution equation we consider can be written in the abstract form:
\begin{equation}
\label{eq:SEEintro}
\left\{
\begin{aligned}
        \dd u(t) + A(t,u(t)) \, \dd t  & = B(t,u(t))\, \dd W(t),
\\ u(0) &= u_0.
\end{aligned}
\right.
\end{equation}
Here $A:\R_+\times\Omega\times V\to \calL(V, V^*)$ and $B:\R_+\times\Omega\times V\to \calL_2(U, H)$, and we assume the structure stated in \eqref{eq:ABstruc} (see also Assumption \ref{ass:condFG}). Moreover $W$ is a $U$-cylindrical Brownian motion, where $U$ is a separable Hilbert space.

One of the advantages of the variational approach is that it gives \emph{global well-posedness} for a large class of stochastic evolution equations. Other approaches typically only give local well-posedness and sometimes (after a lot of hard work) global well-posedness under additional conditions.

\subsection{The classical variational setting}

In the variational setting there are several conditions on the nonlinearities $(A,B)$. The most important one is: for $\theta,M>0$
\begin{align}
\label{eq:coercivity_standard_intro}
\langle A(t, v), v \rangle -\tfrac{1}{2} \|B(t, v)\|^2_{\calL_2(U, H)} &\geq \theta\|v\|_V^\alpha - |\phi(t)|^2 - M\|v\|_H^2 & \text{(Coercivity)},
\end{align}
which allows to use It\^{o}'s formula to deduce a priori bounds for the solution. Moreover, in the proof of these bounds, but also at other places, one needs some boundedness properties of $(A,B)$:
\begin{align}
\label{eq:growthA}\|A(t, v)\|_{V^*}^{\frac{\alpha}{\alpha-1}}& \leq K_A(|\phi(t)|^2+\|v\|_V^{\alpha})(1+\|v\|_H^\beta) & \text{(Boundedness of $A$),}
\\
\label{eq:growthB} \|B(t, v)\|_{\calL_2(U, H)}^2 &\leq |\phi(t)|^2+K_B\|v\|_H^2 + K_\alpha\|v\|_V^\alpha   & \text{(Boundedness of $B$).}
\end{align}
In many cases the leading order part of $A$ is linear or quasi-linear and uniformly elliptic, and in that case, the coercivity condition usually forces one to take $\alpha = 2$, in which case $B$ needs to be of linear growth. Moreover, the growth condition on $A$ becomes
\[\|A(t, v)\|_{V^*}^2 \leq K_A(|\phi(t)|^2+\|v\|_V^{2})(1+\|v\|_H^\beta).\]
From an evolution equation perspective it seems more natural to have a symmetric growth condition in terms of intermediate spaces between $H$ and $V$ for the non-leading part of $A$, which is exactly what we assume later on.

In many examples a problematic requirement is the {\em local monotonicity condition}:
\begin{align*}
-2\langle A(t, u)-A(t, v),  u-v\rangle &+ \|B(t, u) - B(t, v) \|^2_{\calL_2(U, H)}  &\\
 &\leq K (1+\|v\|_V^\alpha)(1+\|v\|_H^\beta)\|u-v\|_H^2
 & \  \text{(Local Monotonicity).}
\end{align*}
One of the difficulties with this is that the ``one-sided Lipschitz constant'' on right-hand side can only grow as $\|v\|_H^{\beta}$ (or equivalently $\|u\|_{H}^{\beta}$), but not in $\|u\|_{H}^{\beta}$ and $\|v\|_{H}^{\beta}$ at the same time. Typically it fails whenever spatial derivatives of $u$ appear in a nonlinear way. Examples of nonlinearities for which local monotonicity fails are
\begin{itemize}
\item the Cahn--Hilliard nonlinearity $\Delta( f(u))$, where $f(y) =  \partial_y [\frac14(1-y^2)^2]$;
\item the (tamed) Navier-Stokes nonlinearity $(u\cdot \nabla) u$ in three dimensions;
\item the Allen-Cahn nonlinearity $u - u^3$ in two or three dimensions;
\item systems of SPDEs with nonlinearities such as $\pm u_1 u_2$.
\end{itemize}
For the (tamed) Navier-Stokes and Allen-Cahn equations the strong setting $V = H^2$ and $H = H^1$ is required, for which local monotonicity does not hold.  Each of the above cases contains a product term, which does not satisfy local monotonicity.

The local monotonicity conditions is further weakened in \cite[Section 5.2]{LR15}, where also some of the above examples are considered. However, therein global well-posedness is only obtained for equations with {\em additive noise}. In what follows, we introduce a new setting in which we can remove such restrictions.

Finally, we would like to mention that of course not all SPDEs can be studied with the variational method, due to the lack of coercivity. Equations for which \eqref{eq:coercivity_standard_intro} does not hold include the stochastic primitive equations \cite{agrestiHHS} and certain systems of reaction-diffusion equations \cite{AVreaction-global}.

\subsection{The critical variational setting}
In this paper we introduce a new variational setting which we call {\em the critical variational setting}. We show that under mild structural restrictions on $(A,B)$, one can replace the local monotonicity condition by a local Lipschitz condition, which does hold for the above mentioned examples. The Lipschitz constants can have arbitrary dependence on $\|u\|_H$ and $\|v\|_H$, and moreover they can grow polynomially in $\|u\|_{V_\beta}$ and $\|v\|_{V_\beta}$, where $V_{\beta} = [V^*, V]_{\beta}$ is the complex interpolation space and $\beta<1$.
A restriction on the nonlinearities $(A,B)$ is that they are of quasi-linear type:
\begin{align}\label{eq:ABstruc}
A(t,v) = A_0(t,v)v-F(t,v)  \quad \text{ and } \quad B(t,v) = B_0(t,v)v+G(t,v), \ \
\end{align}
where $v\in V$. The parts $(A_0,B_0)$ and $(F,G)$ are the quasi-linear and semi-linear parts of $(A,B)$, respectively. Of course many of the classical SPDEs fit into this setting, and actually in many important cases $A_0$ and $B_0$ are linear differential operators, which requires $\alpha=2$ in the coercivity condition \eqref{eq:coercivity_standard_intro}.

To give an idea what to expect we give a special case of the local Lipschitz condition on $F$ and $G$ we will be assuming: for all $n\geq 1$ there exists a $C_n$ such that, for all $u,v\in V$ with $\|u\|_H\leq n$ and $\|v\|_H\leq n$,
\begin{align*}
\|F(t,u) - F(t,v)\|_{V^*}&\leq C_n (1+\|u\|^{\rho}_{V_{\beta}}+\|v\|^{\rho}_{V_{\beta}})\|u-v\|_{V_{\beta}},
\\ \|G(t,u) - G(t,v)\|_{\calL_2(U,H)}&\leq C_n (1+\|u\|^{\rho}_{V_{\beta}}+\|v\|^{\rho}_{V_{\beta}})\|u-v\|_{V_{\beta}}.
\end{align*}
Here $\beta\in (0,1)$ and $\rho\geq 0$ satisfy $2\beta\leq 1+\frac{1}{\rho+1}$, together with an associated growth condition involving the same parameters $\beta$ and $\rho$. In case $2\beta= 1+\frac{1}{\rho+1}$ the nonlinearity is called {\em critical}. It is central in the theory that this case is included as well, as often it has the right scaling properties for SPDEs. In examples, the Sobolev embedding usually dictates which $\beta$ one can use. Then the corresponding $\rho$, gives the growth of the Lipschitz constant which one can allow.

Another feature of our theory is that we can also weaken the linear growth condition on the $B$-term in front of the noise term (i.e. \eqref{eq:growthB} in case $\alpha=2$). In principle we can have arbitrary growth, but in order to ensure global well-posedness, the coercivity condition usually leads to some natural restrictions on the growth order of $B$.

Finally, we mention that the classical variational setting also covers several degenerate nonlinear operators $A$ such as the $p$-Laplace operator and the porous media operator (see \cite{LR15}). At the moment we cannot treat these cases since \cite{AV19_QSEE_1, AV19_QSEE_2} requires a strongly elliptic/parabolic quasi-linear structure. It would be interesting to see whether the latter can be adjusted to include these cases as well. This could lead to a complete extension of the variational setting where the local monotonicity condition is replaced by local Lipschitz conditions. After the first version of the current paper was uploaded to the arXiv, the manuscript \cite{rockner2022well} appeared on the arXiv. In the latter it was shown that the classical variational framework can also be extended to a fully nonlinear framework under very general local monotonicity conditions in the special case the embedding $V\hookrightarrow H$ is compact, by applying Yamada-Watanabe theory. Of course the approaches and conditions are completely different, but there is some overlap in the potential applications.

\subsection{Main results}
In the paper we prove well-posedness for \eqref{eq:SEEintro} in the critical variational setting, i.e. assuming \eqref{eq:ABstruc}, a coercivity condition, and a local Lipschitz condition. The main results for the abstract stochastic evolution equation \eqref{eq:SEEintro} are
\begin{itemize}
\item global existence and uniqueness (Theorems \ref{thm:globalclas} and \ref{thm:globaleta0});
\item continuous dependency on the initial data (Theorem \ref{thm:contdepdata}).
\end{itemize}
The abstract results are applied to the following examples:
\begin{itemize}
\item stochastic Cahn--Hilliard equation (Section \ref{ss:CH});
\item stochastic tamed Navier-Stokes equations (Section \ref{ss:TSNS});
\item generalized Burgers equations (Section \ref{ss:second});
\item Allen-Cahn equation (Section \ref{ss:AllenCahn});
\item a quasi-linear equations of second order (Section \ref{ss:quasi});
\item stochastic Swift-Hohenberg equation (Section \ref{ss:SH}).
\end{itemize}
Some of these equations have been considered in the literature before using different methods. However, the classical variational framework was not applicable. In any case the proofs we present are new and lead to new insights. Our new variational setting avoids
local monotonicity and boundedness, and thus leads to less restrictions. Many more applications could be considered, but to keep the paper at a reasonable length, we have to limit the number of applications. We kindly invite the reader to see what the new theory brings for his/her favorite SPDEs.

Our method of proof differs from the classical variational setting which is based on Galerkin approximation. Local well-posedness is an immediate consequence of our paper \cite{AV19_QSEE_1}, which is the stochastic analogue of the theory of critical evolution equations recently obtained in \cite{PrussWeight2, CriticalQuasilinear, addendum}. In our previous work \cite{AV19_QSEE_2} we provided several sufficient conditions for global well-posedness in terms of blow-up criteria. In the current paper we present a consequence of one of these blow-up criteria, and we check it via the coercivity condition and It\^{o}'s formula. This leads to global well-posedness for the variational setting considered here. Continuous dependence on the initial data turns out to be a delicate issue, and is proved via maximal regularity techniques combined with a local Gronwall lemma (see Appendix \ref{app:gronwall}).

\subsection{Embedding into $L^p(L^q)$-theory}

In follow-up papers we will also present other examples which require $L^p(L^q)$-theory ($L^p$ in time and $L^q$ in space). Choosing $p$ and $q$ large, Sobolev embedding results become better, and more nonlinearities can be included. In many cases the $L^2(L^2)$-setting can be transferred to a $L^p(L^q)$-setting using the new bootstrap methods of \cite[Section 6]{AV19_QSEE_2}. Often this also leads to global well-posedness in an $L^p(L^q)$-setting.

An example where the latter program is worked out in full detail can be found in \cite{AV20_NS}. Here we proved higher order regularity of local solutions to the stochastic Navier--Stokes equations on the $d$-dimensional torus. In the special case of $d=2$, these results hold globally and are proved using such a transference and bootstrapping argument, by going from an $L^2(L^2)$-setting to an $L^p(L^q)$-setting.

There is a price to pay in $L^p(L^q)$-theory. It is usually much more involved, and moreover, the classical coercivity conditions are often not enough or not even known. Exceptions where coercivity conditions in $L^p(L^q)$ are well-understood are:
\begin{itemize}
\item second order equations on $\R^d$ \cite{Kry, VP18};
\item second order equations on the torus $\T^d$ \cite{AV21_SMR_torus};
\item weighted spaces on domains \cite{Kim04b,Kim05}.
\end{itemize}
It is of major importance to further extend the latter linear theory as it is the key in the stochastic maximal regularity (SMR) approach to SPDEs obtained in \cite{AV19_QSEE_1, AV19_QSEE_2}. In future papers we expect that the variational setting considered in the current paper will play a crucial role in establishing global $L^p(L^q)$-theory, which in turn also leads to higher order regularity of the solution.

\subsection*{Notation}

For a metric space $(M,d)$, $\mathcal{B}(M)$ denotes the Borel $\sigma$-algebra.

For a measure space $(S,\mathcal{A},\mu)$, $L^0(S;X)$ denotes the space of strongly measurable functions from $S$ into a normed space $X$. Here and below we identify a.e.\ equal functions. For details on strong measurability the reader is referred to \cite{Analysis1}. For $p\in (0,\infty)$ let $L^p(S;X)$ denote the subset of $L^0(S;X)$ for which
\[\|f\|_{L^p(S;X)} :=\Big(\int_S \|f(s)\|^p \,\dd \mu(s)\Big)^{1/p}<\infty.\]
Finally, let $L^\infty(S;X)$ denote the space of functions for which \[\|f\|_{\infty} = \|f\|_{L^\infty(S;X)} = \esssup_{s\in S} \|f(s)\|_X<\infty.\]
In case $\mathcal{A}_0$ is a sub-$\sigma$-algebra of $\mathcal{A}$, $L^p_{\mathcal{A}_0}(S;X)$ denotes the subspace of strongly $\mathcal{A}_0$-measurable functions in $L^p(S;X)$.

Let $\R_+ = [0,\infty)$. For an interval $I\subseteq [0,\infty]$ we write $C(I;X)$ for the continuous functions from $I\cap \R_+$ into $X$. The space $C_b(I;X)$ is the subspace of $C(I;X)$ of functions such that $\|f\|_{\infty} = \|f\|_{C_b(I;X)} = \sup_{t\in I} |f(t)|<\infty$.

We write $L^p_{\rm loc}(I;X)$ for the space of functions $f:I\to X$ such that $f|_{J}\in L^p(J;X)$ for all compact intervals $J\subseteq I$.

For a Hilbert spaces $U$ and $H$,  let $\calL(U,H)$ denote the bounded linear operators from $U$ into $H$, and $\calL_2(U,H)$ the Hilbert-Schmidt operators from $U$ into $H$.

Given Hilbert spaces $V_0$ and $V_1$, $V_{\beta} = [V_{0},V_1]_{\beta}$ denotes complex interpolation at $\beta\in (0,1)$. For $x\in V_{\beta}$ and $R\in \calL_2(U,V_{\beta})$, we write
\begin{equation*}
\|x\|_{\beta} = \|x\|_{V_{\beta}}, \qquad \text{and} \qquad \nn R\nn_{\beta} = \|R\|_{\calL_2(U,V_{\beta})}.
\end{equation*}
However, in case $\beta\in \{0,1/2, 1\}$ we sometimes write $\|x\|_{V^*}$, $\|x\|_{H}$, $\|x\|_V$, $\nn R\nn_{V^*}$, $\nn R\nn_{H}$ and $\nn R\nn_{V}$ instead.

By the notation $A\lesssim_{p} B$ we mean that there is a constant $C$ depending only on $p$ such that $A \leq C B$.

\subsubsection*{Acknowledgements}

The authors thank Udo B\"ohm, Balint Negyesi, Esm\'ee Theewis and the anonymous referees for helpful comments. We thank Sarah Geiss for discussions on stochastic Gronwall inequalities.

\subsection*{Data sharing and conflict of interest statement}
Data sharing not applicable to this article as no datasets were generated or analysed during the current study.

The authors have no competing interests to declare that are relevant to the content of this article.

\section{Preliminaries}

\subsection{Variational setting}\label{ss:variationalsetting}
To start with we recall the variational setting for evolution equations. For unexplained details on real and complex interpolation and connections to Hilbert space methods the reader is referred to \cite[Section 5.5.2]{ArendtHandbook}, \cite{CHM15}, and \cite[Section IV.10]{KPS}. The
real interpolation method is denoted by $(\cdot, \cdot)_{\theta,2}$ and the complex interpolation method by $[\cdot, \cdot]_{\theta}$ for $\theta\in (0,1)$. When interpolating Hilbert spaces, the latter two methods coincide (see \cite[Corollary C.4.2]{Analysis1}).

In the variational setting for evolution equations one starts with two real Hilbert spaces $(V, (\cdot,\cdot)_V)$ and $(H, (\cdot,\cdot)_H)$ such that $V\hookrightarrow H$, where the embedding is dense and continuous. Identifying $H$ with its dual we can define an imbedding $j:H\to V^*$ by setting $\langle j h,v \rangle = (v, h)_H$ for $v\in V$ and $h\in H$. In this way we may view $H$ as a subspace of $V^*$, and one can show that the embedding $H\hookrightarrow V^*$ is dense and continuous. Moreover, one can check that $(V^*, V)_{\frac12,2} = [V^*, V]_{\frac12} = H$ (see \cite[Section 5.5.2]{ArendtHandbook}).

A converse to the above construction can also be done. Given Hilbert spaces $V_0$ and $V_1$ such that $V_1\hookrightarrow V_0$, where the embedding is dense and continuous, define another Hilbert space $H$ by setting $H = (V_0,V_1)_{\frac12,2} = [V_0, V_1]_{\frac12}$. Let $(\cdot, \cdot)_H$ denote the inner product on $H$. Then the embeddings $V_1\hookrightarrow H\hookrightarrow V_0$ are dense and continuous. Identifying $H$ with its dual, one can check that $V_1^* = V_0$ isomorphically.

Next we collect two basic examples of the above construction. Below $H^{s}(\Dom)$ denotes the Bessel potential spaces  of order $s\in \R$ on an open set $\Dom\subseteq \R^d$. Moreover,
$$
H^{s}_0(\Dom):=\overline{C^{\infty}_c(\Dom)}^{H^{s}(\Dom)}.
$$

\begin{example}[Weak setting]\label{ex:weaksetting}
Let $\Dom\subseteq \R^d$ be open.
In case of second order PDEs with Dirichlet boundary conditions (cf.\ Subsection \ref{ss:second}), the weak setting is obtained by letting $H = L^2(\Dom)$, $V = H^1_0(\Dom)$ and $V^*=H^{-1}(\Dom):=H^1_0(\Dom)^*$. In case the domain is $C^1$ the celebrated result of Seeley \cite{Se} yields
$$
[H,V]_{\theta}=
\left\{
\begin{aligned}
&H_0^{ \theta}(\Dom) \ \ \  & \text{ if }&  \theta\in (\tfrac{1}{2},1),\\
& H^{\theta}(\Dom) \ \ \  &\text{ if }& \theta\in (0,\tfrac{1}{2}).
 \end{aligned}\right.
$$
The case $\theta=\frac12$ is more complicated to describe and not needed below. Without any regularity conditions on $\Dom$, due to the boundedness of the zero-extension operator, one always has $[H,V]_{\theta}\hookrightarrow H^{\theta}(\Dom)$.
\end{example}

\begin{example}[Strong setting]
\label{ex:strong_setting}
Let $\Dom$ be a bounded $C^2$-domain.
In the case of second order PDEs with Dirichlet boundary conditions, the strong setting obtained by letting $H = H^1_0(\Dom)$, $V = H^2(\Dom)\cap H^1_0(\Dom)$ and $V^*=L^2(\Dom)$. To show that $V^*=L^2(\Dom)$ it is suffices to note that $H\embed V^*$ is dense, that one has unique solvability in $V$ of the elliptic problem:
\begin{equation*}
\Delta u= \varphi \in L^2(\Dom), \quad \text{ and } \quad u|_{\partial\Dom} =0,
\end{equation*}
and that $\langle\cdot,\cdot\rangle :V^*\times V\to \R$ satisfies $\langle u,v\rangle =\int_{\Dom}\nabla u\cdot\nabla v \,\dd x$ for all $u,v\in H$.
As in the previous example, \cite{Se} implies that
$$
[H,V]_{\theta}=
\left\{
\begin{aligned}
&H^{1+\theta}(\Dom)\cap H_0^{1}(\Dom) \ \ \  & \text{ if }& \theta\in (\tfrac{1}{2},1),\\
& H_0^{1+\theta}(\Dom) \ \ \  &\text{ if }&  \theta\in (0,\tfrac{1}{2}).
 \end{aligned}\right.
$$
As before the case $\theta = \frac12$ is more complicated to describe.
\end{example}

Note that the choice $V = H^2_0(\Dom)$  in Example \ref{ex:strong_setting}  would lead to a larger space for $V^*$, and would also lead to Dirichlet and Neumann boundary conditions at the same time, which is unnatural in many examples, and actually leads to problems.

Finally, if in above examples $\Dom$ is replaced by either $\R^d$ or the periodic torus $\Tor^d$, then the above examples extend verbatim by noticing that $H^{s}_0(\Dom)=H^{s}(\Dom)$ if $\Dom\in \{\Tor^d,\R^d\}$.

Examples \ref{ex:weaksetting} and \ref{ex:strong_setting} can also be extended to the case of Neumann boundary conditions. In the weak setting this leads to $H = H^1(\Dom)$ and $V = L^2(\Dom)$. In the strong setting this leads to $V = \{u\in H^2(\Dom):\partial_n u|_{\partial\Dom} = 0\}$ and $H = H^1(\Dom)$.
Extensions to higher order operators can be considered as well.

\subsection{Stochastic calculus}
\label{ss:stoch_calculus_notation}
For details on stochastic integration in Hilbert spaces the reader is referred to \cite{DPZ,LR15,Rozov}, and for details on measurability and stopping times to \cite{Kal}. For completeness we introduce the notation we will use.

Let $H$ be a Hilbert space. Let $(\Omega, \mathcal{F}, \P)$ be a probability space with filtration $(\F_t)_{t\geq 0}$. An $H$-valued {\em random variable} $\xi$ is a strongly measurable mapping $\xi:\Omega\to H$.

A {\em process} $\Phi:\R_+\times\Omega\to H$ is a strongly measurable function. It is said to be {\em strongly progressively measurable} if for every $t\geq 0$, $\Phi|_{[0,t]\times\Omega}$ is strongly $\mathcal{B}([0,t])\otimes \F_t$-measurable. The $\sigma$-algebra generated by the strongly progressively measurable processes is denoted by $\Progress$. For a stopping time $\tau$ we set $[0,\tau)\times \O:=\{(t,\om)\,:\, 0\leq t<\tau(\om)\}$. Similar definitions hold for $[0,\tau]\times \O$, $(0,\tau]\times \O$ etc. If $\tau$ is a stopping time, then $\Phi:[0,\tau)\times\Omega\to H$ is called strongly progressively measurable if the extension to $\R_+\times\Omega$ by zero is strongly progressively measurable.

For a real separable Hilbert space $U$, let $(W(t))_{t\geq 0}$ be a \emph{$U$-cylindrical Brownian motion} with respect to $(\F_t)_{t\geq 0}$, i.e.\ $W\in \calL(L^2(\R_+;U),L^2(\Omega))$ is such that for all $t\in (0,\infty)$ and $f,g\in L^2(\R_+;U)$ one has
\begin{enumerate}[(a)]
\item $Wf$ has a normal distribution with mean zero and $\E(Wf \, Wg) = (f, g)_{L^2(\R_+;U)}$;
\item $Wf$ is $\F_t$ measurable if $\supp(f)\subseteq [0,t]$;
\item $Wf $ is independent of $\F_t$ if $\supp(f)\subseteq [t,\infty)$.
\end{enumerate}

If $\Phi:\R_+\times\Omega\to \calL_2(U,H)$ is strongly progressively measurable and $\Phi\in L^2_{\loc}(\R_+;\calL_2(U,H))$ a.s., then one can define the stochastic integral process $\int_{0}^\cdot \Phi(s)\,\dd W(s)$, which has a version with continuous paths a.s., and for each interval $J\subseteq [0,\infty)$ the following two-sided estimate holds
\begin{align*}
C_p^{-1} \E\|\Phi\|_{L^2(J;\calL_2(U,H))}^p \leq \E \sup_{t\in J}\Big\|\int_{J\cap [0,t]} \Phi(s) \,\dd W(s)\Big\|_H^p \leq C_p \E\|\Phi\|_{L^2(J;\calL_2(U,H))}^p,
\end{align*}
where $p\in (0,\infty)$ and $C_p>0$ only depends on $p$. The latter will be referred to as the {\em Burkholder-Davis-Gundy inequality}.

\section{Main results}

As in Section \ref{ss:variationalsetting} let $V_0$, $V_1$ and $H$ be Hilbert spaces such that $V_1\hookrightarrow V_0$ and $H = [V_0, V_1]_{1/2}$. Moreover, we set $V_{\theta} = [V_0, V_1]_{\theta}$ for $\theta\in [0,1]$.
By reiteration for the real or complex interpolation method (see \cite[p.\ 50 and 101]{BeLo}) one also has that
\begin{align}\label{eq:reiteration}
V_{\frac{1+\theta}{2}} = [H,V_1]_{\theta} = (H, V_1)_{\theta,2}.
\end{align}
The following short-hand notation will be used frequently: For $x\in V_{\beta}$ and $R\in \calL_2(U,V_{\beta})$,
\[\|x\|_{\beta} = \|x\|_{V_{\beta}} \qquad \text{and} \qquad \nn R\nn_{\beta} = \|R\|_{\calL_2(U,V_{\beta})},\]
where $\beta\in (0,1)$ and $U$ is the Hilbert space for the cylindrical Brownian motion $W$. However, in case $\beta\in \{0,1/2, 1\}$ we prefer to write $\|x\|_{V^*}, \|x\|_{H}$, $\|x\|_V$, $\nn R\nn_{V^*}$, $\nn R\nn_{H}$ and $\nn R\nn_{V}$ instead. The following interpolation estimate will be frequently used for $\beta\in (1/2, 1)$,
\begin{align}\label{eq:interpolest}
\|u\|_{\beta}\leq C\|u\|_{H}^{2-2\beta} \|u\|_{V}^{2\beta-1}, \quad u\in V.
\end{align}

As explained in Section \ref{ss:variationalsetting} we will also write $V = V_1$ and $V^* = V_0$, where the duality relation is given by
\[\langle  h,v\rangle = (v,h)_H, \ \ \  h\in H\text{ and }v\in V,\]
and extended to a mapping $\langle \cdot, \cdot\rangle:V^*\times V\to \R$ by continuity.

\subsection{Setting}
Let $W$ be a $U$-cylindrical Brownian motion (see Subsection \ref{ss:stoch_calculus_notation}).
Consider the following quasi-linear stochastic evolution equation:
\begin{equation}
\label{eq:SEE}
\left\{
       \begin{aligned}
        \dd u(t) + A(t,u(t))\, \dd t  & = B(t,u(t)) \, \dd W(t),
\\ u(0) &= u_0.
             \end{aligned}
           \right.
\end{equation}

The following conditions will ensure local existence and uniqueness.

\begin{assumption}\label{ass:condFG}\
Suppose that the following conditions hold:
\begin{enumerate}[{\rm(1)}]
\item\label{it:condFG0} $A(t,v) = A_0(t,v)v-F(t,v)-f$ and $B(t,v) = B_0(t,v)v+G(t,v)+g$, where
\[A_0:\R_+\times \Omega\times H\to \mathcal{L}(V, V^*) \ \  \text{and} \ \ B_0:\R_+\times \Omega\times H\to \mathcal{L}(V, \mathcal{L}_2(U, H))\]
are $\Progress\otimes\mathcal{B}(H)$-measurable, and
\[F:\R_+\times \Omega\times V\to V^* \ \ \text{and} \ \ G:\R_+\times \Omega\times V\to \calL_2(U,H) \]
are $\Progress\otimes\mathcal{B}(V)$-measurable, and $f: \R_+ \times \O\to V^*$ and $g:\R_+\times \O\to \calL_2(U,H)$ are $\Progress$-measurable maps such that a.s.\
\begin{align*}
 f\in L^2_{\rm loc}([0,\infty);V^*)\ \  \text{ and }\ \     g\in L^2_{\rm loc}([0,\infty);\calL_2(U,H)).
\end{align*}
\item\label{it:condFG1} For all $ T>0$ and $ n\geq 1$, there exist $ \theta_n>0$ and $ M_n>0$ such that\ a.s.
\begin{align*}
 \langle A_0(t,v)u,u \rangle - \frac12\nn B_0(t,v)u\nn_{H}^2 \geq \theta_n\|u\|_{V}^2-M_n\|u\|_{H}^2 ,
\end{align*}
where $t\in [0,T], u\in V$, and $v\in H$ satisfies $\|v\|_{H}\leq n$.
\item\label{it:condFG2}
Let $\rho_j\geq 0$ and $\beta_j\in (1/2,1)$ be such that
\begin{equation}\label{eq:condbetarho}
2\beta_j\leq 1+\frac{1}{\rho_j+1}, \ \ \ \ j\in \{1, \ldots, m_F+m_G\},
\end{equation}
where $m_F, m_G\in \N$, and suppose that
$\forall n\geq 1$ $\forall T>0$ there exists a constant $C_{n,T}$ such that, a.s.\ for all $t\in [0,T]$ and $u,v,w\in V$ satisfying $\|u\|_{H},\|v\|_{H}\leq n$,
\begin{align*}
 \|A_0(t,u) w\|_{V^*}+\nn B_0(t,u) w\nn_{H}&\leq C_{n,T}\|w\|_{V},
\\ \|A_0(t,u)w - A_0(t,v)w\|_{V^*}&\leq C_{n,T} \|u-v\|_{H}\|w\|_{V},
\\ \nn B_0(t,u)w - B_0(t,v) w\nn_{H}&\leq C_{n,T} \|u-v\|_{H}\|w\|_{V},
\\ \|F(t,u) - F(t,v)\|_{V^*}&\leq C_{n,T} \sum_{j=1}^{m_F} (1+\|u\|^{\rho_j}_{\beta_j}+\|v\|^{\rho_j}_{\beta_j})\|u-v\|_{\beta_j},
\\ \|F(t,u)\|_{V^*}&\leq C_{n,T} \sum_{j=1}^{m_F} (1+\|u\|^{\rho_j+1}_{\beta_j}),
\\ \nn G(t,u) - G(t,v)\nn_{H}&\leq C_{n,T} \sum_{j=m_F+1}^{m_F+m_G} (1+\|u\|^{\rho_j}_{\beta_j}+\|v\|^{\rho_j}_{\beta_j})\|u-v\|_{\beta_j},
\\ \nn G(t,u)\nn_{H}&\leq C_{n,T} \sum_{j=m_F+1}^{m_F+m_G}(1+\|u\|^{\rho_j+1}_{\beta_j}).
\end{align*}
\end{enumerate}
\end{assumption}
Here \eqref{it:condFG1} means that the linear part $(A_0(\cdot,v), B_0(\cdot,v))$ is a usual coercive pair with locally uniform estimates for $\|v\|_H\leq n$. Moreover, the quasi-linearities $x\mapsto A_0(\cdot, x)$ and $x\mapsto B_0(\cdot,x)$ are locally Lipschitz. In the semi-linear case, i.e. when $A_0$ and $B_0$ do not depend on $x$, the conditions on $A_0$ and $B_0$ in \eqref{it:condFG2} become trivial.

The nonlinearities $F$ and $G$ satisfy a local Lipschitz and growth estimate, which contains several tuning parameters $\beta_j$ and $\rho_j$ such that \eqref{eq:condbetarho} holds, or equivalently $(2\beta_j - 1)(\rho_j+1)\leq 1$. Usually $m_F = m_G=1$, so that the sums on the left-hand side disappear. The case of equality in \eqref{eq:condbetarho} for some $j$ leads to so-called {\em criticality} of the corresponding nonlinearity. The growth in the $H$-norm can be arbitrary large and this is contained in the constant $C_{n,T}$. These types of conditions have been introduced for deterministic evolution equation in \cite{PrussWeight2, CriticalQuasilinear, addendum}, and in the stochastic framework in \cite{AV19_QSEE_1, AV19_QSEE_2} in a Banach space setting.

\begin{definition}[Solution]\label{def:solution}
\label{def:defsol}
Let Assumption \ref{ass:condFG} be satisfied and
let $\sigma$ be a stopping time with values in $[0,\infty]$. Let $u:[0,\sigma)\times \O\to V$ be a strongly progressively measurable process.
\begin{itemize}
\item $u$ is called a \emph{strong solution} to \eqref{eq:SEE} (on $[0,\sigma]\times \O$) if a.s.\ $u\in L^2(0,\sigma;V)\cap C([0,\sigma];H)$,
\begin{align*}
F(\cdot,u)\in L^2(0,\sigma;V^*), \quad G(\cdot,u)\in L^2(0,\sigma;\calL_2(U,H)),&
\end{align*}
and the following identity holds a.s.\ and for all $t\in[0,\sigma)$,
\begin{equation}
\label{eq:defsol}
u(t)-u_0 +\int_{0}^{t} A(s,u(s))\, \dd s= \int_{0}^t\one_{[ 0,\sigma)\times \O} B(s,u(s))\, \dd W(s).
\end{equation}
\item $(u,\sigma)$ is called a {\em local solution} to \eqref{eq:SEE}, if there exists an increasing sequence $(\sigma_n)_{n\geq 1}$ of stopping times such that $\lim_{n\uparrow \infty} \sigma_n =\sigma$ a.s.\ and $u|_{[ 0,\sigma_n]\times \O}$ is a strong solution to \eqref{eq:SEE} on $[ 0,\sigma_n]\times \O$. In this case, $(\sigma_n)_{n\geq 1}$ is called a {\em localizing sequence} for $(u,\sigma)$.
\item A local solution $(u,\sigma)$ to \eqref{eq:SEE} is called {\em unique}, if for every local solution $(u',\sigma')$ to \eqref{eq:SEE} for a.a.\ $\om\in \Omega$ and for all $t\in [0,\sigma(\om)\wedge \sigma'(\omega))$ one has $u(t,\omega)=u'(t,\omega)$.
\item A unique local solution $(u,\sigma)$ to \eqref{eq:SEE} is called a {\em maximal (unique) solution}, if for any other local solution $(u',\sigma')$ to \eqref{eq:SEE}, we have a.s.\ $\sigma'\leq \sigma$ and  for a.a.\ $\om \in \Omega$ and all $t\in [0,\sigma'(\om))$, $u(t,\omega)=u'(t,\omega)$.
\item A maximal (unique) local solution $(u,\sigma)$ is called a {\em global (unique) solution} if $\sigma=\infty$ a.s. In this case we write $u$ instead of $(u,\sigma)$.
\end{itemize}
\end{definition}

Note that, if $(u,\sigma)$ is a local solution to \eqref{eq:SEE} on $[0,\sigma]\times \O$, then $u\in L^2(0,\sigma;V)\cap C([0,\sigma];H)$ a.s.\ and by Assumption \ref{ass:condFG}
\begin{align}\label{eq:FGint}
A_0(\cdot,u)u\in L^2(0,\sigma;V^*)\text{ a.s.\ }\ \text{ and }\
B_0(\cdot,u)u\in L^2(0,\sigma;\calL_2(U,H))\text{ a.s.\ }
\end{align}
In particular, the integrals appearing in \eqref{eq:defsol} are well-defined.

Maximal local solutions are always unique in the class of local solutions. This seems to be a stronger requirement compared to \cite[Definition 4.3]{AV19_QSEE_1}. However, due to the coercivity condition on the leading operator $(A_0,B_0)$, i.e.\ Assumption \ref{ass:condFG}\eqref{it:condFG1}, stochastic maximal $L^2$-regularity holds by Lemma \ref{lem:SMR} below. Thus \cite[Remark 5.6]{AV19_QSEE_2} shows that Definition \ref{def:defsol} coincide with the one in \cite[Definition 4.3]{AV19_QSEE_1}.

Under Assumption \ref{ass:condFG} the following local existence and uniqueness result holds, which will be proved in Section \ref{sec:proofs}.
\begin{theorem}[Local existence, uniqueness and blow-up criterion]\label{thm:local}
Suppose that Assumption \ref{ass:condFG} holds. Then for every $u_0\in L^0_{\F_0}(\Omega;H)$, there exists a (unique) maximal solution $(u,\sigma)$ to \eqref{eq:SEE} such that a.s.\ $u\in C([0,\sigma);H)\cap L^2_{\rm loc}([0,\sigma);V)$.
Moreover, the following blow-up criteria holds
\begin{align}\label{eq:blowupcrit}
\P\Big(\sigma<\infty, \sup_{t\in [0,\sigma)} \|u(t)\|^2_{H} + \int_{0}^{\sigma} \|u(t)\|_V^2\, \dd t<\infty \Big) = 0.
\end{align}
\end{theorem}
Clearly, \eqref{eq:blowupcrit} is equivalent to
\begin{align*}
\P\Big(\sigma<T, \sup_{t\in [0,\sigma)} \|u(t)\|^2_{H} + \int_{0}^{\sigma} \|u(t)\|_V^2\, \dd t<\infty \Big) = 0 \ \text{ for all }T\in (0,\infty).
\end{align*}

The blow-up criterion \eqref{eq:blowupcrit} is a variant of our results in \cite{AV19_QSEE_2}. In the semi-linear case the result is already contained in \cite[Theorem 4.10(3)]{AV19_QSEE_2}. We will use this result to obtain global well-posedness in several cases below.

\subsection{Global existence and uniqueness}

Under a coercivity condition we obtain two analogues of the existence and uniqueness result in the classical variational framework. In combination with Theorem \ref{thm:contdepdata} below, \emph{global well-posedness} follows.

\begin{theorem}[Global existence and uniqueness I]\label{thm:globalclas}
Suppose that Assumption \ref{ass:condFG} holds, and
for all $T>0$, there exist $\eta,\theta,M>0$ and a progressively measurable $\phi\in L^2((0,T)\times\Omega)$ such that, for any $v\in V$ and $t\in [0,T]$,
\begin{align}\label{eq:pcoercivityclas1}
\langle A(t,v), v\rangle  - (\tfrac{1}{2}+\eta)\nn B(t,v)\nn_{H}^2 \geq \theta\|v\|_{V}^2-M\|v\|_{H}^2-|\phi(t)|^2.
\end{align}
Then for every $u_0\in L^0_{\F_0}(\Omega;H)$, there exists a unique global solution $u$ of \eqref{eq:SEE} such that a.s.
\[u\in C([0,\infty);H)\cap L^2_{\rm loc}([0,\infty);V).\]
Moreover, for each $T>0$ there is a constant $C_T>0$ independent of $u_0$ such that
\begin{align}\label{eq:aprioriLpclas}
\E \|u\|_{C([0,T];H)}^2 + \E\|u\|_{L^2(0,T;V)}^2\leq C_T(1+\E\|u_0\|_{H}^2 + \E\|\phi\|_{L^2(0,T)}^2).
\end{align}
\end{theorem}
In the above result we do not assume any growth conditions on $A$ and $B$ besides the local conditions in Assumptions \ref{ass:condFG}.
The additional $\eta>0$ in the coercivity condition \eqref{eq:pcoercivityclas1} can be arbitrary small and in most examples it does not create additional restrictions. Setting $\eta = 0$ in the coercivity condition \eqref{eq:pcoercivityclas1}, it reduces to the standard one in the variational approach to stochastic evolution equations (see \eqref{eq:coercivity_standard_intro} or \cite[Section 4.1]{LR15}). From a theoretical perspective it is interesting to note that we can also allow $\eta=0$.

\begin{theorem}[Global existence and uniqueness II]\label{thm:globaleta0}
Suppose that Assumption \ref{ass:condFG} holds, and for all $T>0$, there exist $\theta,M>0$ and a progressively measurable $\phi\in L^2((0,T)\times\Omega)$ such that, for any $v\in V$ and $t\in [0,T]$,
\begin{align}\label{eq:pcoercivityclas2}
\langle A(t,v),v\rangle  - \tfrac{1}{2}\nn B(t,v)\nn_{H}^2 &\geq \theta\|v\|_{V}^2-M\|v\|_{H}^2 -|\phi(t)|^2.
\end{align}
Then for every $u_0\in L^0_{\F_0}(\Omega;H)$, there exists a unique global solution $u$ of \eqref{eq:SEE} such that a.s.
\[u\in C([0,\infty);H)\cap L^2_{\rm loc}([0,\infty);V),\]
and the following estimates hold:
\begin{align}\label{eq:aprioriVH1}
\E\int_0^T \|u(t)\|^2_V \, \dd t & \leq C_T(1+\E\|u_0\|_{H}^2+\E\|\phi\|_{L^2(0,T)}^2),
\\ \label{eq:aprioriVH3} \sup_{t\in[0,T]}\E\|u(t)\|^2_{H}& \leq C_T(1+\E\|u_0\|_{H}^2 + \E\|\phi\|_{L^2(0,T)}^2),
\\ \label{eq:aprioriVH2} \E \sup_{t\in [0,T)}\|u(t)\|_{H}^{2\gamma} + \E\Big|\int_0^T \|u(t)\|^2_V \, \dd t\Big|^{\gamma} & \leq C_{\gamma,T} (1+\E\|u_0\|_{H}^{2\gamma} + \E\|\phi\|_{L^2(0,T)}^{2\gamma}),
\end{align}
for every $\gamma\in (0,1)$.
\end{theorem}
The latter result shows that the full statement of the classical variational setting can be obtained in our setting (see Assumption \ref{ass:condFG}) without assuming the growth condition \eqref{eq:growthA} and without the local monotonicity condition.

\begin{remark}[Variants of the coercivity condition]\label{rem:coercivity}
The coercivity condition \eqref{eq:pcoercivityclas1} can be replaced by the following weaker estimate
\begin{align}\label{eq:pcoercivityoud}
\langle A(t,v),v\rangle - \frac{1}{2}\nn B(t,v)\nn_{H}^2 - \eta &\frac{\| B(t,v)^* v\|_U^2}{\|v\|_{H}^2}
\geq \theta\|v\|_{V}^2-M\|v\|_{H}^2-|\phi(t)|^2, \quad v\in V.
\end{align}
To see this, one can easily adapt the proof of Theorem \ref{thm:globaleta0} which will be given under the more restrictive condition \eqref{eq:pcoercivityclas1}. Since it does not give more flexibility in the examples we consider, we prefer to work with \eqref{eq:pcoercivityclas1}.
A variant of condition \eqref{eq:pcoercivityoud} with $\eta = \frac{p-2}{2}$ with $p> 2$ is considered in \cite{GHV}, where it is used to establish bounds on higher order moments of the solution.

In case \eqref{eq:pcoercivityclas1} holds with $\eta=0$ and
\begin{equation}\label{eq:lingrowtheta0}
\|B(t,v)\|_{\mathcal{L}_2(U,H)}\leq C (\phi(t)+\|v\|_{V}), \quad v\in V,
\end{equation}
then \eqref{eq:pcoercivityclas1} also holds for some $\eta>0$ and a slightly worse $\theta>0$. Moreover, \eqref{eq:lingrowtheta0} can even be weakened by replacing $\mathcal{L}_2(U,H)$ by $\mathcal{L}(U,H)$ if one uses the weaker condition \eqref{eq:pcoercivityoud} instead. Note that \eqref{eq:lingrowtheta0} is usually assumed to hold in the classical framework (see \eqref{eq:growthB} with $\alpha=2$).
\end{remark}

\begin{remark}[$\O$-localization of $(u_0,\phi)$]
The function $\phi$ is used to take care of the possible inhomogeneities $f$ and $g$ (see Assumption \ref{ass:condFG}\eqref{it:condFG0}). We have only considered the case that $\phi\in L^2((0,T)\times \Omega)$ for all $T>0$. However, by a standard stopping time argument and using the uniqueness of solutions to \eqref{eq:SEE}, one can also consider the case that for all $T>0$ a.s.\ $\phi\in L^2(0,T)$. Moreover, the estimates of Theorems \ref{thm:globalclas}-\ref{thm:globaleta0} can be used to prove tail estimates for $u$. For instance, \eqref{eq:aprioriLpclas}, the uniqueness of $u$ and the Chebychev inequality readily yield, for all $R,r>0$,
$$
\P\big(\|u\|_{C([0,T];H)}^2 + \E\|u\|_{L^2(0,T);V)}^2>r \big)\leq \frac{C_T(1+R+\E\|u_0\|_{H}^2)
}{r}  + \P(\|\phi\|_{L^2(0,T)}^2 >R).
$$
For more details see the proof of Lemma \ref{lem:gronwall} where a similar argument is employed.

A similar argument can also be employed to obtain tail estimates in case $u_0\in L^0_{\F_0}(\O;H)$.
\end{remark}

\subsection{Continuous dependence on initial data}
Now that we also have global existence and uniqueness of solutions, we can consider the question whether one has continuous dependence on the initial data (in other words, global well-posedness of \eqref{eq:SEE}). This indeed turns out to be the case in the setting of both of the above results. In the case where the monotonicity
\begin{align*}
-2\langle A(t, u)-A(t, v),  u-v\rangle + \|B(t, u) - B(t, v) \|^2_{\calL_2(U, H)} \leq K \|u-v\|_H^2
\end{align*}
holds, continuous dependence is immediate from It\^{o}'s formula (see \cite[Proposition 4.2.10]{LR15}). The following result requires no monotonicity conditions at all, and only assumed the conditions we already imposed for global existence and uniqueness:
\begin{theorem}[Continuous dependence on initial data]\label{thm:contdepdata}
Suppose that the conditions of Theorem \ref{thm:globalclas} or \ref{thm:globaleta0} hold.
Let $u$ and $u_n$ denote the unique global solutions to \eqref{eq:SEE} with strongly $\F_0$-measurable initial values $u_0$ and $u_{0,n}$, respectively. Suppose that $\|u_{0,n}- u_0\|_{H}\to 0$ in probability as $n\to \infty$. Let $T\in (0,\infty)$. Then
\[\|u-u_n\|_{C([0,T];H)} + \|u-u_n\|_{L^2(0,T;V)}\to 0 \ \ \text{in probability as $n\to \infty$.}\]
If additionally $\sup_{n\geq 1}\E\|u_{0,n}\|_{H}^2<\infty$, then for any $q\in (0,2)$
\[\E\|u-u_n\|_{C([0,T];H)}^q + \E\|u-u_n\|_{L^2(0,T;V)}^q\to 0\quad \text{as $n\to \infty$.}\]
\end{theorem}

\begin{remark}[Feller property]
Suppose that the assumptions of Theorem \ref{thm:globalclas} (resp.\ \ref{thm:globaleta0}) hold and let $u_{\xi}$ be the global solution to \eqref{eq:SEE} with data $\xi\in H$. As usual, one defines the operator
$$({\rm P}_t \varphi)(\xi):=\E[\varphi(u_{\xi}(t))]$$
for $t\geq 0$, $\xi\in H$ and $\varphi\in C_{\rm b}(H)$. Theorem \ref{thm:contdepdata}  shows that ${\rm P}_t :C_{\rm b}(H)\to C_{\rm b}(H)$. This is usually referred as \emph{Feller} property and this is the starting point for investigating existence and uniqueness of invariant measures for \eqref{eq:SEE}. However, this is not the topic of this paper.
\end{remark}

\section{Proofs of the main results}\label{sec:proofs}

In several cases we need an a priori estimate for the solution to the linear equation
\begin{equation}
\label{eq:SEElinear}
\left\{
       \begin{aligned}
               \dd u +\wt{A}(t) u(t) \, \dd t &= f(t) \, \dd t + (\wt{B}(t) u(t) + g(t) ) \, \dd W(t),
\\  u(\lambda) &= u_{\lambda},
             \end{aligned}
           \right.
\end{equation}
where $\lambda$ is a stopping time with values in $[0,T]$ and
$(\wt{A}, \wt{B})$ are linear operators satisfying the boundedness and the variational conditions: there exists $\theta,M>0$ such that, a.e.\ on $ \R_+\times \O$ and for all $v\in V$,
\begin{align}
\label{eq:SEElinearcond0_correction}
\|\wt{A}v\|_{V^*}+\nn\wt{B} v\nn_{H}&\leq M \|v\|_V,\\
\label{eq:SEElinearcond}
 \langle v, \wt{A}v\rangle - \frac12\nn \wt{B} v\nn_{H}^2 &\geq \theta\|v\|_{V}^2-M\|v\|_{H}^2.
\end{align}
Finally, $u_\lambda:\Omega\to H$ is strongly $\F_\lambda$-measurable where $\F_{\lambda}$ denotes the $\sigma$-algebra of the $\lambda$-past. If $(\wt{A},\wt{B})$ are progressively measurable, then a solution to \eqref{eq:SEElinear} is defined in a similar way as in Definition \ref{def:solution}.

The following estimate is well-known in case $\lambda$ is non-random  (see \cite[Theorem 4.2.4]{LR15} and its proof). The random case can be obtained by approximation by simple functions (see \cite[Proposition 3.9]{AV19_QSEE_2}). \begin{lemma}[Stochastic maximal $L^2$-regularity]\label{lem:SMR}
Let $\wt{A}:[0,T]\times \Omega\to \calL(V, V^*)$ and $\wt{B}:[0,T]\times \Omega\to \calL(V, \calL_2(U,H))$ be strongly progressively measurable and suppose that there exist $M,\theta>0$ for which	\eqref{eq:SEElinearcond0_correction} and \eqref{eq:SEElinearcond} hold.
Let $f\in L^2((0,T)\times \Omega;V^*)$ and $g\in L^2((0,T)\times \Omega;\calL_2(U,H))$ be strongly progressively measurable and let $u_\lambda\in L^2_{\F_\lambda}(\Omega;H)$.
Then \eqref{eq:SEElinear} has a unique solution
\[u\in L^2(\Omega;C([\lambda,T];H))\cap L^2(\Omega;L^2(\lambda,T;V)),\]
and there is a constant $C$ independent of $(f, g, u_0)$ such that
\[\E\|u\|_{C([\lambda,T];H)}^2 + \E\|u\|_{L^2(\lambda,T;V)}^2 \leq C\Big(\E\|u_\lambda\|_{H}^2+ \E\|f\|_{L^2(\lambda,T;V^*)}^2 + \E\|g\|_{L^2(\lambda,T;\calL_2(U,H))}^2\Big).\]
\end{lemma}

By \cite[Proposition 3.9 and 3.12]{AV19_QSEE_2}, if the stochastic maximal $L^2$-regularity estimate holds on some stochastic interval $[\lambda,T]$, then it also holds on $[\tau,T]$ for all stopping time $\tau\in [\lambda,T]$. Moreover, the constant in the estimate can be chosen to be independent of $\tau$.

\subsection{Proof of Theorem \ref{thm:local}: local existence, uniqueness and blow-up criterion}
\label{ss:proof_local}
The proof of Theorem \ref{thm:local} consists of two parts. First we show local existence and uniqueness via our recent result \cite[Theorem 4.7]{AV19_QSEE_1}.

\begin{proof}[Proof of Theorem \ref{thm:local}: local existence and uniqueness]
Let $n\geq 1$ and set $u_{0,n} = \one_{\{\|u_0\|_{H}\leq n\}} u_0$. For strongly progressively measurable $f_0\in L^2((0,T)\times \Omega;V^*)$ and $g_0\in L^2((0,T)\times \Omega;\calL_2(U,H))$ consider the following linear equation:
\begin{equation}
\label{eq:SEElinearlocal}
\left\{
       \begin{aligned}
               \dd v +A_0(t,u_{0,n}) v(t) \, \dd t &= f_0(t) \, \dd t + (B_0(t,u_{0,n}) v(t) + g_0(t) ) \, \dd W(t),
\\  v(0) &= 0.
             \end{aligned}
           \right.
\end{equation}
By Assumption \ref{ass:condFG}\eqref{it:condFG1} $(A_0(t,u_{0,n}), B_0(t,u_{0,n}))$ satisfies the conditions of Lemma \ref{lem:SMR}. Therefore, \eqref{eq:SEElinearlocal} has a unique solution $v$ and
\begin{align}\label{eq:SMR}
\E\|v\|_{L^2(0,T;V)}^2 + \E\|v\|_{C([0,T];H)}^2\leq C_n\E\|f_0\|_{L^2(0,T;V^*)}^2 + C_n \E\|g_0\|_{L^2(0,T;\calL_2(U,H))}^2.
\end{align}
It follows that $(A_0,B_0)$ satisfies the $\mathcal{SMR}_{2,0}^{\bullet}(0,T)$-condition of \cite[Theorem 4.7]{AV19_QSEE_1}.

By Assumption \ref{ass:condFG} all other conditions of \cite[Theorem 4.7]{AV19_QSEE_1} are satisfied as well (with $p=2$, $\kappa=0$, $F_c = F$ and $G_c = G$). Therefore, we obtain a (unique) maximal solution $(u^T, \sigma^T)$, where $\sigma^T$ takes values in $[0,T]$. Considering $T = m$, $m\in \N$, we can obtain the required maximal solution $(u,\sigma)$ as explained in \cite[Section 4.3]{AV19_QSEE_2}.
\end{proof}

\begin{remark}
In Theorem \ref{thm:local} instead of Assumption \ref{ass:condFG}\eqref{it:condFG1} one could assume the $\mathcal{SMR}_{2,0}^{\bullet}(0,T)$-condition \eqref{eq:SMR} on each $(A_0(\cdot, u_{0,n}),B_0(\cdot, u_{0,n}))$.
\end{remark}

Next we prove the blow-up criteria \eqref{eq:blowupcrit} of Theorem \ref{thm:local} via our recent result \cite[Theorem 4.9]{AV19_QSEE_2}.
\begin{proof}[Proof of Theorem \ref{thm:local}: blow-up criterion]
It suffices to consider $u_0\in L^{\infty}_{\F_0}(\Omega;H)$ (see \cite[Proposition 4.13]{AV19_QSEE_2}). Fix $T\in (0,\infty)$ and set
\begin{align*}
\mathcal{W}_{\sup} &= \Big\{\sigma<T, \  \sup_{t\in [0,\sigma)} \|u(t)\|_{H} +
\int_{0}^{\sigma} \|u(t)\|_{V}^2\, \dd t<\infty\Big\},
\\ \mathcal{W}_{\lim} &= \Big\{\sigma<T, \ \lim_{t\uparrow \sigma}u(t) \ \text{exists in $H$} \ \text{and}
\ \int_{0}^{\sigma} \|u(t)\|_{V}^2\, \dd t<\infty\Big\}.
\end{align*}
For $n\geq 1$ define the stopping times $\sigma_n$ by
\[\sigma_n = \inf\Big\{t\in [0,\sigma): \|u(t)-u_0\|_{H} +  \int_{0}^{t} \|u(t)\|_{V}^2\, \dd t\geq n\Big\}\wedge T,\]
where we set $\inf\emptyset =\sigma\wedge T$.
By definition of $\mathcal{W}_{\sup}$ and $\sigma_n$, we have
\begin{align}\label{eq:probtozero}
\lim_{n\to \infty}\P(\mathcal{W}_{\sup} \cap\{\sigma_n = \sigma\})=\P(\mathcal{W}_{\sup} ).
\end{align}

Define $\wt{A}:[0,T]\times\Omega\to \calL(V, V^*)$ and $\wt{B}:[0,T]\times\Omega\to \calL(V, \calL_2(U,H))$ by
\[\wt{A}(t)v = A_0(t,\one_{[0,\sigma_n)}(t)u(t))v, \ \ \text{and}  \ \ \wt{B}(t)v = B_0(t,\one_{[0,\sigma_n)}(t)u(t))v.\]
By Assumption \ref{ass:condFG}\eqref{it:condFG1}, $(\wt{A},\wt{B})$ satisfy the conditions of Lemma \ref{lem:SMR}.
Let $\wt{f}(t) = F(t,\one_{[0,\sigma_n)}u(t))$ and  $\wt{g}(t) = G(t,\one_{[0,\sigma_n)}u(t))$.
Note that, by the interpolation inequality \eqref{eq:interpolest}, it follows that $\|v\|_{\beta_j}^{(\rho_j+1)}\lesssim\|v\|_{H}^{(2-2\beta_j)(\rho_j+1)}(1+ \|v\|_{V})$ for all $v\in V$ and $j\in \{1,\dots,m_F+m_G\}$, where $(\beta_j,\rho_j,m_F,m_G)$ are as Assumption \ref{ass:condFG}\eqref{it:condFG2}. The definition of $\sigma_n$ shows that $(\wt{f},\wt{g})$ are progressively measurable,
$$
\wt{f}\in L^2((0,T)\times \O ;V^*)\quad \text{ and } \quad \wt{g}\in L^2((0,T)\times \O ;\calL_2(U,H)).
$$
Thus Lemma \ref{lem:SMR} implies that the equation
\begin{equation*}
\left\{
       \begin{aligned}
               \dd v +\wt{A}(t) v(t) \, \dd t &= \wt{f}(t) \, \dd t + (\wt{B}(t) v(t) + \wt{g}(t) ) \, \dd W(t)
\\  v(0) &= u_0.
             \end{aligned}
           \right.
\end{equation*}
has a unique solution $v\in L^2(\Omega;C([0,T];H))\cap L^2(\Omega;L^2(0,T;V))$.
From the definition of $(\wt{A}, \wt{B})$ we see that $(v,\sigma_n)$ is a local solution to \eqref{eq:SEE}. Therefore, by uniqueness $u \equiv v$ on $[0,\sigma_n)$ a.s. In particular, we obtain
\begin{align}\label{eq:sigmauv}
\lim_{t\uparrow \sigma} u(t) = \lim_{t\uparrow \sigma_n} u(t) = \lim_{t\uparrow \sigma_n} v(t) = v(\sigma_n) \quad \text{in $H$   \ a.s.\ on $\{\sigma_n=\sigma <T\}$}
\end{align}
Since $\sigma<T$ on $\mathcal{W}_{\sup}$,
it remains to note that
\begin{align*}
 \P( \mathcal{W}_{\sup}) \stackrel{\eqref{eq:probtozero}}{=} \lim_{n\to \infty} \P(\{\sigma_n=\sigma\}\cap \mathcal{W}_{\sup})  \stackrel{\eqref{eq:sigmauv}}{\leq} \lim_{n\to \infty}\P( \mathcal{W}_{\lim}) =0,
\end{align*}
where in the last step we applied \cite[Theorem 4.9(3)]{AV19_QSEE_2}.
 \end{proof}

\subsection{Proof of Theorem \ref{thm:globalclas}: global existence and uniqueness}

We first prove the following global energy estimates. In the proof of Theorem \ref{thm:globalclas} we will see that $\sigma = \infty$ in the result below.

\begin{proposition}\label{prop:globalesteta}
Suppose that Assumption \ref{ass:condFG} holds, and for all $T>0$, there exist $\eta,\theta,M>0$ and a progressively measurable $\phi\in L^2((0,T)\times\Omega)$ and for any $v\in V$ and $t\in [0,T]$,
\begin{align}\label{eq:pcoercivityclas3}
\langle A(t,v),v\rangle  - (\tfrac{1}{2}+\eta)\nn B(t,v)\nn_{H}^2 &\geq \theta\|v\|_{V}^2-M\|v\|_{H}^2-|\phi(t)|^2.
\end{align}
Let $u_0\in L^2_{\F_0}(\Omega;H)$. Let $(u,\sigma)$ be the maximal solution to \eqref{eq:SEE} provided by Theorem \ref{thm:local}.
Then
for every $T>0$ there is a constant $C_T>0$ independent of $u_0$ such that
\begin{align}\label{eq:apriorietasigma}
\E\sup_{t\in [0,\sigma\wedge T)}\|u(t)\|_{H}^2 +  \E\int_0^{\sigma\wedge T} \|u(t)\|^2_V  \, \dd t
\leq C_T(1+\E\|u_0\|_{H}^2 + \E\|\phi\|_{L^2(0,T)}^2).
\end{align}
\end{proposition}
\begin{proof}
Fix $T\in (0,\infty)$.
Let $(\tau_k)_{k\geq 1}$ be a localizing sequence for $(u, \sigma\wedge T)$. Then in particular $u\in C([0,\tau_k];H)\cap L^2(0,\tau_k;V)$ a.s. For all $k\geq 1$, let
$$
\sigma_{k}:=\inf\Big\{ t\in [0,\tau_k]\,:\, \|u(t)-u_0\|_H \geq k \ \text{ and }\ \int_{0}^{t} \|u(s)\|_V^2 \, \dd s\geq k\Big\}
$$
where $\inf\emptyset:=\tau_k$. Then $(\sigma_k)_{k\geq 1}$ is a localizing sequence as well. Letting $u^k(t) = u(t\wedge \sigma_k)$ we have $u^k\in L^2(\O;C([0,T];H))\cap L^2(\O;L^2(0,\tau_k;V))$.
It suffices to find a $C>0$ independent of $(k,u_0)$ such that for all $t\in [0,T]$,
\begin{equation}
\label{eq:gronwall_application_sub_optimal}
\begin{aligned}
\E\sup_{s\in [0,t]}\|u^k(s)\|_H^2
&+ \E \int_{0}^{t} \one_{[0,\sigma_k]}(s)\|u(s)\|_V^2 \, \dd s\\
&\leq C\Big(1+\E\|u_0\|^2_H+ \E\|\phi\|_{L^2(0,T)}^2 + \E \int_{0}^{t} \|u^k(s)\|_H^2\, \dd s \Big).
\end{aligned}
\end{equation}
Indeed, from the latter Gronwall's lemma first gives
\[\E \sup_{t\in [0,T]}\|u^k(t)\|_H^2  \leq C e^{CT}\Big(1+\E\|u_0\|^2_H + \E\|\phi\|_{L^2(0,T)}^2\Big).\]
This combined with \eqref{eq:gronwall_application_sub_optimal} implies
\[\E \int_{0}^{t} \one_{[0,\sigma_k]}(s)\|u(s)\|_V^2 \, \dd s\leq C'(1+\E\|u_0\|^2_H).\]
It remains to let $k\to \infty$ in the latter two estimates.
The proof of \eqref{eq:gronwall_application_sub_optimal} will be divided into two steps.

{\em Step 1: Proof of the estimates \eqref{eq:ItoH_p_grater_than_2}--\eqref{eq:intermediate_estimate_suboptimal} below}. By It\^o's formula (see \cite[Theorem 4.2.5]{LR15}) applied to $\frac12\|\cdot\|^2_{H} $ we obtain
\begin{equation*}
\begin{aligned}
\frac12\|u^{k}(t) \|_{H}^2 &  + \int_0^t \one_{[0,\sigma_k]}(s)  \mathcal{E}_u(s) \, \dd s \\ &
=\frac12\|u_0\|_{H}^2 + \int_0^t \one_{[0,\sigma_k]}(s)  B(s,u(s))^* u(s)\, \dd W(s),
\end{aligned}
\end{equation*}
where
\[
\mathcal{E}_u(t) = \langle A(t, u(t)), u(t)\rangle - \frac12\nn B(t,u(t))\nn_{H}^2 .
\]
Therefore, by \eqref{eq:pcoercivityclas3}
\begin{equation}\label{eq:ItoH_p_grater_than_2}
\begin{aligned}
\frac{1}{2}  \|u^{k}(t) \|_{H}^2
&+ \int_0^t \one_{[0,\sigma_k]}(s)\Big(\theta\|u(s)\|_V^2+
\eta \nn B(s,u(s))\nn_{H}^2 \Big)  \, \dd s
\\
&\leq  \frac{1}{2}\|u_0\|_{H}^2
+ M\int_0^t \one_{[0,\sigma_k]}(s) \|u(s)\|_{H}^2\, \dd s +\|\phi\|_{L^2(0,t)}^2 \\
& \qquad +\int_0^t \one_{[0,\sigma_k]}(s) B(s,u(s))^* u(s)  \, \dd  W(s).
\end{aligned}
\end{equation}
Taking the expected value on both sides of \eqref{eq:ItoH_p_grater_than_2}, we have
\begin{equation}
\label{eq:intermediate_estimate_suboptimal}
\begin{aligned}
 \E\int_0^t &\one_{[0,\sigma_k]}(s)\Big(\theta\|u(s)\|_V^2+\eta \nn B(s,u(s))\nn_{H}^2\Big) \, \dd s  \\
&\leq  \frac{1}{2}\E\|u_0\|_{H}^2 + \E\|\phi\|_{L^2(0,T)}^2
+ M \E\int_0^t \one_{[0,\sigma_k]}(s) \|u(s)\|_{H}^2 \, \dd s.
\end{aligned}
\end{equation}
In particular, this proves the $V$-term part of the estimate \eqref{eq:gronwall_application_sub_optimal} as $u = u^k$ on $[0,\sigma_k]$.

\emph{Step 2: Estimating the martingale part on the RHS\eqref{eq:ItoH_p_grater_than_2}}. Set
$$
\mathcal{S}_u(t):=\int_0^t \one_{[0,\sigma_k]}(s)  B(s,u(s))^* u(s) \, \dd  W(s).
$$
The Burkholder-Davis-Gundy inequality implies
\begin{align*}
&\E \sup_{s\in [0,t]} |\mathcal{S}_u(s)|  \leq C \E \Big(\int_0^t\one_{[0,\sigma_k]}(s) \|B(s,u(s))^* u(s)\|^2_{U} \, \dd s\Big)^{1/2}
\\ & \leq C \E\big[ \big( \sup_{s\in [0,t]} \|u^k(s)\|_{H}^2\big)^{1/2}
\big(\int_0^T \one_{[0,\sigma_k]}(s) \nn B(s,u(s))\nn^2_{H} \, \dd s\big)^{1/2}\Big]
\\ & \leq \frac{1}{4} \E \sup_{s\in [0,t]} \|u^k(s)\|_{H}^2 +
C \E \int_0^t \one_{[0,\sigma_k]}(s) \nn B(s,u(s))\nn^2_{H}  \, \dd s\\
& \leq
\frac{1}{4} \E \sup_{s\in [0,t]} \|u^k(s)\|_{H}^2 +  C(\E\|u_0\|_{H}^2 + \E\|\phi\|_{L^2(0,T)}^2)
+ CM \E\int_0^t \one_{[0,\sigma_k]}(s) \|u(s)\|_{H}^2 \, \dd s,
\end{align*}
where in the last step we applied \eqref{eq:intermediate_estimate_suboptimal}.
Taking $\E[\sup_{s\in [0,t]}|\cdot|]$ in \eqref{eq:ItoH_p_grater_than_2} and using the above estimate we obtain
\[\E\sup_{s\in [0,t]}\|u^k(s)\|_H^2 \leq C'\Big(1+\E\|u_0\|^2_H + \E\|\phi\|_{L^2(0,T)}^2 + \E \int_{0}^{t} \|u^k(s)\|_H^2\, \dd s \Big).
\]
\end{proof}

The energy estimate of Proposition \ref{prop:globalesteta} allows us to prove Theorem \ref{thm:globalclas} via the blow-up criteria \eqref{eq:blowupcrit} of Theorem \ref{thm:local}.

\begin{proof}[Proof of Theorem \ref{thm:globalclas}]
The proof is divided into two steps.

{\em  Step 1: Proof of $\sigma = \infty$.}  We claim that for every $T>0$
\begin{align}\label{eq:claimHeta}
\sup_{t\in [0,\sigma\wedge T)} \|u(t)\|_{H}<\infty \quad \text{and} \quad
\int_{0}^{\sigma\wedge T} \|u(t)\|_{V}^2 \, \dd t <\infty \ \text{a.s.}
\end{align}
Set $u_{0,j} = \one_{\{\|u_0\|_{H}\leq j\}} u_0$. Let $(v_j,\sigma_j)$ denote the maximal solution to \eqref{eq:SEE}  with initial condition $u_{0,j}$. By localization (see \cite[Theorem 4.7]{AV19_QSEE_1}) $v_j = u$ and $\sigma_j = \sigma$ on $\{\|u_0\|_{H}\leq j\}$. Since by Proposition \ref{prop:globalesteta}
\begin{align*}
\sup_{t\in [0,\sigma\wedge T)} \|v_j(t)\|_{H}<\infty \ \ \text{and} \ \ \int_{0}^{\sigma\wedge T} \|v_j(t)\|_{V}^2 \, \dd t <\infty \ \text{a.s.},
\end{align*}
we see that \eqref{eq:claimHeta} holds on $\{\|u_0\|_{H}\leq j\}$. It remains to let $j\to \infty$.

From the claim \eqref{eq:claimHeta} and the blow-up criteria \eqref{eq:blowupcrit} it follows that
\[\P(\sigma<T) = \P\Big(\sigma<T, \ \sup_{t\in [0,\sigma\wedge T)} \|u(t)\|_{H}<\infty \ \ \text{and} \ \ \int_{0}^{\sigma\wedge T} \|u(t)\|_{V}^2 \, \dd t <\infty\Big) = 0.\]
Therefore, $\sigma \geq  T$ a.s., and since $T$ was arbitrary, we obtain $\sigma = \infty$ a.s.

{\em Step 2: A priori bounds.} The estimate \eqref{eq:aprioriLpclas}
follows from \eqref{eq:apriorietasigma} and $\sigma=\infty$ a.s.
\end{proof}

\subsection{Proof of Theorem \ref{thm:globaleta0}}

First we obtain global estimates under the sharper coercivity estimate. As above, in Theorem \ref{thm:globaleta0} we will see that $\sigma = \infty$ in the results below.

\begin{proposition}\label{prop:globalesteta0}
Suppose that Assumption \ref{ass:condFG} holds, and for all $T>0$, there exist $\theta, M>0$ and a progressively measurable $\phi\in L^2((0,T)\times\Omega)$ and for any $v\in V$ and $t\in [0,T]$,
\begin{align*}
\langle A(t,v),v\rangle  - \tfrac{1}{2}\nn B(t,v)\nn_{H}^2 &\geq \theta\|v\|_{V}^2-M\|v\|_{H}^2 -|\phi(t)|^2.
\end{align*}
Let $u_0\in L^2_{\F_0}(\Omega;H)$. Let $(u,\sigma)$ be the maximal solution to \eqref{eq:SEE} provided by Theorem \ref{thm:local}.
Then
for every $T>0$ and $\gamma\in (0,1)$ there are constants $C_T,C_{\gamma,T}>0$ independent of $u_0$ such that
\begin{align}\label{eq:aprioriVH1prop1eta0}
\E\int_0^{\sigma\wedge T} \|u(t)\|^2_V \, \dd t & \leq C_T(1+\E\|u_0\|_{H}^2 +\E\|\phi\|_{L^2(0,T)}^2),
\\ \E \sup_{t\in [0,\sigma\wedge T)}\|u(t)\|_{H}^{2\gamma} + \E\Big|\int_0^{\sigma\wedge T} \|u(t)\|^2_V \, \dd t\Big|^{\gamma}& \leq C_{\gamma,T} (1+\E\|u_0\|_{H}^{2\gamma} +\E\|\phi\|_{L^2(0,T)}^{2\gamma}).
\label{eq:aprioriVH1prop2eta0}
\end{align}
\end{proposition}
The estimate \eqref{eq:aprioriVH1prop2eta0} can be improved if $B$ has linear growth (see Remark \ref{rem:coercivity}).

\begin{proof}
We begin by repeating the localization argument used in the proof of Proposition \ref{prop:globalesteta}. Throughout the proof we fix $T\in (0,\infty)$.
Let $(\tau_k)_{k\geq 1}$ be a localizing sequence for $(u, \sigma\wedge T)$, cf.\ Definition \ref{def:defsol}. Then in particular, $u\in C([0,\tau_k];H)\cap L^2(0,\tau_k;V)$ a.s. Let
\[\sigma_{k} = \inf\Big\{t\in [0,\tau_k]: \|u(t)-u_0\|_{H}\geq k \ \text{and} \ \int_0^t \|u(s)\|_{V}^2 \, \dd s\geq k \Big\}, \]
where we set $\inf\emptyset  = \tau_k$. Then $(\sigma_k)_{k\geq 1}$ is a localizing sequence for $(u, \sigma)$ as well. Letting $u^k(t) = u(t\wedge \sigma_k)$ we have $u^k\in L^2(\Omega;C([0,T];H))\cap L^2(\Omega;L^2(0,\tau_k;V))$.

The idea will be to eventually apply a stochastic Gronwall lemma. To this end we set $u^{k}(t) = u(t\wedge \sigma_k)$ for $t\in [0,T]$ and $k\geq 1$.
Then a.s.\ for all $t\in [0,T]$,
\begin{align*}
u^{k}(t) = u_0- \int_0^t \one_{[0,\sigma_k]}(s) A(s,u^k(s))   \, \dd s + \int_0^t \one_{[0,\sigma_k]}(s) B(s,u^k(s)) \, \dd W(s).
\end{align*}

{\em Step 1: Proof of the estimate \eqref{eq:ItoHeta0} below}. By It\^o's formula (see \cite[Theorem 4.2.5]{LR15}) applied to $\frac12\|\cdot\|^2_{H}$ we obtain
\begin{align*}
\frac12\|u^{k}(t) \|_{H}^2 &  + \int_0^t \one_{[0,\sigma_k]}(s)  \mathcal{E}_{u^k}(s) \, \dd s
 =\frac12\|u_0\|_{H}^2 + \int_0^t \one_{[0,\sigma_k]}(s) B(s,u^k(s))^* u \, \dd W(s),
\end{align*}
where
\[\mathcal{E}_{u^k}(t) = \langle A(t, u^k(t)), u^k(t)\rangle - \frac12\nn B(t,u^k(t))\nn_{H}^2.\]
Therefore, by the coercivity condition we find that
\begin{equation}\label{eq:ItoHeta0}
\begin{aligned}
\frac{1}{2} \|u^{k}(t) \|_{H}^2  & + \theta  \int_0^t \one_{[0,\sigma_k]}(s)\|u^k(s)\|_{V}^2 \, \dd s
\\ \leq & \frac{1}{2}\|u_0\|_{H}^2
+ M \int_0^t \one_{[0,\sigma_k]}(s) \|u^k(s)\|_{H}^2 \, \dd s + \int_0^t|\phi(s)|_{L^2(0,t)}^2\,\dd s \\ & \qquad +\int_0^t \one_{[0,\sigma_k]}(s) B(s,u^k(s))^* u^k(s)\, \dd  W(s).
\end{aligned}
\end{equation}

{\em Step 2: $u\in L^2(\O;L^2(0,\tau;V))$ and the estimate \eqref{eq:aprioriH22Hwithoutketa0} below holds.}
Taking expectations in \eqref{eq:ItoHeta0} gives that for all $t\in [0,T]$,
\begin{equation}\label{eq:Ito1Heta0}
\begin{aligned}
\frac12\E\|u^{k}(t) \|_{H}^2& + \theta\E \int_0^t \one_{[0,\sigma_k]}  \|u^k(s)\|_{V}^2 \, \dd s \\ & \leq
\frac12\E\|u_0\|_{H}^2+ M \E \int_0^t \one_{[0,\sigma_k]} \|u^k(s)\|_{H}^2 \, \dd s + \E\|\phi\|_{L^2(0,t)}^2.
\end{aligned}
\end{equation}
Therefore, by the classical Gronwall inequality applied to $t\mapsto \E\|u^{k}(t)\|^2_{H}$,
\begin{align}\label{eq:Ito2Heta0}
\sup_{t\in [0,T]}\E\|u^{k}(t) \|_{H}^2\leq  C_T(\E\|\phi\|_{L^2(0,T)}^2+\E\|u_0\|_{H}^2).
\end{align}
Thus also
\begin{equation}\label{eq:Ito3Heta0}
\E\int_0^t  \one_{[0,\sigma_k]}(s) \|u^k(s)\|_{H}^2 \, \dd s
\leq \int_0^t \E(\|u^{k}(s)\|_{H}^2) \, \dd s
\leq T C_T(\E\|\phi\|_{L^2(0,T)}^2+\E\|u_0\|_{H}^2).
\end{equation}
Applying \eqref{eq:Ito2Heta0} and \eqref{eq:Ito3Heta0} in \eqref{eq:Ito1Heta0} also yields the estimate
\begin{align}\label{eq:aprioriH22Heta0}
\E \int_0^t \one_{[0,\sigma_k]}  \|u(s)\|_{V}^2 \, \dd s  \leq C_T'(\E\|\phi\|_{L^2(0,T)}^2+\E\|u_0\|_{H}^2).
\end{align}
Letting $k\to \infty$ in \eqref{eq:Ito3Heta0} and \eqref{eq:aprioriH22Heta0}, by Fatou's lemma we obtain that
\begin{align}\label{eq:aprioriH22Hwithoutketa0}
\E\int_0^{\sigma\wedge T}  \|u(s)\|_{V}^2 ds\leq C_T''(\E\|u_0\|_{H}^2 + \E\|\phi\|_{L^2(0,T)}^2).
\end{align}
In particular \eqref{eq:aprioriVH1prop1eta0} follows.

{\em Step 3: Conclusion}.
From \eqref{eq:ItoHeta0} and the stochastic Gronwall lemma of \cite[Theorem 2.1]{MeSch} with $Z_t=\sup_{s\in [0,t]}\|u(s)\|_{H}^2+\int_0^t \|u(s)\|_V^2\,\dd s$ (also see \cite[Theorem 4.1]{geiss2021sharp}) we obtain
that, for all $\gamma\in (0,1)$, there exists $C_{\gamma,T}>0$ independent of $k\geq 1$ such that
\begin{align*}
\E\Big|\int_0^{\sigma_k\wedge T} \|u(t)\|^2_V \, \dd t\Big|^{\gamma} + \E \sup_{t\in [0,T]}\|u^{k}(t) \|_{H}^{2\gamma}
& \leq C_{\gamma,T} (1+\E\|u_0\|_{H}^{2\gamma} +\E\|\phi\|_{L^2(0,T)}^{2\gamma}).
\end{align*}
Letting $k\to \infty$ we obtain that
\begin{align*}
\E\Big|\int_0^{\sigma\wedge T} \|u(t)\|^2_V \, \dd t\Big|^{\gamma}+
\E \sup_{t\in [0,T]}\|u(t) \|_{H}^{2\gamma}
& \leq C_{\gamma,T} (1+\E\|u_0\|_{H}^{2\gamma} +\E\|\phi\|_{L^2(0,T)}^{2\gamma}).
\end{align*}
Thus \eqref{eq:aprioriVH1prop2eta0} follows.
\end{proof}

By the  estimates of Proposition \ref{prop:globalesteta0} we can check the blow-up criteria \eqref{eq:blowupcrit}:

\begin{proof}[Proof of Theorem \ref{thm:globaleta0}]
The proof is divided into two steps.

{\em  Step 1: Proof of $\sigma = \infty$.}  Reasoning as in Step 1 in the proof of Theorem \ref{thm:globalclas}, by a localization argument, Proposition \ref{prop:globalesteta0} shows that
\begin{align}\label{eq:claimHeta0}
\sup_{t\in [0,\sigma\wedge T)} \|u(t)\|_{H}<\infty \ \ \text{and} \ \ \int_{0}^{\sigma\wedge T} \|u(t)\|_{V}^2 \, \dd t <\infty \ \text{a.s.}
\end{align}
From \eqref{eq:claimHeta0} and the blow-up criteria \eqref{eq:blowupcrit} it follows that
\[\P(\sigma<T) = \P\Big(\sigma<T, \sup_{t\in [0,\sigma\wedge T)} \|u(t)\|_{H}<\infty \ \ \text{and} \ \ \int_{0}^{\sigma\wedge T} \|u(t)\|_{V}^2 \, \dd t <\infty\Big) = 0.\]
Therefore, $\sigma \geq  T$ a.s., and since $T$ was arbitrary, we obtain $\sigma = \infty$ a.s.

{\em Step 2: A priori bounds.}
The bounds \eqref{eq:aprioriVH1} and \eqref{eq:aprioriVH2} are immediate from \eqref{eq:aprioriVH1prop1eta0} and \eqref{eq:aprioriVH1prop2eta0}.
To prove, \eqref{eq:aprioriVH3} note that we can replace
all $\sigma_k$'s by $T$ in the proof of Proposition \ref{prop:globalesteta0},
and thus the required bound follows from \eqref{eq:Ito2Heta0}.
\end{proof}

\subsection{Proof of Theorem \ref{thm:contdepdata}: continuous dependence on initial data}
In this section we use the notation $Z:=C([0,T];H)\cap L^2(0,T;V)$.
The following tail estimate is the key to the proof of the continuous dependency.
\begin{proposition}\label{prop:estimateLpcont}
Suppose that the conditions of Theorem \ref{thm:globaleta0} hold.
Let $u$ and $v$ denote the solution to \eqref{eq:SEE} with initial values $u_0$ and $v_0$ in $L^0_{\F_0}(\Omega;H)$, respectively, and where $\|u_0\|_{L^\infty(\Omega;H)}+\|v_0\|_{L^\infty(\Omega;H)}\leq r$ for some $r>0$. Then for every $T>0$ there exist $\psi_1,\psi_2:[r,\infty)\to (0,\infty)$ both independent of $u_0$ and $v_0$ such that $\lim_{R\to \infty}\psi_2(R) = 0$ and for all $\varepsilon>0$ and $R\geq r$
\begin{align*}
\P(\|u-v\|_Z\geq \varepsilon) \leq \varepsilon^{-2}\psi_1(R) \E\|u_0 - v_0\|_{H}^2 + \psi_2(R)(1+\E\|u_0\|^2_{H} + \E\|v_0\|_{H}^2)
\end{align*}
where $Z=C([0,T];H)\cap L^2(0,T;V)$.
\end{proposition}

\begin{proof}
Let $Z_a^b:=C([a,b];H)\cap L^2(a,b;V)$. Let $w = u-v$. Then $w$ is the solution to
\begin{align*}
       \left\{\begin{aligned}
        \dd w + A_0(\cdot, u) w \, \dd t &= f_w \, \dd t + (B_0(\cdot, u) w  + g_w) \, \dd W(t),
\\        w(0) &= u_0-v_0,
\end{aligned}\right.
\end{align*}
where $f_w = f_1+f_2$ and $g_w = g_1+g_2$ are given by
\begin{align*}
f_1 & = (A_0(\cdot, v) - A_0(\cdot, u))v, & f_2 &= F(\cdot, u) - F(\cdot, v),
\\ g_1 &= (B_0(\cdot, u) - B_0(\cdot, v))v, & g_2  &= G(\cdot, u) - G(\cdot, v).
\end{align*}
In order to derive an a priori estimate for $w$, we want to apply Lemma \ref{lem:SMR} to the pair $(A_0(\cdot, u),B_0(\cdot, u))$. In order to check \eqref{eq:SEElinearcond} we will use Assumption \ref{ass:condFG}\eqref{it:condFG1} and a suitable stopping time argument to ensure $\|u\|_{H}+\|v\|_{H}\leq R$, where $R\geq r$. Let
$$
\tau_R:=\inf\big\{t\in[0,T]\,:\, \|u(t)\|_{H}+\|v(t)\|_{H}\geq R\big\},\quad \text{ where } \inf\emptyset:=T.
$$
Note that $\{\tau_{R}=T\}=\{\sup_{t\in [0,T]}\|u(t)\|_{H} + \|v(t)\|_{H}\leq R\}$. Thus
 \begin{align*}
\P(\|w\|_{Z}\geq \varepsilon)
& \leq \P(\|w\|_{Z_0^{\tau_R}}\geq \varepsilon, \tau_R=T) + \P(\tau_R<T)
\\ & \leq \P \big(\|w\|_{Z_{0}^{\tau_R}}\geq \varepsilon\big)+ \frac{C_T}{R} (1+\E\|u_0\|_{H}^2 + \E\|v_0\|_{H}^2 + \E\|\phi\|_{L^2(0,T)}^2),
\end{align*}
where we used \eqref{eq:aprioriVH2} with $\gamma=1/2$ in the last step.

It remains to estimate $\P \big(\|w\|_{Z_{0}^{\tau_R}}\geq \varepsilon\big)$. To do so we apply the stochastic Gronwall Lemma \ref{lem:gronwall}. Let $0\leq \lambda\leq \Lambda\leq \tau_R$ be stopping times. Let $\wt{w}$ be the solution to
\begin{align}\label{eq:diffeqcont}
       \left\{\begin{aligned}
        \dd \wt{w} + A_0(\cdot, u^{\tau_R})\wt{ w} \, \dd t
        &= \one_{[\lambda,\Lambda]}f_w \, \dd t + (B_0(\cdot, u^{\tau_R}) \wt{w}  + \one_{[\lambda,\Lambda]} g_w) \, \dd W(t),
\\        \wt{w}(\lambda) &= u(\lambda)-v(\lambda).
\end{aligned}\right.
\end{align}
As above $u^{\tau_R}=u(\cdot\wedge \tau_R)$ which is well-defined as $u\in C([0,\infty);H)$.
By uniqueness of solutions to the linear system \eqref{eq:diffeqcont}, we have $\wt{w} = w$ on $[\lambda,\Lambda]$.
Since $(A_0(\cdot, u^{\tau_R}), B_0(\cdot, u^{\tau_R}))$ satisfies \eqref{eq:SEElinearcond} with constant $\theta_R$ and $M_R$, the maximal regularity Lemma \ref{lem:SMR} applied to \eqref{eq:diffeqcont} on $[\lambda,T]$ gives
\begin{align}\label{eq:SMRestw1}
\begin{aligned}
\E \|w\|^2_{Z_{\lambda}^\Lambda}& = \E \|\wt{w}\|^2_{Z_{\lambda}^{T}} \leq C_R\Big(\E\|u(\lambda) - v(\lambda)\|_{H}^2+ \E\|f_{w}\|_{L^2(\lambda,\Lambda;V^*)}^2 + \E\|g_{w}\|_{L^2(\lambda,\Lambda;\calL_2(U,H))}^2\Big),
\end{aligned}
\end{align}
where the constant $C_R$ in \eqref{eq:SMRestw1} is independent of $(\lambda,\Lambda)$ (see below Lemma \ref{lem:SMR}).

Next we estimate the $f_{w}$ and $g_{w}$ term in terms of $u,v$ and $w$ by using Assumption \ref{ass:condFG}\eqref{it:condFG2}. Without loss of generality (by increasing $\rho_j$ if necessary) we can assume $2\beta_j = 1+\frac{1}{\rho_j+1}$ for every $j\in \{1, \ldots, n\}$.
Then
\begin{align*}
\|A_0(\cdot, x)z - A_0(\cdot, y)z\|_{V^*} + \nn B_0(\cdot, x)z - B_0(\cdot, y)z\nn_{H} \leq C_R\|x-y\|_{H} \|z\|_{V},
\\  \|F(t,x) - F(t,y)\|_{V^*} + \nn G(t,x) - G(t,y)\nn_{H} \leq C_R \sum_{j=1}^{m_F+m_G} (1+\|x\|^{\rho_j}_{\beta_j}+\|y\|^{\rho_j}_{\beta_j})\|x-y\|_{\beta_j},
\end{align*}
for all $x,y\in V$  such that $\|x\|_{H},\|y\|_{H}\leq R$, $z\in V$ and $t\in [0,T]$.
From the interpolation estimate $\|x\|_{\beta_j}\leq C\|x\|_{H}^{2-2\beta_j} \|x\|_{V}^{2\beta_j-1}$, we obtain
\begin{align*}
\|x\|_{\beta_j}^{\rho_j} \|z\|_{\beta_j} &\leq \big[C^{1+\rho_j} \|x\|_{H}^{(2-2\beta_j)\rho_j} \|x\|_{V}^{(2\beta_j-1) \rho_j} \|z\|_{H}^{2-2\beta_j}\big] \|z\|_{V}^{2\beta_j-1}
\\ & \leq C_{\delta}  \|x\|_{H}^{\rho_j} \|x\|_{V} \|z\|_{H} + \delta \|z\|_{V},
\end{align*}
where we used Young's inequality with exponents  $1/(2-2\beta_j)$ and $1/(2\beta_j-1)$, and the fact that $\frac{(2\beta_j-1) \rho_j}{2-2\beta_j} = 1$. Therefore,
\begin{align*}
\|F(t,x) - F(t,y)\|_{V^*} + & \ \nn G(t,x) - G(t,y)\nn_{H} \\ & \leq C_{R,\delta} (1+\|x\|_{V}+\|y\|_{V})\|x-y\|_{H} + (m_F+m_G)\delta \|x-y\|_{V}.
\end{align*}

Adding the $(A,B)$-estimate, we can conclude that
\begin{align*}
\|f_w\|_{V^*} + \nn g_w\nn_{H}\leq C_R (1+\|u\|_{V}+\|v\|_{V})\|w\|_{H} + (m_F+m_G) \delta \|w\|_{V}.
\end{align*}
Combining this with \eqref{eq:SMRestw1} and choosing $\delta_R>0$ small enough we find that
\begin{align*}
\E \|w\|^2_{Z_\lambda^\Lambda}\leq C_R' \E\|u(\lambda) - v(\lambda)\|_{H}^2+ C_R' \E \int_\lambda^\Lambda (1+\|u\|_{V}^2+\|v\|_{V}^2)\|w\|_{H}^2 \, \dd t.
\end{align*}

Since $u,v\in L^2(0,T;V)$ a.s.\ and $C_R'$ is independent of $(\lambda,\Lambda)$, we can apply the stochastic Gronwall lemma \ref{lem:gronwall} with $R$ replaced by $C_R' R$, to find a constant $K_R$ such that
\begin{align*}
\P\big(\|w\|_{Z_0^{\tau_R}}\geq \varepsilon \big)
&\leq \varepsilon^{-2}  K_R\E\|u_0-v_0\|^2_{H}+ \P\Big(T+\|u\|_{L^2(0,T;V)}^2+\|v\|_{L^2(0,T;V)}^2 \geq R\Big).
\\ & \leq \varepsilon^{-2}  K_R\E\|u_0-v_0\|^2_{H} + \frac{T+\E\|u\|_{L^2(0,T;V)}^2 + \E\|v\|_{L^2(0,T;V)}^2}{R}
\\ & \leq \varepsilon^{-2}  K_R\E\|u_0-v_0\|^2_{H} + \frac{C_T}{R}(T+\E\|u_0\|^2_{H} + \E\|v_0\|^2_{H}),
\end{align*}
where we used \eqref{eq:aprioriVH1}. Thus the required estimate follows.
\end{proof}

After these preparation we can now prove the continuous dependence result.
\begin{proof}[Proof of Theorem \ref{thm:contdepdata}]
It suffices to prove the result under the conditions of Theorem \ref{thm:globaleta0} as they are weaker than the conditions of Theorem \ref{thm:contdepdata}.
In Step 1 we deal with the uniformly bounded case, and in Step 2 we reduce to this case. Below we fix $T>0$ and recall that  $Z=C([0,T];H)\cap L^2(0,T;V)$.

{\em Step 1: Case of uniformly bounded initial data}. Suppose that there exists a constant $r>0$ such that $\|u_{0,n}\|_{H} + \|u_0\|_{H}\leq r$ a.s.\ for all $n\geq 1$. We will show that for every $\varepsilon>0$,
\begin{align}\label{eq:toprovelimsuptail}
\limsup_{n\to \infty} \P(\|u-u_n\|_Z\geq \varepsilon) = 0.
\end{align}
Let $\varepsilon>0$. By Proposition \ref{prop:estimateLpcont} we obtain that for all $R\geq r$
\begin{align*}
\P(\|u-u_n\|_Z\geq \varepsilon) \leq \varepsilon^{-2}\psi_1(R) \E\|u_0 - u_{0,n}\|_{H}^2 + \psi_2(R) (1+\E\|u_0\|_{H}^2 + \E\|u_{0,n}\|_{H}^2)
\end{align*}
Therefore,
\begin{align*}
\limsup_{n\to \infty} \P(\|u-u_n\|_Z\geq \varepsilon)  \leq \psi_2(R) (1+2\E\|u_0\|_{H}^2).
\end{align*}
Since $\psi_2(R)\to 0$ as $R\to \infty$, \eqref{eq:toprovelimsuptail} follows.

{\em Step 2: General case}. We will show that  $u_{n}\to u$ in $Z$ in probability. By considering subsequences we may additionally suppose that $u_{0,n}\to u_0$ in $H$ a.s. Let $\delta>0$.
By Egorov's theorem we can find $\Omega_{0}\in \F_0$ and $k\geq 1$ such that $\P(\Omega_{0})\geq 1-\delta$ and $\|u_{0,n}\|_{H}\leq k$ on $\Omega_0$.

Let $u^k$ and $u^k_n$ denote the solutions to \eqref{eq:SEE} with initial data
$u^k_{0} = \one_{\Omega_0} u_{0}$ and $u^k_{0,n} = \one_{\Omega_0} u_{0,n}$, respectively. By $\O$-localization (i.e.\  \cite[Theorem 4.7]{AV19_QSEE_1}), $u^k = u$ and $u^k_n = u_n$ on $\Omega_0$. Therefore,
\begin{align}\label{eq:contredtobdd}
\begin{aligned}
\P(\|u-u_n\|_Z>\varepsilon)
&=  \P(\{\|u-u_n\|_Z>\varepsilon\} \cap \Omega_0) + \P(\Omega_0^c) \\ & \leq \P(\|u^k-u_n^k\|_Z>\varepsilon) + \delta
\end{aligned}
\end{align}
Since $u^k_{0,n}\to u^k_{0}$ in $L^2(\Omega;H)$ and are uniformly bounded in $H$, Step 1 implies $\limsup_{n\to \infty} \P(\|u^k-u_n^k\|_Z\geq \varepsilon) = 0$. Therefore, \eqref{eq:contredtobdd} gives $\limsup_{n\to \infty} \P(\|u_n-u\|_Z >\varepsilon) \leq \delta$. Since $\delta>0$ was arbitrary, this implies the required result.

{\em Step 3:} It remains to prove the $L^q$-convergence. Fix $q_0\in (q, 2)$.
Let $C_0:=\sup_{n\geq 1}\|u_{0,n}\|_{L^2(\Omega;H)}<\infty$.
By Fatou's lemma we see that $u_0\in L^2(\Omega;H)$ with $\|u_0\|_{L^2(\Omega;H)}\leq C_0$.
From either \eqref{eq:aprioriLpclas} or \eqref{eq:aprioriVH2}, we see that
\[\E \|u_n\|_{C([0,T];H)}^{q_0} + \E\|u_n\|_{L^2(0,T;V)}^{q_0}\leq C_T(1+C_0),\]
and the same holds for $u$.
Therefore, $\xi_n:=\|u-u_n\|_{C([0,T];H)} + \|u-u_n\|_{L^2(0,T;V)}$ is uniformly bounded in $L^{q_0}(\Omega)$. Since $\xi_n\to 0$ in probability, from \cite[Theorem 5.12]{Kal} it follows that $\xi_n\to 0$ in $L^q(\Omega)$ for any $q\in (0,2)$.
\end{proof}

\section{Applications to stochastic PDEs}
Throughout this section $(w_t^n\,:\,t\geq 0)_{n\geq 1}$ denotes a sequence of standard Brownian motions on
a probability space $(\Omega, \mathcal{F}, \P)$ with respect to a filtration $(\F_t)_{t\geq 0}$.  To such sequence one can associate an $\ell^2$-cylindrical Brownian motion by setting $W_{\ell^2}(f):=\sum_{n\geq 1} \int_{\R_+} f_n(t) dw_t^n$ for $f=(f_n)_{n\geq 1}\in L^2(\R_+;\ell^2)$.

\subsection{Stochastic Cahn--Hilliard equation}\label{ss:CH}
The stochastic Cahn--Hilliard equation was considered in many previous works, and the reader is for instance referred to \cite{CarMil, DaPDeb} for the case of multiplicative and additive noise, respectively.

On an open and bounded $C^2$-domain $\Dom\subseteq \R^d$ consider the Cahn-Hilliard equation with trace class gradient noise term:
\begin{equation}
\label{eq:CahnHilliard}
\left\{
\begin{aligned}
\dd u +\Delta^2 u\, \dd t &= \Delta (f(u))\, \dd t + \sum_{n\geq 1} g_n(u, \nabla u) \, \dd w^n_t, &\quad &\text{ on }\Dom,
\\ \nabla u \cdot n&=0 \quad \text{and} \quad \nabla (\Delta u) \cdot n=0,  &\quad &\text{ on }\partial \Dom,
\\ u(0)&=u_0, &\quad &\text{ on }\Dom.
\end{aligned}\right.
\end{equation}
Unbounded domains could also be considered using a variation of the assumptions below.

We make the following assumptions on $f$ and $g$:
\begin{assumption}\label{ass:CH}
Let $d\geq 1$ and $\rho\in [0,\frac{4}{d}]$. Suppose that $f\in C^1(\R)$ and $g:[0,\infty)\times \Omega\times\Dom\times\R^{1+d}\to \ell^2$ is $\Progress\otimes \Dom \otimes \mathcal{B}(\R^{1+d})$-measurable and there are constants $L, C\geq 0$ such that a.e.\ on $\R_+\times \O\times \Tor^d$ and for all $y,y'\in\R$ and $z,z'\in \R^d$
\begin{align*}
 |f(y)-f(y')|
& \leq L (1+|y|^{\rho}+|y'|^{\rho})|y-y'|,
\\ |f(y)|&\leq L(1+|y|^{\rho+1}),
\\ f'(y)&\geq -C,
\\ \|g(\cdot,y,z)-g(\cdot,y',z')\|_{\ell^2}
&\leq L(|y-y'|+|z-z'|),
\\ \|(g(\cdot,y,z))_{n\geq 1}\|_{\ell^2} &\leq L(1+|y|+|z|).
\end{align*}
\end{assumption}
The standard example $f(y) = y(y^2-1) = \partial_y [\frac14(1-y^2)^2]$ (double well potential) satisfies the above conditions for $d\in \{1, 2\}$.
Note that in this case the nonlinearity $\Delta f(u)$ does not satisfy the classical local monotonicity condition for stochastic evolution equations, and therefore there are difficulties in applying the classical variational framework to obtain well-posedness for \eqref{eq:CahnHilliard}. In the case of additive noise, well-posedness was studied in \cite[Example 5.2.27 and Remark 5.2.28]{LR15}. Our setting applies in the setting of multiplicative noise under the same conditions on the nonlinearity $\Delta f(u)$.

Let
\[H^2_N(\Dom) = \{u\in H^2(\Dom): \partial_n u|_{\partial\Dom}  =0\},\]
where $\partial_n u = \nabla u \cdot n$ and $n$ denotes the outer normal vector field on $\partial\Dom$.

\begin{theorem}[Global well-posedness]
\label{thm:CH}
Suppose that Assumption \ref{ass:CH} holds. Let $u_0\in L^0_{\F_0}(\Omega; L^2(\Dom))$.
Then \eqref{eq:CahnHilliard} has a unique global solution
\begin{equation*}
u\in C([0,\infty);L^2(\Dom))\cap L^2_{\rm loc}([0,\infty);H^2_N(\Dom)) \ a.s.
\end{equation*}
and for every $T>0$ there exists a constant $C_T$ independent of $u_0$ such that
\begin{align}\label{eq:aprioriLpCH}
\E \|u\|_{C([0,T];L^2(\Dom))}^2 + \E\|u\|_{L^2(0,T;H^2(\Dom))}^2\leq C_T(1+\E\|u_0\|_{L^2(\Dom)}^2).
\end{align}
Finally,  $u$ depends  continuously on the initial data $u_0$ in probability in the sense of Theorem \ref{thm:contdepdata} with $H=L^2(\Dom)$ and $V = H^2_N(\Dom)$.
\end{theorem}
\begin{proof}
We first formulate \eqref{eq:CahnHilliard} in the form \eqref{eq:SEE}.
Let $H = L^2(\Dom)$ and $V = H^2_N(\Dom)$.
By \eqref{eq:reiteration} for $\theta\in (0,1)$ one has
\begin{align*}
V_{\frac{1+\theta}{2}} = [H, V]_{\theta} \hookrightarrow [L^2(\Dom), H^{2}(\Dom)]_{\theta} = H^{2\theta}(\Dom),
\end{align*}
where in the last step we used the smoothness of $\Dom$ and standard results on complex or real interpolation.

Let $A = A_0 + F$, where we define $A_0\in \calL(V,V^*)$ and $F:V\to V^*$  by
\[\langle A_0 u, v\rangle = (\Delta v, \Delta u)_{L^2}, \quad \text{and} \quad \langle F(u),v\rangle = (\Delta v, f(u))_{L^2} = -(\nabla v, \nabla (f(u)))_{L^2}.\]
Let $B = B_0 + G$, where $B_0 = 0$ and $G:[0,\infty)\times \Omega\times V\to \calL_2(U,H)$ is defined by
\[(G_n(t,u))(x) = g_n(t,x,u(x), \nabla u(x)).\]

In order to apply Theorem \ref{thm:globalclas} we check Assumption \ref{ass:condFG}.
By the smoothness of $\Dom$ and standard elliptic theory for second order operators (see \cite[Theorem 8.8]{GT83}) there exist $\theta, M>0$ such that for all $u\in V$
\[\|u\|_{H^2(\Dom)}^2\leq \theta \|\Delta u\|_{L^2(\Dom)}^2 + M\|u\|_{L^2(\Dom)}^2.\]
Hence, for all $u\in V$,
\begin{align}\label{eq:coerCH1}
 \langle A_0 u,u\rangle = \|\Delta u\|_{L^2(\Dom)}^2 \geq \theta\|u\|_{V}^2-M\|u\|_{H}^2.
\end{align}
Note that
\begin{align*}
& \|F(u) - F(v)\|_{V^*} \lesssim \|f(\cdot, u) - f(\cdot, v)\|_{L^2(\Dom)}
\\ & \lesssim  \|(1+|u|^{\rho}+|v|^{\rho}) (u-v)\|_{L^{2}(\Dom)} & \text{(by Assumption \ref{ass:CH})}
\\ & \lesssim (1+\|u\|^{\rho}_{L^{2(\rho+1)}(\Dom)}+\|v\|^{\rho}_{L^{2(\rho+1)}(\Dom)} )\|u-v\|_{L^{2(\rho+1)}(\Dom)} & \text{(by H\"older's inequality)}
\\ & \lesssim (1+\|u\|^{\rho}_{H^{4\beta-2}(\Dom)}+\|v\|^{\rho}_{H^{4\beta-2}_0(\Dom)})  \|u-v\|_{H^{4\beta-2}(\Dom)} & \text{(by Sobolev embedding).}
\end{align*}
In the Sobolev embedding we need $4\beta -2- \frac{d}{2} \geq -\frac{d}{2(\rho+1)}$. Therefore, the condition \eqref{eq:condbetarho} leads to $\rho\leq \frac{4}{d}$.
Moreover, we can consider the critical case $2\beta = 1+\frac{1}{\rho+1}$.

One easily checks that $G$ satisfies Assumption \ref{ass:condFG}\eqref{it:condFG2} with $\rho_2 = 0$. Indeed,
\begin{align*}
\|G(t,u) - G(t,v)\|_{L^2(\Dom;\ell^2)}   & \lesssim \|u - v\|_{L^2(\Dom)} + \|\nabla u - \nabla v\|_{L^2(\Dom)}
\lesssim \|u-v\|_{H^{4\beta_2-2}_0(\Dom)},
\end{align*}
where $\beta_2 = 3/4$. The growth estimate can be checked in the same way:
\begin{align*}
\|G(t,u)\|_{L^2(\Dom;\ell^2)}  & \lesssim 1+ \|u\|_{L^2(\Dom)} + \|\nabla u\|_{L^2(\Dom)}
\lesssim 1+ \|u\|_{H^{4\beta_2-2}_0(\Dom)},
\end{align*}
where we used $|\Dom|<\infty$.

Finally concerning the coercivity condition \eqref{eq:pcoercivityclas1} note that interpolation estimates give that for every $\varepsilon>0$ there exists a $C_{\varepsilon}$ such that
\begin{align}
\nonumber
\langle F(u),u\rangle
& = -(\nabla u, \nabla (f(u)))_{L^2}
\\
\label{eq:coerCH3}
 &\leq  -\int_{\Dom} f'(u) |\nabla u|^2 \, \dd x
\\
\nonumber
& \leq C \|\nabla u\|_{L^2(\Dom)}^2\leq \varepsilon \|u\|_{V}^2 + C_{\varepsilon} \|u\|_{L^2(\Dom)}^2.
\end{align}
Similarly, $\|G(t,u)\|_{L^2(\Dom;\ell^2)}^2\leq \varepsilon \|u\|_{V}^2 + C_{\varepsilon} (1+\|u\|_{L^2(\Dom)}^2)$.
Therefore, combining this with \eqref{eq:coerCH1} and \eqref{eq:coerCH3}, we obtain that for $\kappa : = \frac{1}{2}+\eta$ (with $\eta>0$ arbitrary)
\begin{align*}
&\langle A(u),u\rangle - \kappa \nn B(\cdot, u)\nn_{H}^2 \\ & = \langle A_0 u,u\rangle - \langle  F(u),u\rangle -  \kappa \nn G(\cdot, u)\nn_{H}^2 \\ & \geq \theta\|u\|_{V}^2-M\|u\|_{H}^2 - \varepsilon\|u\|_{V}^2 - C_{\varepsilon} \|u\|_{L^2(\Dom)}^2 - \kappa C_{\varepsilon}(1+\|u\|_{L^2(\Dom)}^2) - \kappa\varepsilon\|u\|_{V}^2 \\ & = (\theta-(1+\kappa)\varepsilon)\|u\|_{V}^2-\wt{M}\|u\|_{H}^2.
\end{align*}
Thus taking $\varepsilon>0$ small enough, the result follows from Theorems \ref{thm:globalclas} and \ref{thm:contdepdata}.
\end{proof}

\begin{remark}
It is also possible to add a noise term of the form:
\begin{align*}
B_0(t) u(x) = \sum_{n\geq 1} \sum_{|\alpha|=2}b_{n,\alpha}(t,x) \partial^{\alpha} u(x).
\end{align*}
where we need the stochastic parabolicity condition: for some $\lambda>0$ and all $u\in V$,
\[\int_{\Dom} |\Delta u(x)|^2 \, \dd x -\frac12 \sum_{n\geq 1} \int_{\Dom}\Big|\sum_{|\alpha|=2}b_{n,\alpha}(\cdot,x) \partial^{\alpha}u(x) \Big|^2 \, \dd x\geq \lambda \int_{\Dom} |\Delta u(x)|^2 \, \dd x.\]
Depending on the precise form of $f$, one can allow superlinear $g$ as well. The only requirement is
\begin{equation}
\label{eq:CH_coercivity_sharp}
(\tfrac12+\eta)\int_{\Dom} \sum_{n\geq 1} |g_n(t,x,u, \nabla u)|^2 \, \dd x
\leq \int_{\Dom} f'(u) |\nabla u|^2 \, \dd x + \|\Delta u\|_{L^2}^2+ C\|u\|_{L^2}^2+C
\end{equation}
for some $\eta>0$.
However, for constant functions $u$, the right-hand side only grows quadratically. If $f(u) = y(y^2-1)$, then it would be possible to consider nonlinearities of the form $g(u,\nabla u) = (4-2\eta)u\nabla u$ as well. Using Theorem \ref{thm:globaleta0} instead one can even take $\eta=0$ if $\|\Delta u\|_{L^2}^2$ is replaced by $\varepsilon\|\Delta u\|_{L^2}^2$ for some $\varepsilon<1$ in \eqref{eq:CH_coercivity_sharp}.  However the estimate \eqref{eq:aprioriLpCH} needs to be replaced by \eqref{eq:aprioriVH1}-\eqref{eq:aprioriVH2}.
\end{remark}

\subsection{Stochastic tamed Navier-Stokes equations}\label{ss:TSNS}

In \cite{RockZha} the stochastic tamed Navier-Stokes equations with periodic boundary conditions were considered. In \cite{BrzDha} the problem was studied on $\R^3$. Both these papers construct martingale solutions and prove pathwise uniqueness. Below we show that one can argue more directly using our framework. We will consider the problem on the full space, but the periodic case can be covered by the same method. Remarks about the case of Dirichlet boundary conditions can be found in Remark \ref{rem:DirichletSNS}.

On $\R^3$ consider the tamed Navier-Stokes equations with gradient noise:
\begin{equation}
\label{eq:Navier_Stokes}
\left\{
\begin{aligned}
\dd u &=\big[\Delta u -(u\cdot \nabla)u- \phi_N(|u|^2) u -\nabla p  \big] \, \dd t \\ & \qquad +\sum_{n\geq 1}\big[(\btwod_{n}\cdot\nabla) u -\nabla \wt{p}_n+g_n(\cdot,u)\big] \, \dd w^n_t, \ \ &\text{ on }\R^3,
\\
\div \,u&=0, \ \ &\text{ on }\R^3,
\\ u(0)&=u_0 \ \ &\text{ on }\R^3.
\end{aligned}\right.
\end{equation}
Here $u:=(u^k)_{k=1}^d:[0,\infty)\times \O\times \R^3\to \R^3$ denotes the unknown velocity field, $p,\wt{p}_n:[0,\infty)\times \O\times \R^3\to \R$ the unknown pressures,
\begin{equation*}
(\btwod_{n}\cdot\nabla) u:=\Big(\sum_{j=1}^3 \btwod_n^j \partial_j u^k\Big)_{k=1}^3 \quad \text{ and }\quad
(u\cdot \nabla ) u:=\Big(\sum_{j=1}^3 u^j \partial_j u^k\Big)_{k=1}^3.
\end{equation*}
The function $\phi_N:[0,\infty)\to [0,\infty)$ is a smooth function such that \[\phi_N(x) = 0 \ \text{for $x\in [0,N]$}, \ \ \phi_N(x) = x-N \ \text{for $x\geq N+1$} \ \text{and} \ 0\leq \phi'_N\leq 2.
\]
In the deterministic setting the motivation to study \eqref{eq:Navier_Stokes} comes from the fact if $u$ is a strong solution to the usual Navier-Stokes equations and $\|u\|_{L^\infty((0,T)\times \R^3)}^2\leq N$, then it is also a solution to \eqref{eq:Navier_Stokes} for $N$ large enough. On the other hand, it is possible to give conditions under which \eqref{eq:Navier_Stokes} has a unique \emph{global} strong solution.

\begin{assumption}\label{ass:TSNS}
Let $d=3$. Let $b^j\in W^{1,\infty}(\R^3;\ell^2)$ and set $\sigma^{ij} = \sum_{n\geq 1} b_n^i b_n^j$ for $i,j\in \{1, 2, 3\}$, and suppose that a.e.\ on $\R_+\times\Omega\times\R^3$
\begin{align}\label{eq:stochparSNS}
\sum_{i,j=1}^3\sigma^{ij} \xi_i \xi_j \leq |\xi|^2 \quad \text{for all $\xi\in \R^d$.}
\end{align}
Suppose that there exist $M,\delta>0$ such that for all $j\in \{1,2,3\}$ a.s.\ for all $t\in \R_+$,
$$
\|b^j(t,\cdot)\|_{W^{1,\infty}(\R^3;\ell^2)}\leq M.
$$
The mapping $g:\R_+\times \O\times \R^3\times \R^3\to \ell^2(\N;\R^3)$ is $\Progress\otimes \Borel(\R^3)\otimes \Borel(\R^3)$-measurable. Moreover, assume that for each $t\geq 0$, $g(t,\cdot) \in C^1(\R^3\times\R^3)$ and
there exists $C\geq 0$ such that a.e.\ on $\R_+\times\O\times \R^3$ and for all $y,y'\in \R^3$,
\begin{align}
\label{eq:condgLipSNS}
\|g(\cdot, y) - g(\cdot, y')\|_{\ell^2} +
\| \nabla g(\cdot,y)-\nabla g(\cdot,y')\|_{\ell^2}&\leq C|y-y'|,
\\ \|g(t,\cdot,0)\|_{L^2(\R^3;\ell^2)} + \|\nabla g(t,\cdot,0)\|_{L^2(\R^3;\ell^2)}&\leq C,\label{eq:condgzeroSNS}
\end{align}
where the gradient is taken with respect to $(x,y)\in \R^3\times\R^3$.
\end{assumption}
Note that, as in previous works on \eqref{eq:Navier_Stokes}, our stochastic parabolicity condition \eqref{eq:stochparSNS} is not optimal. This is due to the fact that the taming term $-\phi_N(|u|^2)u$ needs to be handled as well.

Let $U = \ell^2$, $V = \Hs^2$, $H = \Hs^1$ and $V^* = \Ls^2$, where for $k\in \{0, 1, 2\}$ we set
\[
\Hs^k = \{u\in H^k(\R^3;\R^3): \div\,u = 0 \ \text{in} \ \mathscr{D}'(\R^3)\} \quad \text{ and }\quad \Ls^2:=\Hs^0.
\]
Let
\begin{align*}
(u,v)_{H} &= (u,v)_{L^2} + (\nabla u,\nabla v)_{L^2}, & &u,v\in H,
\\ \langle u,v\rangle & = (u,v)_{L^2}  - ( u, \Delta v)_{L^2} =  (u , v- \Delta v)_{L^2},& &u\in V^*,\ v\in V.
\end{align*}

Let $P\in \calL(H^k(\R^3;\R^3))$ be the Helmholtz projection, i.e.\ the orthogonal projection onto $\Hs^k$ for $k\in \{0, 1, 2\}$. After applying the Helmholtz projection, \eqref{eq:Navier_Stokes} can be rewritten as
\eqref{eq:SEE}
where
\begin{equation}
\label{eq:choice_ABFG_NS}
\begin{aligned}
A(u) &= A_0 u - F_1(u) - F_2(u) := -\P \Delta u +\P[(u\cdot \nabla)u]+ \P[\phi_N(|u|^2)  u],
\\ B(t,u)_n &= (B_0(t) u)_n + (G(t,u))_n := \P[(\btwod_{n}\cdot\nabla) u] +\P g_n(u).
\end{aligned}
\end{equation}

From Theorem \ref{thm:globalclas} we will derive the following result.
\begin{theorem}[Global well-posedness]
Suppose that Assumption \ref{ass:TSNS} holds.
Then for every $u\in L^0_{\F_0}(\Omega,\Hs^1)$, \eqref{eq:Navier_Stokes}  has a unique global solution
\begin{equation*}
u\in C([0,\infty);\Hs^1)\cap L^2_{\rm loc}([0,\infty);\Hs^2) \ a.s.,
\end{equation*}
and for every $T>0$ there exists a constant $C_T$ independent of $u_0$ such that
\begin{align*}
\E \|u\|_{C([0,T];\Hs^1)}^2 + \E\|u\|_{L^2(0,T;\Hs^2)}^2\leq C_T(1+\E\|u_0\|_{H^1(\R^3)}^2).
\end{align*}
Finally,  $u$ depends  continuously on the initial data $u_0$ in probability in the sense of Theorem \ref{thm:contdepdata} with $H=\Hs^1$ and $V = \Hs^2$.
\end{theorem}

\begin{proof}
To economize the notation, we write $L^2$ instead of $L^2(\R^3;\R^3)$ etc.

{\em Step 1: Assumption \ref{ass:condFG}\eqref{it:condFG1} holds}. Let $\varepsilon>0$ be fixed. Then
\begin{align*}
\langle  A_0u,u \rangle = (- \Delta u,u-\Delta u)_{L^2}
&\geq \|\Delta u\|_{L^2}^2 - \|u\|_{L^2} \|\Delta u\|_{L^2}\\
& \geq (1-\varepsilon)\|\Delta u\|_{L^2}^2 - C_{\varepsilon}\|u\|_{L^2}.
\end{align*}
Since $\|\P\|_{\calL(L^2)}\leq 1$ and commutes with derivatives, for the $B_0$-part we can write
\begin{align*}
\nn B_0 u\nn_{H}^2  & \leq \sum_{n\geq 1} \|(b_n\cdot \nabla)u\|_{L^2}^2 + \|\nabla[(b_n\cdot \nabla)u]\|_{L^2}^2.
\end{align*}
By \eqref{eq:stochparSNS} the first summand satisfies
\begin{align*}
\sum_{n\geq 1} \|(b_n\cdot \nabla)u\|_{L^2}^2 = \sum_{i,j=1}^3 \int_{\R^3}\sigma^{i,j} \partial_i u \cdot \partial_j u \, \dd x
\leq  \|\nabla u\|_{L^2}^2.
\end{align*}
The second summand is of second order term and satisfies
\begin{align*}
\sum_{n\geq 1}\|\nabla[(b_n\cdot \nabla)u]\|_{L^2}^2
&\leq C_{\varepsilon}\sum_{n\geq 1} \big\||\nabla b_n|\, | \nabla u|\big\|_{L^2}^2 + (1+\varepsilon) T_2.
 \end{align*}
The $C_{\varepsilon}$-term is again of first order:
\begin{align*}
C_{\varepsilon}\sum_{n\geq 1} \big\||\nabla b_n|\, | \nabla u|\big\|_{L^2}^2
\leq C_{\varepsilon} \|b\|_{W^{1,\infty}(\ell^2)}^2 \|\nabla u\|_{L^2}^2.
\end{align*}
For the second order term $T_2$, by \eqref{eq:stochparSNS}, we can write
\begin{align*}
T_2  & = \sum_{i,j,k,\ell=1}^3 \int_{\R^3}\sigma^{i,j} \partial_\ell \partial_i u^k \partial_{\ell }\partial_j u^k \, \dd x
 \leq \sum_{j,k,\ell=1}^3 \int_{\R^3}|\partial_j\partial_\ell  u^k|^2  \, \dd x.
\end{align*}
For later purposes we note that this gives
\begin{align*}
\nn B_0(t,u)\nn_{H}\leq C(1+\|u\|_{V}).
\end{align*}
On the other hand,  by integration by parts (and approximation)
\begin{align*}
\|\Delta u\|_{L^2}^2
= \sum_{i,j,k=1}^3\int_{\R^3} \partial_i^2 u^k \partial_j^2 u^k  \, \dd x
= \sum_{j,k=1}^3\int_{\R^3} |\partial_i \partial_{j} u^k|^2 \, \dd x.
\end{align*}
Therefore, collecting terms and choosing $\varepsilon>0$ small enough we obtain that
\begin{align}\label{eq:TSNScoercivityABpart}
\langle  A_0u,u \rangle - \frac{1+\varepsilon}{2} \nn B_0 u\nn_{H}^2
\geq \frac{1-2\varepsilon}{2}\|u\|_{V}^2 - C_{\varepsilon,\delta,b} \|u\|_{H}^2.
\end{align}

{\em Step 2: $F_1, F_2$ and $G$ in \eqref{eq:choice_ABFG_NS} satisfy Assumption \ref{ass:condFG}}.
For $u,v\in V$ with $\|u\|_{H}\leq n$ and $\|v\|_{H}\leq n$, we have
\begin{align*}
\|F_1(u) - F_1(v)\|_{V^*}&\leq \|(u\cdot \nabla) u - (v\cdot \nabla) v\|_{L^2}
\\ & \leq \|(u\cdot \nabla) (u-v)\|_{L^2}+ \|((u-v)\cdot \nabla) v\|_{L^2}.
\\ & \leq \|u\|_{L^{12}} \|\nabla(u-v)\|_{L^{12/5}} + \|u-v\|_{L^{12}} \|\nabla v\|_{L^{12/5}}
\\ & \leq (\|u\|_{H^{\frac54,2}}+ \|v\|_{H^{\frac54,2}})\|u-v\|_{H^{\frac54,2}},
\end{align*}
where we used H\"older's inequality with $\frac12 = \frac1{12}+\frac{5}{12}$, and  Sobolev embedding with $\frac{5}{4} - \frac32 = -\frac{3}{12}$ and $\frac{5}{4} - \frac32 = 1-\frac{3}{\frac{12}{5}}$. Setting $\beta_1 = \frac58$
we find that
\begin{align*}
\|F_1(u) - F_1(v)\|_{V^*} \leq (\|u\|_{\beta_1} + \|u\|_{\beta_2}) \|u-v\|_{\beta_1}.
\end{align*}
Therefore, we can set $\rho_1=1$ and thus the required condition in Assumption \ref{ass:condFG} since $2\beta_1\leq 1+\frac{1}{\rho_1+1}$.

For $F_2$ note that for $u,v\in V$ with $\|u\|_{H}\leq n$ and $\|v\|_{H}\leq n$ we have
\begin{align*}
\|F_2(u) - F_2(v)\|_{V^*}&\leq \|(\phi_N(|u|^2) - \phi_N(|v|^2)) u\|_{L^2} + \|\phi_N(|v|^2) (u-v)\|_{L^2}
\\ & \leq 2 \|(|u|^2 - |v|^2) u\|_{L^2} + 2\||v|^2(u-v)\|_{L^2}
\\ & \leq 4 \|(|u|^2 + |v|^2)(u-v)\|_{L^2}
\\ & \stackrel{(i)}{\leq} 4 (\|u\|_{L^6}^2+ \|v\|_{L^6}^2)\|u-v\|_{L^6}
\\ & \stackrel{(ii)}{\leq} C (\|u\|_{H}^2+ \|v\|_{H}^2)\|u-v\|_{H}
\\ & \leq 2n^2 C \|u-v\|_{H},
\end{align*}
where we used H\"older's inequality in $(i)$, and Sobolev embedding in $(ii)$.
Thus $F_2$ satisfies the required condition in Assumption \ref{ass:condFG} with any $\beta_2\in (\frac12,1)$ and $\rho_2 = 0$, so in particular we could take $\beta_2 = \beta_1$ and $\rho_2 = \rho_1$ as before.

For $G$ let $u,v\in V$ with $\|u\|_{H}\leq n$ and $\|v\|_{H}\leq n$, and note that
\begin{align*}
\nn G(t,u) - G(t,v)\nn_{H}^2 & \leq \|g(t,\cdot, u) -  g(t,\cdot, v)\|_{L^2(\R^3)}+ \|\partial_{x} g(t,\cdot, u) -  \partial_{x} g(t,\cdot, v)\|_{L^2(\R^3)}  \\ & \qquad +\|\partial_y g(t,\cdot, u)\nabla u -  \partial_y  g(t,\cdot, v)\nabla v\|_{L^2(\R^3)}
\end{align*}
By \eqref{eq:condgLipSNS} the first two terms can be estimated by $L_g \|u-v\|_{L^2(\R^3)}$. Concerning the last term we note that by \eqref{eq:condgLipSNS}
\begin{align*}
& \|\partial_y g(t,\cdot, u)\nabla u -  \partial_y  g(t,\cdot, v)\nabla v\|_{L^2} \\ & \leq
\|\partial_y g(t,\cdot, u)(\nabla u-\nabla v)\|_{L^2} + \|(\partial_y  g(t,\cdot, u) - \partial_y  g(t,\cdot, v))\nabla v\|_{L^2}
\\ & \leq L_g\|\nabla u-\nabla v\|_{L^2} + L_{g}'\|u-v\|_{L^{\infty}} \|\nabla v\|_{L^2}
\\ & \leq C_n \|u-v\|_{\beta_3},
\end{align*}
where in the last step we used $\|\nabla v\|_{L^2}\leq n$ and Sobolev embedding with $\beta_3\in(3/4,1)$. Similarly, by using \eqref{eq:condgzeroSNS}, one can check
\begin{align}\label{eq:SNSGgrowth}
\begin{aligned}
\nn G(t,u)\nn_{H}&\leq \nn G(t,u) - G(t,0)\nn_{H} + \nn G(t,0)\nn_{H}
\leq C(\|u\|_{H} + 1).
\end{aligned}
\end{align}

{\em Step 3: \eqref{eq:pcoercivityclas1} holds}. By \eqref{eq:TSNScoercivityABpart}, \eqref{eq:SNSGgrowth}, and the elementary estimate $(x+y)^2\leq (1+\varepsilon)x^2+ C_{\varepsilon} y^2$, it is enough to show that
\begin{align}\label{eq:estFSNS}
\langle F(u),u\rangle  & \leq \frac14\|u\|_{V}^2+M(\|u\|_{H}^2 +1).
\end{align}
Recall that $F_1$ and $F_2$ are as in \eqref{eq:choice_ABFG_NS}.
For $F_1$ we have
\begin{align*}
\langle F_1(u),u\rangle & \leq  \|u\|_{V} \|F_1(u)\|_{V^*}
\leq \frac14 \|u\|_{V}^2 + \|F_1(u)\|_{V^*}^2
\leq \frac14 \|u\|_{V}^2 + \int_{\R^3} |u|^2 |\nabla u|^2\, \dd x.
\end{align*}
For $F_2$, using that $\phi_N(x)\geq x-N$ for all $x\geq 0$ gives
\begin{align*}
\langle F_2(u),u\rangle & = -  \langle u, \phi_N(|u|^2) u\rangle
\\ & = -  \int_{\R^3} |u|^2 \phi_N(|u|^2)\, \dd x  + \int_{\R^3} u \cdot \Delta [u\, \phi_N(|u|^2)]\, \dd x
\\ & = -  \int_{\R^3} |u|^2 \phi_N(|u|^2) \, \dd x  - \int_{\R^3} |\nabla u|^2 \phi_N(|u|^2)\, \dd x
- 2\int_{\R^3} |u|^2|\nabla u|^2 \phi_N'(|u|^2) \, \dd x
\\ & \leq - \int_{\R^3} |\nabla u|^2 (|u|^2-N)\, \dd x
\end{align*}
where we used that $\phi_N(x)=x-N$ for $x\geq N+1$.
Combining the estimates for $F_1$ and $F_2$, we obtain \eqref{eq:estFSNS}. It remains to apply Theorems \ref{thm:globalclas} and \ref{thm:contdepdata}.
\end{proof}

\begin{remark}\label{rem:DirichletSNS}
One can also try to consider Dirichlet boundary conditions. When working with unweighted function spaces, this leads to serious difficulties as the noise needs to map into the right function spaces with boundary conditions (after applying the Helmholtz projection), which leads to strange assumptions. Moreover, the Helmholtz does not commute with the differential operators, which make the analysis more involved. Also the spaces are more involved since for $V$ one needs to take the divergence free subspace of $H^2(\Dom)\cap H^1_0(\Dom)$ and for $H$ the divergence free subspace of $H^1_0(\Dom)$. In that way $V^*$ can be identified with the divergence free subspace of $L^2(\Dom)$. Note that the divergence free subspace of $C^\infty_c(\Dom)$ is not dense in $V$, and the dual (with respect to $H$) of the closure of the later space of test functions is not the divergence free subspace of $L^2(\Dom)$ (cf. Example \ref{ex:strong_setting}).
\end{remark}

\begin{remark}
By estimating the $F_1$ term as
\[\langle u, F_1(u)\rangle \leq (1-\delta)\|u\|_{V}^2 + \frac{1}{4(1-\delta)}\int_{\R^3} |u|^2 |\nabla u|^2\, \dd x,\]
for suitable $\delta\in (0,1)$,
and strengthening the stochastic parabolicity condition \eqref{eq:stochparSNS}, one can also consider $g$ with quadratic growth in the $y$-variable (see Subsection \ref{ss:AllenCahn} for a related situation).
\end{remark}

\begin{remark}
In Assumption \ref{ass:TSNS}, the condition $b^j\in W^{1,\infty}(\R^3;\ell^2)$ can be weakened to $b^j\in W^{1,\infty}(\R^d;\ell^2)+ W^{1,3+\delta}(\R^3;\ell^2)$ for some $\delta>0$. The reader is referred to the proof of Theorem \ref{thm:AC} below for details.
\end{remark}

\subsection{Stochastic second order equations}\label{ss:second}
Below we consider a second order problem with a gradient noise term, and nonlinearities $f$ and $g$ which does not need to be of linear growth. These type of equations have been considered in many previous works, and below we merely indicate how far one can get using our improved variational framework. In particular, with our techniques one can extend the class of examples in \cite[Example 5.1.8]{LR15} in several ways. We will only deal with the weak setting (see Example \ref{ex:weaksetting}). The strong setting (see Example \ref{ex:strong_setting}) will be considered in Sections \ref{ss:AllenCahn} and \ref{ss:quasi} to treat the Allen-Cahn equation for $d\in \{2, 3, 4\}$ and a quasi-linear problem for $d=1$.

On an open and bounded domain $\Dom\subseteq \R^d$, we consider
\begin{align}
\label{eq:reaction_diffusion}
\left\{
\begin{aligned}
\dd u  &= \big[\div(a\cdot\nabla u)  + f(\cdot, u) + \div(\of(\cdot, u))\big]\, \dd t
\\  & \qquad \qquad \qquad \qquad
+ \sum_{n\geq 1}  \big[(b_{n}\cdot \nabla) u+ g_{n}(\cdot,u) \big]\, \dd w^n_t,\ \ &\text{ on }\Dom,\\
u &= 0, \ \ &\text{ on }\partial\Dom,
\\
u(0)&=u_{0},\ \ &\text{ on }\Dom,
\end{aligned}\right.
\end{align}
where $(w^n_t\,:\,t\geq 0)_{n\geq 1}$ are independent standard Brownian motions.

\begin{assumption}\label{ass:2ndorder1}
Suppose that
\[
\rho_1 \in \left\{
\begin{aligned}
&[0,3] & \text{if }& d=1,\\
& [0,2) & \text{if }& d=2,\\
& \Big[0,\frac{4}{d}\Big] & \text{if }& d\geq 3,
\end{aligned}\right.
\qquad \text{and} \qquad \rho_2,\rho_3\in \Big[0,\frac{2}{d}\Big],
\]
and
\begin{enumerate}[{\rm(1)}]
 \item\label{ass:2ndorder11} $\am^{j,k}:\R_+\times \O\times \Dom\to \R$ and $b^j:=(\bm^{j}_{n})_{n\geq 1}:\R_+\times \O\times \Dom\to \ell^2$ are $\Progress\otimes \Borel(\Tor^d)$-measurable and uniformly bounded.
\item\label{ass:2ndorder12}
There exists $\ellip>0$ such that a.e.\ on $ \R_+\times \O\times \Dom$,
$$
\sum_{j,k=1}^d \Big(a^{j,k}(t,x)-\frac12\sum_{n\geq 1} b^j_{n}(t,x)b^k_{n}(t,x)\Big)
 \xi_j \xi_k
\geq  \ellip |\xi|^2 \quad \text{ for all }\xi\in \R^d.
$$
\item\label{it:growth_nonlinearities}
The mappings $f:\R_+\times \O\times \Dom\times \R\to \R$, $\of:\R_+\times \O\times \Dom\times \R\to \R^d$ and $g:=(g_{n})_{n\geq 1}:\R_+\times \O\times \Dom\times \R\to \ell^2$,
are $\Progress\otimes \Borel(\Dom)\otimes \Borel(\R)$-measurable, and there is a constant $C$ such that a.e.\ on $\R_+\times \O\times \Dom$ and $y\in\R$,
\begin{align*}
 |f(\cdot,y)-f(\cdot,y')|
& \leq C(1+|y|^{\rho_1}+|y'|^{\rho_1})|y-y'|,
\\ |\of(\cdot,y)-\of(\cdot,y')|
& \leq C(1+|y|^{\rho_2}+|y'|^{\rho_2})|y-y'|,
\\ \|g(\cdot,y)-g(\cdot,y')\|_{\ell^2}
&\leq C(1+|y|^{\rho_3}+|y'|^{\rho_3})|y-y'|,
\\ |f(\cdot,y)|&\leq C(1+|y|^{\rho_1+1}),
\\ |\of(\cdot,y)|& \leq C (1+|y|^{\rho_2+1}),
\\ \|g(\cdot,y)\|_{\ell^2} &\leq C(1+|y|)^{\rho_3+1}. \ \
\end{align*}
\item\label{it:secondcoercity} There exist $M,C\geq 0$ and $\eta>0$ such that a.e.\ in $\R_+\times\Omega$ for all $u\in C^\infty_c(\Dom)$
\begin{align*}
&(\am \nabla u, \nabla u)_{L^2(\Dom)} + (\of(\cdot, u), \nabla u)_{L^2(\Dom)} -  (f(\cdot, u),u)_{L^2(\Dom)}  \\
 & -(\tfrac12+\eta)  \sum_{n\geq 1} \|(b_n \cdot \nabla) u + g_n(\cdot, u)\|_{L^2(\Dom)}^2
\geq  \theta \|\nabla u\|^2_{L^2(\Dom)} - M\|u\|^2_{L^2(\Dom)} -C.
\end{align*}
\end{enumerate}
\end{assumption}
Condition \eqref{it:secondcoercity} is technical, but can be seen as a direct translation of the coercivity condition \eqref{eq:pcoercivityclas1}. Simpler sufficient conditions will be give in Examples \ref{ex:seconddiff1} and \ref{ex:seconddiff2} below. Moreover, some simplification will also be discussed in Lemma \ref{lem:simplifcoersecond}.

In order to formulate \eqref{eq:reaction_diffusion} as \eqref{eq:SEE} we set $U = \ell^2$, $H = L^2(\Dom)$, $V = H^1_0(\Dom)$ and $V^* = H^{-1}(\Dom)$. Note that for $\beta\in [1/2,1)$, $V_{\beta} = [V, V^*]_{\beta} \hookrightarrow  H^{2\beta-1}(\Dom)$ (see Example \ref{ex:weaksetting}).

Let $A_0:\R_+\times\Omega\to \calL(V,V^*)$ and $B_0:\R_+\times \Omega\to \calL(V, \calL_2(U,H))$ are given by
\begin{align*}
 A_0(t) u &= \div(a(t,\cdot)\cdot \nabla u),
\\ (B_0(t) u)_n & = (b_{n}(t,\cdot)  \cdot \nabla) u.
\end{align*}
Let $F = F_1+F_2$, where $F_1,F_2:\R_+\times\Omega\times V\to V^*$ and $G:\R_+\times\Omega\times V\to \calL_2(U,H)$ be given by
\begin{align*}
F_1(t,u)(x) &= f(t,x,u(x)), \qquad F_2(t,u)(x) = \div [\of(t,x,u(x))], \\  (G(t,u))_n(x) &= g_n(t,x,u(x)).
\end{align*}
We say that $u$ is a solution to \eqref{eq:reaction_diffusion} if
$u$ is a solution to \eqref{eq:SEE} with the above definitions.

In order to prove global well-posedness we check the conditions of Theorem \ref{thm:globalclas}. It is standard to check that Assumption \ref{ass:condFG}\eqref{it:condFG0}-\eqref{it:condFG1} are satisfied. To check \eqref{it:condFG2} we only consider the local Lipschitz estimates, since the growth conditions can be checked in the same way. Note that for $F_1$ we have
\begin{align*}
\|&F_1(t,u) - F_1(t,v)\|_{V^*} \lesssim \|F_1(t,u)-F_1(t,v)\|_{L^r(\Dom)}   & \text{(by Sobolev embedding)}
\\ & \lesssim \|(1+|u|^{\rho_1}+|v|^{\rho_1}) (u-v)\|_{L^{r}(\Dom)} & \text{(by Assumption \ref{ass:2ndorder1})}
\\ & \lesssim (1+\|u\|^{\rho_1}_{L^{r(\rho_1+1)}(\Dom)}+\|v\|^{\rho_1}_{L^{r(\rho_1+1)}(\Dom)} )\|u-v\|_{L^{r(\rho_1+1)}(\Dom)} & \text{(by H\"older's inequality)}
\\ & \lesssim (1+\|u\|^{\rho_1}_{\beta_1}+\|v\|^{\rho_1}_{\beta_1})  \|u-v\|_{\beta_1} & \text{(by Sobolev embedding).}
\end{align*}
First let $d\geq 3$. In the first Sobolev embedding we choose $r\in (1, \infty)$ such that $-\frac{d}{r} = -1-\frac{d}{2}$. The second Sobolev embedding requires
$2\beta_1 -1- \frac{d}{2} \geq -\frac{d}{r(\rho_1+1)} = -\frac{1}{\rho_1+1} - \frac{d}{2(\rho_1+1)}$. In Assumption \ref{ass:condFG}\eqref{it:condFG2}, we need $2\beta_1\leq 1 + \frac{1}{\rho_1+1}$ and hence
\[\frac{d}{2}-\frac{1}{\rho_1+1} - \frac{d}{2(\rho_1+1)} \leq \frac{1}{\rho_1+1}.\]
The latter is equivalent to $\rho_1\leq \frac{4}{d}$. If $d=1$, then we can take $r=1$, and this leads to
\[\frac{3}{2} -\frac{1}{\rho_1+1} \leq 2\beta_1 \leq 1 + \frac{1}{\rho_1+1},\]
which holds if $\rho_1\leq 3$. If $d=2$, we can take $r = 1+\delta$ for any $\delta>0$, and this leads to $\rho_1<2$. In all cases we can take $2\beta_1 = 1 + \frac{1}{\rho_1+1}$.

To prove the estimates for $\of$ we argue similarly:
\begin{align*}
&\|F_2(t,u) - F_2(t,v)\|_{V^*} \lesssim \|\of(t,\cdot,u) - \of(t,\cdot,v)\|_{L^2(\Dom)} \\ & \lesssim \|(1+|u|^{\rho_2}+|v|^{\rho_2}) (u-v)\|_{L^{2}(\Dom)} & \text{(by Assumption \ref{ass:2ndorder1})}
\\ & \lesssim (1+\|u\|^{\rho_2}_{L^{2(\rho_2+1)}(\Dom)}+\|v\|^{\rho_2}_{L^{2(\rho_2+1)}(\Dom)}) \|u-v\|_{L^{2(\rho_2+1)}(\Dom)} & \text{(by H\"older's inequality)}
\\ & \lesssim (1+\|u\|^{\rho_2}_{\beta_2}+\|v\|^{\rho_2}_{\beta_2})  \|u-v\|_{\beta_2} & \text{(by Sobolev embedding).}
\end{align*}
In the Sobolev embedding we need $2\beta_2 -1- \frac{d}{2} \geq -\frac{d}{2(\rho_2+1)}$. In Assumption \ref{ass:condFG}\eqref{it:condFG2}, the condition $2\beta_2\leq 1 + \frac{1}{\rho_2+1}$ leads to $\rho_2\leq \frac{2}{d}$.
Moreover, we can consider the critical case $2\beta_2 = 1+\frac{1}{\rho_2+1}$.

To prove the estimates for $G$ we argue similarly:
\begin{align*}
&\|G(t,u) - G(t,v)\|_{L^2(\Dom;\ell^2)}  \\ & \lesssim \|(1+|u|^{\rho_3}+|v|^{\rho_3}) (u-v)\|_{L^{2}(\Dom)} & \text{(by Assumption \ref{ass:2ndorder1})}
\\ & \lesssim (1+\|u\|^{\rho_3}_{L^{2(\rho_3+1)}(\Dom)}+\|v\|^{\rho_3}_{L^{2(\rho_3+1)}(\Dom)} )\|u-v\|_{L^{2(\rho_3+1)}(\Dom)} & \text{(by H\"older's inequality)}
\\ & \lesssim (1+\|u\|^{\rho_3}_{\beta_3}+\|v\|^{\rho_3}_{\beta_3})  \|u-v\|_{\beta_3} & \text{(by Sobolev embedding).}
\end{align*}
As before we need $2\beta_3 -1- \frac{d}{2} \geq -\frac{d}{2(\rho_3+1)}$ and $\rho_3\leq \frac{2}{d}$.
Moreover, we can take $2\beta_3 = 1+\frac{1}{\rho_3+1}$.

Finally we note that the coercivity condition \eqref{eq:pcoercivityclas1} coincides with Assumption \ref{ass:2ndorder1}\eqref{it:secondcoercity}.
Therefore, Theorem \ref{thm:globalclas} gives the following:

\begin{theorem}[Global well-posedness]
\label{thm:second}
Suppose that Assumption \ref{ass:2ndorder1} holds. Let $u_0\in L^0_{\F_0}(\Omega; L^2(\Dom))$.
Then \eqref{eq:reaction_diffusion} has a unique global solution
\begin{equation}\label{eq:solspacesecond}
u\in C([0,\infty);L^2(\Dom))\cap L^2_{\rm loc}([0,\infty);H^1_0(\Dom)) \ a.s.
\end{equation}
and for every $T>0$ there is a constant $C_T$ independent of $u_0$ such that
\begin{align}\label{eq:aprioriLpsec}
\E \|u\|_{C([0,T];L^2(\Dom))}^2 + \E\|u\|_{L^2(0,T;H^1_0(\Dom))}^2\leq C_T(1+\E\|u_0\|_{L^2(\Dom)}^2).
\end{align}
Moreover, $u$ depends  continuously on the initial data $u_0$ in probability in the sense of Theorem \ref{thm:contdepdata} with $H=L^2(\Dom)$ and $V=H^{1}_0(\Dom)$.
\end{theorem}

\begin{remark}\label{rem:eta0second}
If Assumption \ref{ass:2ndorder1} holds with $\eta=0$, then a version of Theorem \ref{thm:second} still holds, but with \eqref{eq:aprioriLpsec} replaced by \eqref{eq:aprioriVH1}, \eqref{eq:aprioriVH3}, and \eqref{eq:aprioriVH2}. Indeed, instead one can apply
Theorem \ref{thm:globaleta0}.
\end{remark}

In the next lemma we further simplify some of the terms appearing in Assumption \ref{ass:2ndorder1}\eqref{it:secondcoercity} in special cases.
\begin{lemma}\label{lem:simplifcoersecond}
Suppose that Assumption \ref{ass:2ndorder1} holds. Suppose that $\of$ and $g$ only depends on $(t,\omega,y)$, and $b$ only depends on $(t,x,\omega)$ and additionally $\div(b) = 0$ in distributional sense. Then for al $u\in C^\infty_c(\Dom)$
\begin{enumerate}[{\rm(1)}]
\item\label{it:simplifcoersecond1} $(\of(\cdot, u), \nabla u)_{L^2(\Dom)} = 0$;
\item\label{it:simplifcoersecond2} $((b_n \cdot \nabla) u, g_n(\cdot, u))_{L^2(\Dom)} = 0$.
\end{enumerate}
Therefore, Assumption \ref{ass:2ndorder1}\eqref{it:secondcoercity} holds if there exist $M,C\geq 0$ and $\eta>0$ such that a.e.\ in $\R_+\times\Omega$ for all $y\in \R$,
\begin{align}\label{eq:fgsecondpoint}
  (f(\cdot, y),y)   + (\tfrac12+\eta) \|g(\cdot, y)\|_{\ell^2}^2
\leq   M|y|^2 +C.
\end{align}
\end{lemma}
\begin{proof}
By extending $u$ as zero we may assume that $\Dom$ is an open ball. In particular, this gives that $\Dom$ is a smooth domain.

\eqref{it:simplifcoersecond1}:
Let $$\mathcal{F}(t,y):=\int_0^{y} \of(t,y')\, \dd y', \quad y\in \R.$$
Since $\of$ is continuous, by the chain rule we obtain that for $u\in C^\infty_c(\Dom)$
$$
\div_x\big[ \mathcal{F}(u(x))\big]
= \of(u(x))\cdot \nabla u
$$
By the divergence theorem and the fact that $u=0$ on $\partial O$, we obtain
\[\int_{\Dom} \of(u(x))\cdot \nabla u \, \dd x = \int_{\Dom} \div_x\big[ \mathcal{F}(u(x))\big] \, \dd x
= \int_{\partial \Dom} F(u(x))\cdot n(x) \, \dd  S(x)=0.\]
Therefore the stated result follows.

\eqref{it:simplifcoersecond2}: We use a similar idea. Set
\[\mathcal{G}_n(t,y):= \int_0^y g_n(t,y')\, \dd y', \quad y\in \R.\]
Then the chain rule gives that
$$
\partial_{x_j}\big[\mathcal{G}_n(\cdot,x,u(x))\big]
=  g_n(\cdot,x,u(x))\partial_j u(x).
$$
Integrating by parts and arguing as before, we find that $((b_n \cdot \nabla) u, g_n(\cdot, u))_{L^2(\Dom)}$ can be written as
\begin{align*}
\int_{\partial \Dom} b_n \cdot n \, \mathcal{G}_n(\cdot, u(x)) \, \dd S(x)
 -\int_{\Dom} \div(b_n) \mathcal{G}_n(\cdot, u(x)) \, \dd x = 0
\end{align*}
where in the last equality we used $\div\,b_n=0$ that $\mathcal{G}_n(\cdot, u(x))|_{\partial\Dom}=0$ as $u|_{\partial\Dom}=0$ and $\mathcal{G}(\cdot,0)=0$.

For the final assertion note that by \eqref{it:simplifcoersecond1} and \eqref{it:simplifcoersecond2}, Assumption \ref{ass:2ndorder1}\eqref{it:secondcoercity} becomes
\begin{align*}
&(\am \nabla u, \nabla u)_{L^2(\Dom)} -(\tfrac12+\eta)  \|(b_n \cdot \nabla) u\|_{L^2(\Dom;\ell^2)}^2 \\ & -  (f(\cdot, u),u)_{L^2(\Dom)}
  -(\tfrac12+\eta)  \|g(\cdot, u)\|_{L^2(\Dom;\ell^2)}^2  \geq  \theta \|\nabla u\|^2_{L^2(\Dom)} - M\|u\|^2_{L^2(\Dom)} -C.
\end{align*}
By Assumption \ref{ass:2ndorder1} for $\eta$ small enough we can find $\theta>0$ such that
\[(\am \nabla u, \nabla u)_{L^2(\Dom)} -(\tfrac12+\eta)  \|(b_n \cdot \nabla) u\|_{L^2(\Dom;\ell^2)}\geq  \theta\|\nabla u\|_{L^2(\Dom)}.\]
Thus it remains to check
\[-  (f(\cdot, u),u)_{L^2(\Dom)}  -(\tfrac12+\eta)  \|g(\cdot, u)\|_{L^2(\Dom;\ell^2)}^2  \geq   - M\|u\|^2_{L^2(\Dom)} -C.\]
The latter follows from \eqref{eq:fgsecondpoint}.
\end{proof}

Next we specialize to the setting of the generalized Burgers equation of \cite[Example 5.1.8]{LR15}. It turns out that our setting leads to more flexibility under the mild restriction that the nonlinearities are locally Lipschitz with some polynomial growth estimate on the constants. Basically the natural restriction in our setting is given in \eqref{eq:dissLR} below which says that $y f(y)\leq M(1+|y|^2)$ which is weaker than the usual one-sided Lipschitz condition used for the local monotonicity. Moreover, we can allow a gradient noise term. Further comparison can be found in Remark \ref{rem:localmon} below. For convenience we only consider coefficients which are independent of $(t,\omega,x)$, but in principle this is not needed.
\begin{example}\label{ex:seconddiff1}
Let $d\geq 1$ and let $\Dom$ be a bounded $C^1$-domain. Consider the problem
\begin{equation}
\label{eq:reaction_diffusionexdiff1}
\left\{
\begin{aligned}
\dd u  &= [\Delta u+ f(u) + \div(\of(u))] \, \dd t
+ \sum_{n\geq 1} \big[(b_n\cdot \nabla)u + g_{n}(u)\big] \, \dd w_t^n,&\quad & \text{ on }\Dom,\\
u &= 0& \quad & \text{ on }\partial\Dom,
\\
u(0)&=u_{0},&\quad & \text{ on }\Dom.
\end{aligned}\right.
\end{equation}
Here $(b_n)_{n\geq 1}$ are real numbers such that stochastic parabolicity condition holds:
\[\theta := 1 - \frac12\|(b_n)_{n\geq 1}\|_{\ell^2}^2>0.\]
For the nonlinearities we assume that there is a constant $C\geq 0$ such that
\begin{align*}
|f(y)-f(y')| & \leq C (1+|y|^{\rho_1}+|y'|^{\rho_1})|y-y'|,
\\ |\of(y)-\of(y')| & \leq C (1+|y|^{\rho_2}+|y'|^{\rho_2})|y-y'|,
\\ |f(y)|&\leq C(1+|y|^{\rho_1+1}),
\\ |\of(y)|& \leq C(1+|y|^{\rho_2+1}),
\\ \|g(y)-g(y')\|_{\ell^2}&\leq C|y-y'|,
\\ \|(g_n(y))_{n\geq 1}\|_{\ell^2} &\leq C(1+|y|),
\end{align*}
where $\rho_1\in [0,\min\{\frac{4}{d},3\}]$ if $d\neq 2$, $\rho_1\in [0,2)$ if $d=2$, and $\rho_2\in (0,\frac{2}{d}]$ (cf.\ Assumption \ref{ass:2ndorder1}).
Suppose that the following dissipativity condition holds: there is an $M\geq 0$ such that
\begin{align}\label{eq:dissLR}
y f(y)\leq M(1+|y|^2).
\end{align}
In particular, if $d\in \{1, 2\}$, Burgers type nonlinearities are included: take $\of(y) = y^2$. Moreover, if $d=1$, Allen-Cahn type nonlinearities such as $f(y) = y-y^3$ are included as well (see Section \ref{ss:AllenCahn} for $d\in \{2, 3, 4\}$).

One can check that Assumption \ref{ass:2ndorder1} is satisfied (see Lemma \ref{lem:simplifcoersecond}). Thus Theorem \ref{thm:second} implies that for every $u_0\in L^0_{\F_0}(\Omega; L^2(\Dom))$, there exists a unique global solution $u$ to \eqref{eq:reaction_diffusionexdiff1} which satisfies \eqref{eq:solspacesecond} and \eqref{eq:aprioriLpsec}.

\end{example}
\begin{remark}\label{rem:localmon}
For comparison let us note that the usual local monotonicity condition would require $b = 0$, and the more restrictive one-sided Lipschitz estimate
\begin{align*}
(f(y)  - f(y'))(y-y')\leq C(1+|y'|^s) (y-y')^2, \  \ x,y\in \R.
\end{align*}
Note that setting $y'=0$, the latter implies \eqref{eq:dissLR}.
Concerning the growth rate at infinity of order $|\cdot|^{\rho_i+1}$ in our setting and say the $|\cdot|^{\overline{\rho}_i+1}$ in the classical variational setting (see \cite[Example 5.1.8]{LR15}), we make a comparison in Table \ref{ta:comp}.
In particular, in $d=2$ the Burgers equation is included in our setting, but not in the classical variational framework.
\begin{table}[h]
\renewcommand{\arraystretch}{1.25}
\begin{tabular}{|c|c|c|c|}
  \hline
   & $d=1$ & $d=2$ & $d=3$ \\
  \hline
  $\rho_1$ & $3$ & $<2$ & $\frac43$\\
  $\overline{\rho}_1$ & $2$ & $<2$ & $\frac43$ \\
  $\rho_2$ & $2$ & $1$ & $\frac{2}{3}$\\
  $\overline{\rho}_2$ & $2$ & $0$ & $0$ \\
  \hline
\end{tabular}
\vspace{0.2cm}
\caption{Our setting $\rho_i$; variational setting $\overline{\rho}_i$.}
\label{ta:comp}
\end{table}
\end{remark}

The coercivity condition of Assumption \ref{ass:2ndorder1}\eqref{it:secondcoercity} can be seen as a combined dissipativity condition on the nonlinearities $f$ and $g$: better $f$ gives less restrictions on $g$. Below, we give an explicit case in which it applies, where for simplicity we take $\of = 0$ and $b_n=0$.
\begin{example}\label{ex:seconddiff2}
Let $d\geq 1$ and suppose that $\rho\in [0,\min\{\frac{4}{d},3\}]$ if $d\neq 2$ and $\rho\in [0,2)$ if $d=2$. Let $\lambda>0$. Consider the problem
\begin{equation}
\label{eq:reaction_diffusionexdiff2}
\left\{
\begin{aligned}
\dd u &=\big[ \Delta u -\lambda |u|^{\rho} u \big]\, \dd t
+ \sum_{n\geq 1}  g_{n}(u) \, \dd w_t^n,&\quad &\text{ on }\Dom,\\
u &= 0 &\quad & \text{ on $\partial \Dom$},
\\
u(0)&=u_{0},&\quad &\text{ on }\Dom.
\end{aligned}\right.
\end{equation}
Let $f(y) = -\lambda|u|^{\rho} u$, and suppose $g:\R\to \ell^2$ is such that for some $\eta,C>0$
\begin{align}
 \nonumber \|g(y) - g(y')\|^2_{\ell^2}&\leq C(1+|y|^{\rho}+|y'|^{\rho})|y-y'|^2,
\\ \label{eq:coercsecondex}(\tfrac12+\eta) \|g(y)\|^2_{\ell^2}&\leq C(1+|y|^2) + \lambda|y|^{\rho+2},
\end{align}
where $y,y'\in \R$.
Then Theorem \ref{thm:second} implies that \eqref{eq:reaction_diffusionexdiff2} has a unique global solution $u$ as in \eqref{eq:solspacesecond} and \eqref{eq:aprioriLpsec} holds. Indeed, Assumptions \ref{ass:2ndorder1}\eqref{ass:2ndorder11},\eqref{ass:2ndorder12},\eqref{it:growth_nonlinearities}  are clearly satisfied with $\rho_1 = \rho$ and $\rho_3 = \rho/2$. For \eqref{it:secondcoercity} it remains to note that one has, a.e.\ on $\R_+\times \O\times \Dom$,
\[-(\tfrac12+\eta)\sum_{n\geq 1}|g_n(\cdot, y)|^2 - y f(y)\geq -C(1+|y|^2), \ \ \   \text{ and for all } y\in \R.\]
In case $\eta=0$, one can use Remark \ref{rem:eta0second} to obtain well-posedness. Here something special occurs in the case $d=1$ and $\rho\in (2, 3]$. In the latter case, \eqref{eq:coercsecondex} can be replaced by: there exists a $C>0$ such that
\[\|g(y)\|^2_{\ell^2} \leq C(1+|y|^2) + C|y|^{4}.\]
\end{example}

\subsection{Stochastic Allen-Cahn equation}\label{ss:AllenCahn}
The Allen–Cahn equation is one of the well-known reaction--diffusion equations of mathematical physics, and it is used to  describe phase transition processes. It is considered by many authors (see for instance \cite{RW13, S00_AC}). Quite often it is considered without gradient noise, but it seems natural to add (see the discussion in \cite[Section 1.3]{AVreaction-local}).

Here we consider the following stochastic Allen-Cahn equation with transport noise on $\Tor^d$:
\begin{align}
\label{eq:allen_cahn}
\left\{
\begin{aligned}
\dd u  & =  \big(\Delta u +u-u^3\big)\, \dd t+ \sum_{n\geq 1}  \big[(b_{n}\cdot \nabla) u+ g_{n}(\cdot,u)\big] \, \dd w^n_t,
&\quad &\text{ on }\Tor^d,
\\
u(0)&=u_{0},&\quad &\text{ on }\Tor^d.
\end{aligned}
\right.
\end{align}
The arguments below also apply if the term $u-u^3$ is replaced by a more general nonlinearity $f(u)$ which behaves like $u-u^3$. See Remark \ref{r:AC_more_general_f} for some comments. The main novelty in our result is that we can consider gradient noise and a quadratic diffusion term $g$, and that the equations fits into our variational framework for dimensions $d\leq 4$.

Table \ref{ta:comp} in Remark \ref{rem:localmon}  shows that \eqref{eq:allen_cahn} cannot be considered in the weak setting  (i.e.\ $H = L^2$, $V = H^1$, and $\rho_1 = 2$) if $d\geq 2$. In the current section we show that one can treat the stochastic Allen-Cahn equation for $2\leq d\leq 4$ by considering the strong setting instead (i.e.\ $V =H^2$ and $H = H^1$). In this setting local monotonicity does not holds, but fortunately local Lipschitz estimates are satisfied.

\begin{assumption}
\label{ass:AC}
Let $b^j=(b_n^j)_{n\geq 1}:\R_+\times \O\times \Tor^d\times \R\to \ell^2$ for $j\in \{1, \ldots,d\}$, and $g=(g_n)_{n\geq 1}:\R_+\times \O\times \Tor^d\times \R\to \ell^2$ be $\Progress\otimes \Borel(\Tor^d)$- and $\Progress\otimes \Borel(\Tor^d)\otimes\Borel( \R)$-measurable maps, respectively. Assume that
\begin{enumerate}[{\rm(1)}]
\item\label{it:AC_1} Suppose that there exist $\nu\in (0,2)$ such that a.e.\ on $\R_+\times\O\times \Tor^d$,
\begin{equation*}
\sum_{n\geq 1} \sum_{i,j=1}^d b_n^i b_n^j \xi_i \xi_j \leq \nu |\xi|^2 \ \ \text{ for all }\xi\in \R^d.
\end{equation*}
\item\label{it:AC_2} There exist $M,\delta>0$ such that for all $j\in \{1,\dots,d\}$ a.s.\ for all $t\in \R_+$,
$$
\|b^j(t,\cdot)\|_{W^{1,d+\delta}(\Tor^d;\ell^2)}\leq M.
$$
\item\label{it:AC_3} the mapping $(x,y)\mapsto g(\cdot,x,y)$ is $C^1(\T^d\times\R)$ a.e.\ on $\R_+\times \O$.
\item\label{it:AC_4} There exists $C\geq 0$ such that a.e.\ on $\R_+\times\O\times \Tor^d$ and for all $y,y'\in \R$,
\begin{align*}
\|\nabla_x g(\cdot,y)\|_{\ell^2}+\|g(\cdot,y)\|_{\ell^2}&\leq C(1+|y|^2),\\
\| \nabla_x g(\cdot,y)-\nabla_x g(\cdot,y')\|_{\ell^2}+\|g(\cdot,y)-g(\cdot,y')\|_{\ell^2}&\leq C(1+|y|+|y'|)|y-y'|,\\
\|\partial_y g(\cdot,y)\|_{\ell^2}&\leq C(1+|y|),\\
\|\partial_y  g(\cdot,y)-\partial_y  g(\cdot,y')\|_{\ell^2}&\leq C|y-y'|.
\end{align*}
\item\label{it:AC_5} There exist $C,\eta,\theta> 0$ such that for all $v\in C^\infty(\Tor^d)$ and a.e.\ on $\R_+\times \O$,
\begin{align*}
(\tfrac{1}{2}+\eta)  \sum_{n\geq 1} \|(b_n \cdot \nabla) v + g_n(\cdot, v)\|_{H^1}^2
&\leq \int_{\Tor^d} (|v|^4 + 3 |v|^2 |\nabla v|^2) \, \dd x\\
 &+ (1-\theta) \int_{\Tor^d} |\Delta v|^2\, \dd x + C (\|v\|^2_{H^1} +1).
\end{align*}
\end{enumerate}
\end{assumption}

Some remarks on Assumption \ref{ass:AC} may be in order. \eqref{it:AC_2} and the Sobolev embedding $H^{1,d+\delta}(\ell^2)\embed L^{\infty}(\ell^2)$ show that $\|b^k\|_{L^{\infty}(\ell^2)}\lesssim M$.
The quadratic growth assumption on $g$ is optimal from a scaling point of view \cite[Section 1]{AVreaction-local}.
\eqref{it:AC_5} is equivalent to the coercivity condition \eqref{eq:pcoercivityclas2}. Note that the conditions
\eqref{it:AC_4} and \eqref{it:AC_5} are compatible as the RHS of the estimate in \eqref{it:AC_5} allows for quadratic growth of $g$. To check \eqref{it:AC_5} it is also convenient to use that
\begin{equation}
\label{eq:isometry_H_2_integration_by_parts}
\sum_{j,k=1} \int_{\Tor^d}|\partial_{j,k}^2 v|^2\, \dd x= \int_{\Tor^d} |\Delta v|^2\, \dd x \ \ \ \text{ for all }v\in  H^2(\T^d).
\end{equation}
The above follows by approximation by smooth functions, and integration by part arguments.

As remarked above, Assumption \ref{ass:AC}\eqref{it:AC_5} allows $g$ with quadratic growth. In special cases, one can even obtain explicit description of constants. For instance, if $b=0$ and $g$ is $x$-independent, then one can check that \eqref{it:AC_5} holds if
\[(\tfrac12+\eta)\|g(y)\|_{\ell^2}^2\leq |y|^4 + C\quad  \text{and} \ \ (\tfrac12+\eta)\|\partial_y g(y)\|_{\ell^2}^2 \leq 3|y|^2 + C.\]

The main result of this subsection reads as follows.

\begin{theorem}[Global well-posedness]
\label{thm:AC}
Let $2\leq d\leq 4$.
Suppose that Assumption \ref{ass:AC} holds. Let $u_0\in L^0_{\F_0}(\O;H^1(\Tor^d))$. Then \eqref{eq:allen_cahn} has a unique global solution
$$
u\in C([0,\infty);H^1(\Tor^d))\cap L^2_{\loc}([0,\infty);H^2(\Tor^d)) \text{ a.s.\ }
$$
Moreover, for all $T\in (0,\infty)$, there exists $C_T>0$ independent of $u_0$ such that
\begin{align}
\label{eq:ACestimate_1}
\E\int_0^T \|u(t)\|^2_{H^2(\Tor^d)} \, \dd t & \leq C_T(1+\E\|u_0\|_{H^1(\Tor^d)}^2),
\\   \label{eq:ACestimate_2} \E\Big[\sup_{t\in [0,T]}\|u\|_{H^1(\Tor^d)}^2\Big] & \leq C_T(1+\E\|u_0\|_{H^1(\Tor^d)}^2),
\end{align}
Finally, $u$ depends  continuously on the initial data $u_0$ in probability in the sense of Theorem \ref{thm:contdepdata} with $H=H^1(\Tor^d)$ and $V=H^{2}(\Tor^d)$.
 \end{theorem}

\begin{remark}\label{rem:eta0Allen-Cahn}
In case Assumption \ref{ass:AC}\eqref{it:AC_5} holds with $\eta=0$, the above theorem still holds. However, the estimate \eqref{eq:ACestimate_2} has to be replaced by the weaker bounds \eqref{eq:aprioriVH1}-\eqref{eq:aprioriVH2} with $H=H^1(\Tor^d)$ and $V=H^2(\Tor^d)$.
The proof is the same as below, but one has to use Theorem \ref{thm:globaleta0} instead of Theorem \ref{thm:globalclas}.
\end{remark}

As mentioned at the beginning of this subsection, the deterministic nonlinearity in \eqref{eq:allen_cahn} does not satisfy the classical local monotonicity condition for stochastic evolution equations, and therefore there are difficulties in applying the classical framework to obtain well-posedness for \eqref{eq:allen_cahn}.

\begin{proof}
As usual, we view \eqref{eq:allen_cahn} in the form \eqref{eq:SEE} by setting $U=\ell^2$, $H=H^1(\Tor^d)$, $V=H^2(\Tor^d)$ and, for $v\in V$,
\begin{align*}
A_0 v&= -\Delta v, & B_0 v&=((b_n\cdot\nabla)v)_{n\geq 1},\\
F(\cdot,v)&=v-v^3, &  G(\cdot,v)&=(g_n(\cdot,v))_{n\geq 1}.
\end{align*}
As usual, $H$ is endowed with the scalar product $(f,g)_{H}=\int_{\Tor^d}  (fg +\nabla f\cdot \nabla g)\,\dd x$. Therefore, $V^*=L^2$ and
$
\langle f,g \rangle =\int_{\Tor^d} f g -(\Delta f) g\,\dd x
$
(cf.\ Example \ref{ex:strong_setting}). In particular,
\begin{equation}
\label{eq:AC_A_0_coercivity}
\langle A_0 v,v\rangle = \int_{\Tor^d} \big(|\nabla v|^2 + |\Delta v|^2\big)\, \dd x , \ \ \ v\in H^2(\Tor^d).
\end{equation}

The claim of Theorem \ref{thm:AC} follows from Theorem \ref{thm:globalclas} provided we check the assumptions.
Note that \eqref{eq:pcoercivityclas1} follows from Assumption \ref{ass:AC}\eqref{it:AC_5}. It remains to check
Assumption \ref{ass:condFG}.
To begin we check Assumption \ref{ass:condFG}\eqref{it:condFG1}. In the proof below we write $\norm{\cdot}_{L^2(\Tor^d)} = \|\cdot\|_{\calL_2(\ell^2;L^2(\T^d))} = \|\cdot\|_{L^2(\T^d;\ell^2)}$. Note that for all $v\in V$ and $\varepsilon>0$,
\begin{align*}
\frac12&\norm{\nabla B_0 v}_{L^2(\Tor^d)}^2
 \\ &\leq \frac{1+\varepsilon}{2} \sum_{n\geq 1}\sum_{k=1}^d \int_{\Tor^d}\Big|\sum_{j=1}^d b_n^j\partial^2_{j,k} v\Big|^2\, \dd x + C_{\varepsilon}\max_{j}\int_{\Tor^d}\|\nabla  b^j\|^2_{\ell^2} | \nabla v|^2\, \dd x\\
&\stackrel{(i)}{\leq}
(1+\varepsilon)\frac{\ellip}{2} \sum_{j,k=1} \int_{\Tor^d}|\partial^2_{j,k} v|^2\, \dd x
 + C_{\varepsilon}\max_{j}\|\nabla b^j\|_{L^{d+\delta}(\ell^2)}^2\|\nabla v\|_{L^{r}(\T^d)}^2
\\ &\stackrel{(ii)}{\leq}
(1+\varepsilon)\frac{\ellip}{2} \|\Delta v\|^2_{L^2(\T^d)} + C_{\varepsilon}M^2\|\nabla v\|_{L^{r}(\T^d)}^2.
\end{align*}
In $(i)$  we used  Assumption \ref{ass:AC}\eqref{it:AC_1} for the first term, and H\"older's for the second term with exponents $\frac{1}{d+\s}+\frac{1}{r}=\frac{1}{2}$ (where $\s$ is as in  Assumption \ref{ass:AC}\eqref{it:AC_2}). In $(ii)$ we used \eqref{eq:isometry_H_2_integration_by_parts} and Assumption \ref{ass:AC}\eqref{it:AC_2}.

Since $r\in (2,\frac{2d}{d-2})$, there exists $\mu\in (0,1)$ such that $H^{\mu}(\Tor^d)\embed L^r(\Tor^d)$  by Sobolev embedding. Thus by standard interpolation inequalities, for every $\gamma>0$
\[\|\nabla v\|_{L^{r}(\T^d)}^2\leq C_{0} \|\nabla v\|_{H^{\mu}(\T^d)}^2\leq \mu \|\Delta v\|_{L^2(\T^d)}^2  + C_{\mu} \|\nabla v\|_{L^2(\T^d)}^2.\]
Thus, by choosing $\mu>0$ small enough we obtain
\begin{align*}
\frac{1}{2}\norm{\nabla B_0 v}_{L^2}^2\leq
(1+2\varepsilon)\frac{\ellip}{2} \|\Delta v\|^2_{L^2(\T^d)} + C_{\varepsilon,\ellip} M^2\| v\|_{H^1(\T^d)}^2.
\end{align*}
Since $H^{1,d+\delta}(\Tor^d)\embed L^{\infty}(\Tor^d)$, we also have
$
\norm{B_0 v}_{L^2}^2\leq C M^2 \|v\|_{H^1}^2.
$
Therefore, choosing $\varepsilon>0$ such that $\theta:=1-(1+2\varepsilon)\frac{\ellip}{2}>0$, \eqref{eq:AC_A_0_coercivity} give
\begin{align*}
\langle A_0 v, v\rangle - \frac{1}{2}\norm{\nabla B_0 v}_{H^1(\T^d)}\geq \theta\|\Delta v\|_{L^2(\T^d)}^2 - C_{\varepsilon,\ellip}' M^2\| v\|_{H^1}^2.
\end{align*}
By \eqref{eq:isometry_H_2_integration_by_parts}, this implies Assumption \ref{ass:condFG}\eqref{it:condFG1}.

Finally, we check Assumption \ref{ass:condFG}\eqref{it:condFG2}. For $u,v\in V$ note that
\begin{align}
\label{eq:AC_estimate_F}
\|F(\cdot,u)-F(\cdot,v)\|_{L^2}
&\lesssim \|(1+|u|^2+|v|^{2})|u-v|\|_{L^2}\\
\nonumber
&\lesssim (1+\|u\|_{L^6}^2 + \|v\|_{L^6}^2) \|u-v\|_{L^6}\\
\nonumber
&\lesssim (1+\|u\|_{\beta_1}^2 + \|v\|_{\beta_1}^2) \|u-v\|_{\beta_1}
\end{align}
where $\beta_1=\frac{2}{3}$ and we used that $H^{2\beta}(\Tor^d)\embed L^6(\Tor^d)$ since $d\leq 4$. Setting $m_F=1$ and  $\rho_1=2$, the condition \eqref{eq:condbetarho} follows for $j=1$.

To estimate $G$, first observe that for all $u,v\in H^2(\T^d)$ by H\"older's inequality and Sobolev embedding (using $d\leq 4$), for $k\in \{0,1\}$,
\begin{align}\label{eq:obsallencahn}
\big\||\partial_j^k u||v|\big\|_{L^2(\T^d)}\leq
\| u\|_{H^{1,8/3}(\T^d)}\|v\|_{L^{8}(\T^d)}\lesssim \|u\|_{H^{3/2}(\T^d)}\|v\|_{H^{3/2}(\T^d)}.
\end{align}
Note that for $u,v\in H^2(\T^d)$
\begin{align*}
\norm{G(u) - G(v)}_{H^1(\T^d)}^2\leq \|g(u) - g(v)\|_{L^2(\T^d;\ell^2)}^2 + \sum_{j=1}^d\|\partial_j[g(u) - g(v)]\|_{L^2(\T^d;\ell^2)}^2.
\end{align*}
Assumption \ref{ass:AC}\eqref{it:AC_5} and \eqref{eq:obsallencahn} imply that
\begin{align*}
\|g(u) - g(v)\|_{L^2(\T^d;\ell^2)} &\lesssim \|1+|u|+|v|(u-v)\|_{L^2(\T^d)}\\ & \lesssim  (1+\|u\|_{H^{3/2}(\T^d)}+\|v\|_{H^{3/2}(\T^d)}) \|u-v\|_{H^{3/2}(\T^d)}.
\end{align*}
For the derivative term we can write
\[\|\partial_j[g(u) - g(v)]\|_{\ell^2} \leq \|\partial_{x_j} g(u) -\partial_{x_j} g(v)\|_{\ell^2}+ \|\partial_y g(u) \partial_j u -\partial_{y} g(v) \partial_j v\|_{\ell^2}.\]
By Assumption \ref{ass:AC}\eqref{it:AC_5}, the first term can be estimated as before. For the second term Assumption \ref{ass:AC}\eqref{it:AC_5} gives
\begin{align*}
\|\partial_y g(u) \partial_j u -\partial_{y} g(v) \partial_j v\|_{\ell^2}& \leq \|\partial_y g(u) (\partial_j u -\partial_j v)\|_{\ell^2} + \|(\partial_y g(u) -\partial_{y} g(v)) \partial_j v\|_{\ell^2}
\\ & \lesssim (1+|u|)|u-v| + (1+\partial_j v)|u-v|.
\end{align*}
Therefore, taking $L^2(\T^d)$-norms and applying \eqref{eq:obsallencahn} we find that
\begin{align*}
\|\partial_y g(u) \partial_j u -\partial_{y} g(v) \partial_j v\|_{L^2(\T^d;\ell^2)}\lesssim  (1+\|u\|_{H^{3/2}(\T^d)} + \|v\|_{H^{3/2}(\T^d)}) \|u-v\|_{H^{3/2}(\T^d)}.
\end{align*}
The growth estimate can be proved in a similar way, and thus Assumption \ref{ass:condFG}\eqref{it:condFG1} holds with $m_G=1$, $\rho_2=1$  and $\beta_2=\frac{3}{4}$, where we note that \eqref{eq:condbetarho} holds for $j=2$
\end{proof}

\begin{example}
\label{ex:AC_g}
Suppose that $
g_n(\cdot,v)=(\g_n v^2)_{n\geq 1}$ where $\g=(\g_n)_{n\geq 1}\in \ell^2$. For convenience we set $b_n\equiv 0$.
It is immediate to see that Assumption \ref{ass:AC}\eqref{it:AC_4} holds.  One can readily check that Assumption \ref{ass:AC}\eqref{it:AC_5} is satisfied with $\eta>0$ if $\|\g\|_{\ell^2}^2<\frac{3}{2}$, and it is satisfied with $\eta=0$ if $\|\g\|_{\ell^2}^2=\frac{3}{2}$ (see Remark \ref{rem:eta0Allen-Cahn}).
\end{example}

\begin{remark}
\label{r:AC_more_general_f}
As in the previous subsections, we may replace the nonlinearity $u-u^3$ by a more general one $f(u)$. Indeed, inspecting the above proof it is enough to assume that $F(\cdot,v):=f(\cdot,v)$ satisfies
\eqref{eq:AC_estimate_F} and that the term $\int_{\Tor^d} (|v|^4 + 3 |v|^2 |\nabla v|^2)\, \dd x$ on the RHS in the condition of Assumption \ref{ass:AC}\eqref{it:AC_5} is replaced by
$$
-\int_{\Tor^d} \big[ f'(u) |\nabla u|^2 + f(u)u\big]\, \dd x.
$$
\end{remark}

\begin{remark}[Kraichnan's noise]
In the study of fluid flows transport noise $(b_n\cdot\nabla)u$ is typically used to model turbulence, see e.g.\ \cite{F15_book,GK96,K68,MiRo04}. In Kraichnan's theory, it is important to choose $b$ as rough as possible. In this respect, Assumption \ref{ass:AC}\eqref{it:AC_2} allows us to cover only \emph{regular} Kraichnan's noise (see e.g.\ \cite[Section 5]{GI21_transport} and the references therein). For the irregular case (e.g.\ $b\in C^{\varepsilon}(\ell^2)$ for $\varepsilon>0$ small), Theorem \ref{thm:AC} cannot be applied.
However, by using $L^p$-theory one can show that global well-posedness for \eqref{eq:allen_cahn} still holds, see  \cite{AVreaction-global}.
 \end{remark}

\subsection{A stochastic quasi-linear second order equation}\label{ss:quasi}
In this section we give a toy example of a quasi-linear SPDE in one dimension to which our setting applies.
Due to the quasi-linear structure we are forced to work with the strong setting $V = H^2$ and $H = H^1$, since we need $H\hookrightarrow L^\infty$. In case one would use $L^q$-theory, then one can actually handle higher dimensions, since the Sobolev embeddings theorems become better for $q$ large. However, $L^q$-theory is outside the scope of the current paper.

On $\R$ consider the problem:
\begin{equation}
\label{eq:quasilin}
\left\{
\begin{aligned}
\dd u
&=\big[a(u) u''+  f(u)\big] \, \dd t + \sum_{n\geq 1} \big[b_n(u)u'  + g_n(u)\big] \, \dd w_t^n,&\quad &\text{ on }\R,
\\ u(0)&=u_0,&\quad &\text{ on }\R.
\end{aligned}
\right.
\end{equation}

\begin{assumption}\label{ass:local}
Suppose that $a:\R\to \R$, $b:\R\to \ell^2$, $f:\R\to \R$ and $g:\R\to \ell^2$ are mappings for which there exist $\theta>0$, $C\geq 0$ such that for all $x,y\in \R$
\begin{align*}
&a(y) - \frac12\sum_{n\geq 1} |b_n(y)|^2\geq \theta,
\qquad\quad b,g\in C^1(\R;\ell^2),
\\& a,b, b',  f, g, g' \ \text{are locally Lipschitz}, \  \text{and} \ f(0) = 0, \  g(0)=0.
\end{align*}
 \end{assumption}
It is also possible to consider $(t,\omega)$ and space-dependent coefficients $(a,b,f,g)$. In that case the conditions $f(0)=0$ and $g(0)=0$ can be weakened to an integrability condition.

In order to reformulate \eqref{eq:quasilin} as \eqref{eq:SEE}, we need some smoothness in the space $H$ in order to deal with the quasi-linear terms. Therefore, let $U = \ell^2$,  $H = H^1(\R)$, $V = H^2(\R)$ and $V^* = L^2(\R)$. Here we use $(u,v)_{H} = (u,v)_{L^2} + (u',v')_{L^2}$ and
for the duality between $V^*$ and $V$ we set $\langle u,v\rangle = (u,v)_{L^2} - (u,v'')_{L^2}$.

In order to prove local existence and uniqueness we reformulate \eqref{eq:quasilin} in the form \eqref{eq:SEE}. Let
\[A_0(u)v = -a(v)u'',  \ \ \ B_0(v)u = b(v)u', \ \ \ F(u) = f(u), \ \ \  G(u) = (g_n(u))_{n\geq 1}.\]
and set $A(u) =  A_0(u)u - F(u)$, and $B(u) = B_0(u)u + G(u)$.
For local existence and uniqueness, it remains to check that Assumption \ref{ass:condFG} holds.

The following estimates will be used below:
\begin{align}\label{eq:interpestd=1}
\|u\|_{C_b} \leq K\|u\|_{H^1}, \qquad \|u\|_{L^4} \leq  K \|u\|_{H^{1/4}}.
\end{align}
Assumption \ref{ass:condFG}\eqref{it:condFG0} is simple to verify. To check condition Assumption \ref{ass:condFG}\eqref{it:condFG1} note that for all $u\in H$ with $\|u\|_{H}\leq m$  and $v\in V$,
\begin{equation}\label{eq:lincoerc}
\begin{aligned}
 \langle &A_0(u)v, v\rangle  -  \frac12 \nn B_0(u) v\nn_{H}^2
\\ &=  - (a(u) v'', v)_{L^2} +(a(u) v'', v'')_{L^2} - \frac{1}{2} \sum_{n\geq 1} \|(b_n(u) v')'\|_{L^2}^2 + \|b_n(u)v'\|_{L^2}^2
\\ & = \Big(\big[a(u) - \frac12\sum_{n\geq 1} |b_n(u)|^2\big] v'', v''\Big)_{L^2} - R
\\ & \geq \theta\|v\|_{V}^2 - \theta \|v\|_{H}^2 - R,
\end{aligned}
\end{equation}
where, by Assumption \ref{ass:local},  for all $\varepsilon>0$ the rest term $R$ satisfies
\begin{align*}
R &= (a(u) v'', v)_{L^2} + \frac{1}{2} \sum_{n\geq 1}\big( \|b_n'(u) u'v'\|_{L^2}^2 +(b_n'(u) u'v',b_n(u)v'')_{L^2} + \|b_n(u)v'\|_{L^2}^2\big)
\\ & \leq C_m (\|v\|_{V} \|v\|_{V^*} +  \|v\|_{H}^2 + \|v\|_{H} \|v\|_{V}+  \|v\|_{H}^2)
\\ & \leq \varepsilon \|v\|_{V}^2 + C_{\varepsilon,m} \|v\|_{H}^2.
\end{align*}
In the above we used \eqref{eq:interpestd=1} and $\|u\|_H\leq m$.
Taking $\varepsilon\in (0,\theta)$, Assumption \ref{ass:condFG}\eqref{it:condFG1} follows.

To check Assumption \ref{ass:condFG}\eqref{it:condFG2} let $u,v\in V$ be such that $\|u\|_{H},\|v\|_{H}\leq m$ and $w\in V$. Then, by \eqref{eq:interpestd=1} and Assumption \ref{ass:local},
\begin{align*}
\|A_0(u)w - A_0(v)w\|_{V^*} &= \|(a(u)-a(v))w''\|_{L^2} \\ & \leq \|a(u) - a(v)\|_{L^{\infty}} \|w\|_{V} \\ & \lesssim_m \|u-v\|_{L^{\infty}} \|w\|_{V}\lesssim_m K\|u-v\|_{H} \|w\|_{V}.
\end{align*}
The growth estimate is proved in the same way. Analogously,
\begin{align*}
\nn B_0(u)w - B_0(v) w\nn_{H} &\leq \|(b(u)-b(v))w'\|_{L^2(\ell^2)}+ \|[(b(u)-b(v))w']'\|_{L^2(\ell^2)}
\\ & \leq \|(b'(u) - b'(v)) u' w'\|_{L^2(\R;\ell^2)} + \|b'(u) (u' - v') w'\|_{L^2(\ell^2)} \\ & \quad + \|(b(u)  - b(v))w''\|_{L^2(\ell^2)} + \|(b(u) - b(v))w' \|_{L^2(\ell^2)}
\\ & \lesssim_m \|u - v\|_{L^{\infty}} \|u'\|_{L^2} \|w'\|_{L^{\infty}} +
 \|u' - v'\|_{L^2} \|w'\|_{L^{\infty}}
\\ & \quad + \|u-v\|_{L^{\infty}} (\|w''\|_{L^2} +  \|w'\|_{L^2})
\\ & \lesssim_m \|u-v\|_{H} \|w\|_{V}.
\end{align*}
The growth estimate is similar.
To estimate $F$ note that for all $u,v\in V$ with $\|u\|_{H},\|v\|_{H}\leq m$
\begin{align*}
\|F(u) - F(v)\|_{V^*}&\leq C_{f,m}\|u-v\|_{L^{\infty}} \leq C_{m} \|u-v\|_{H},
\end{align*}
where we used \eqref{eq:interpestd=1}. The required growth estimate follows as well since $F(0) = 0$. Finally, as for the $B_0$-term, for $G$ we can write
\begin{align*}
\nn G(u) - G(v)\nn_{H}&\leq \|g(u) - g(v)\|_{L^2(\ell^2)} + \|g'(u)u' - g'(v)v'\|_{L^2(\ell^2)}.
\end{align*}
The first term is clearly $\lesssim_m \|u-v\|_{L^2}$ by \eqref{eq:interpestd=1}. From the latter inequality and Assumption \ref{ass:local}, the second term can be estimated as
\begin{align*}
\|g'(u)u' - g'(v)v'\|_{L^2(\ell^2)} & \leq \|g'(u)(u' - v')\|_{L^2(\ell^2)} + \|(g'(u) - g'(v))u'\|_{L^2(\ell^2)}
\\ & \lesssim_m \|u'-v'\|_{L^2} + \|u-v\|_{L^{\infty}} \|u'\|_{L^2}
\\ & \lesssim_m \|u-v\|_{H}.
\end{align*}
Again, the growth estimate follows by the above as $G(0)=0$ due to Assumption \ref{ass:local}.

From the above and Theorem \ref{thm:local} we obtain the following result.
\begin{theorem}[Local existence, uniqueness and blow-up criterion]\label{thm:localquasi}
Suppose Assumption \ref{ass:local} holds. Let $u_0\in L^0_{\F_0}(\Omega;H^1(\R))$. Then there exists a (unique) maximal solution $(u,\sigma)$ of \eqref{eq:quasilin} such that $u\in C([0,\sigma);H^1(\R))\cap L^2_{\rm loc}([0,\sigma);H^2(\R))$ a.s. Moreover,
\begin{align*}
\P\Big(\sigma<\infty, \sup_{t\in [0,\sigma)} \|u(t)\|^2_{H^1(\R)} + \int_{0}^{\sigma} \|u(t)\|_{H^2(\R)}^2\, \dd t<\infty \Big) = 0.
\end{align*}
\end{theorem}

The next condition will ensure global well-posedness.
\begin{assumption}\label{ass:global}
Suppose that $b(u)$ is constant in $u$, $f\in C^1(\R)$ and there exists $C\geq 0$ such that
\[ |a(x)|\leq C, \ \ \  x f(x)\leq C(|x|^2 +1), \ \ \ f'(x)\leq C \ \ \text{ and } \ \ \|g'(x)\|_{\ell^2}\leq C  \ \ \ x\in\R.\]
\end{assumption}

We do not know if the assumption on $b$ can be avoided.
As the proof below shows, Assumption \ref{ass:global} can be weakened to a joint condition on $(a,b,f,g)$, cf.\ Assumption \ref{ass:AC}\eqref{it:AC_5} for a similar situation.
Note that the condition on $f$ holds for the important class of functions of the form
\begin{equation*}
f(x) = -c|x|^h x + \phi(x),
\end{equation*}
where $c,h>0$, and $\phi\in C^1(\R)$ is such that $\phi(x)\leq C(1+|x|^{h+1-\delta})$ and $\phi'(x)\leq C(1+|x|^{h-\delta})$ for some $C\geq 0$ and $\delta\in (0,h)$.

\begin{theorem}[Global well-posedness]
\label{thm:globalquasi}
Suppose that Assumptions \ref{ass:local} and \ref{ass:global} hold. Let $u_0\in L^0_{\F_0}(\Omega;H^1(\R))$. Then \eqref{eq:quasilin} has a unique global solution
\begin{equation*}
u\in C([0,\infty);H^1(\R))\cap L^2_{\rm loc}([0,\infty);H^2(\R)) \ a.s.
\end{equation*}
and for every $T>0$ there exists a constant $C_T$ independent of $u_0$ such that
\begin{align*}
\E \|u\|_{C([0,T];H^1(\R))}^2 + \E\|u\|_{L^2(0,T;H^2(\R))}^2\leq C_T(1+\E\|u_0\|_{H^1(\R)}^2).
\end{align*}
Finally,  $u$ depends  continuously on the initial data $u_0$ in probability in the sense of Theorem \ref{thm:contdepdata} with $H=H^1(\R)$ and $V=H^{2}(\R)$.
\end{theorem}

\begin{proof}
By Theorem \ref{thm:globalclas} it remains to check \eqref{eq:pcoercivityclas1}.
Using that $b$ is independent of $u$ and the boundedness of $a$, taking $u=v$ in \eqref{eq:lincoerc}, one can check that $R$ can be estimated independently of $m$, and thus in the same way as before for $\eta>0$ small enough
\begin{align*}
\langle A_0(v)v, v\rangle & -  \big(\frac12 +\eta\big)\nn B_0 v\nn_{H}^2 \geq \theta'\|v\|_{V}^2 - C_{\varepsilon}'\|v\|_{H}^2.
\end{align*}

Concerning the $F$-term we have
\begin{align*}
 \langle F(v), v\rangle  &= (f(v),v)_{L^2}  - (f(v), v'')_{L^2}
 \\ & =  (f(v),v)_{L^2}  + (f'(v)v', v')_{L^2}
 \\ & \leq C(1+\|v\|^2_{L^2}) + C\|v'\|_{L^2}
 \\ & \leq C(1+\|v\|^2_{H}),
\end{align*}
and for $G$-term it is suffices to note that $\nn G(v)\nn_{H}\leq C(1+\|v\|_{H})$ by Assumption \ref{ass:global} and $G(0)=0$ (cf.\ Assumption \ref{ass:local}). Putting the estimates together we see that
\begin{align*}
\langle v, A(t,v)\rangle  - \big(\frac{1}{2}+\eta\big)\nn B(t,v)\nn_{H}^2 & =\langle A_0(v)v, v\rangle  -  (\tfrac12 +\eta)\nn B_0 v\nn_{H}^2 \\ &  \ -  \langle F(v), v\rangle - (\tfrac12+\eta)\nn G(v)\nn_{H}^2 - (G(v), B_0(v))
\\ & \geq \theta'\|v\|_{V}^2 - C'(1+\|v\|_{H}^2) - (1+2\eta)(G(v), B_0(v)).
\end{align*}
By Cauchy-Schwarz' inequality for the $\calL_2(U,H)$-inner product we obtain that for every $\delta>0$, $(G(v), B_0(v))\leq \delta \nn B_0(v)\nn_{H}^2  + C_{\delta}\nn G(v)\nn_{H}^2$, and hence \eqref{eq:pcoercivityclas1} follows.

For the continuous dependence it remains to apply Theorem \ref{thm:contdepdata}.
\end{proof}

\subsection{Stochastic Swift--Hohenberg equation}\label{ss:SH}
The stochastic Swift-Hohenberg equation has been studied by several authors using different methods (see \cite{Gao} and references therein).

On an open bounded $C^2$-domain $\Dom\subseteq \R^d$ consider
\begin{equation}
\label{eq:Swift-Hohenberg}
\left\{
\begin{aligned}
\dd u  &= [-\Delta^2 u - 2\Delta u+f(\cdot, u)]\, \dd t + \sum_{n\geq 1}g_n(\cdot, u) \, \dd w_t^n,&\quad & \text{ on } \Dom,
\\ u& =0, \ \  \text{ and } \ \ \Delta u =0,&\quad & \text{ on }\partial \Dom,
\\ u(0)&=u_0,&\quad  &\text{ on } \Dom.
\end{aligned}\right.
\end{equation}
Unbounded domains could also be considered using a variation of the assumptions below.

\begin{assumption}\label{ass:SH}
Let $d\geq 1$, and
\[
\rho \in
\left\{
\begin{aligned}
&\Big[0,\frac{d+4}{d}\Big] &\quad & \text{if \ $d\in\{1,2,3\}$},
\\ &[0,2) & &\text{if \ $d= 4$},
\\ &\Big[0,\frac{8}{d}\Big] & &\text{if \ $d\geq 5$}.
\end{aligned}\right.
\]
Suppose that $f\in C^1(\R)$ and $g:[0,\infty)\times \Omega\times\Dom\times\R^{1+d}\to \ell^2$ are $\Progress\times \Dom \times \mathcal{B}(\R^{1+d})$ and there exist constants $C,\eta>0$ such that a.s.\ for all $t\in \R_+$, $x\in \Dom$, $y,y'\in\R$, $z,z'\in \R^d$
\begin{align*}
 |f(y)-f(y')|
& \leq C (1+|y|^{\rho}+|y'|^{\rho})|y-y'|,
\\ |f(y)|&\leq C(1+|y|^{\rho+1}),
\\ \|g(t,x,y,z)-g(t,x,y',z')\|_{\ell^2}
&\leq C (1+|y|^{\frac{\rho}{2}} + |y'|^{\frac{\rho}{2}}) |y-y'| + C|z-z'|,
\\ \|(g(t,x,y,z))_{n\geq 1}\|_{\ell^2} &\leq C(1+|y|^{\frac{\rho}{2}} + |y'|^{\frac{\rho}{2}})(1+|y|)+C(1+|z|),
\\ f(y) y + (\tfrac{1}{2}+\eta) \|g(t,x,y,z)\|_{\ell^2}^2& \leq C(1+|y|^2+|z|^2).
\end{align*}
\end{assumption}
The classical Swift-Hohenberg nonlinearity $f(y) = c y - y^3$ with $\rho = 2$ satisfies the above condition in the physical dimensions $d\in \{1,2,3\}$. Local monotonicity and \eqref{eq:growthB} hold if $g$ has linear growth, but not in the case $g$ has quadratic growth which we also allow.

\begin{theorem}[Global well-posedness]
\label{thm:SH}
Suppose that Assumption \ref{ass:SH} holds. Let $u_0\in L^0_{\F_0}(\Omega; L^2(\Dom))$.
Then \eqref{eq:Swift-Hohenberg} has a unique global solution
\begin{equation}\label{eq:solspaceSH}
u\in C([0,\infty);L^2(\Dom))\cap L^2_{\rm loc}([0,\infty);H^2(\Dom)\cap H^1_0(\Dom)) \ a.s.,
\end{equation}
and for every $T>0$ there exists a constant $C_T$ independent of $u_0$ such that
\[\E \|u\|_{C([0,T];L^2(\Dom))}^2 + \E\|u\|_{L^2(0,T;H^2(\Dom))}^2\leq C_T(1+\E\|u_0\|^2_{L^2(\Dom)}).\]
Finally,  $u$ depends  continuously on the initial data $u_0$ in probability in the sense of Theorem \ref{thm:contdepdata} with $H=L^2(\Dom)$ and $V = H^2(\Dom)\cap H^1_0(\Dom)$.
\end{theorem}

\begin{proof}
We formulate \eqref{eq:Swift-Hohenberg} in the form \eqref{eq:SEE}.
Let $H = L^2(\Dom)$ and $V = H^2(\Dom)\cap H^1_0(\Dom)$.
Then by \eqref{eq:reiteration} for $\theta\in (0,1)$ one has
\begin{align*}
V_{\frac{1+\theta}{2}} = [H, V]_{\theta} \hookrightarrow [L^2(\Dom), H^{2}(\Dom)]_{\theta} = H^{2\theta}(\Dom),
\end{align*}
where in the last step we used the smoothness of $\Dom$ and standard results on complex or real interpolation.

Let $A = A_0 + F$ where, we define $A_0\in \calL(V,V^*)$ and $F:V\to V^*$  by
\[\langle A_0 u,v\rangle = ( \Delta u,\Delta v)_{L^2} + 2(\nabla u, \nabla v)_{L^2} \quad \text{and} \quad \langle F(u),v\rangle = -(f(u),v)_{L^2}.\]
Let $B = B_0 + G$, where $B_0 = 0$ and $G:[0,\infty)\times \Omega\times V\to \calL_2(U,H)$ is defined by
\[G(t,u)_n(x) = g_n(t,x,u(x), \nabla u(x)).\]

We check Assumption \ref{ass:condFG}\eqref{it:condFG1}.
By the smoothness and boundedness of $\Dom$ and standard elliptic theory for second order operators (see \cite[Theorem 8.8]{GT83}) there exist $\theta>0$ such that for all $u\in V$
\begin{align*}
 \langle u, A_0 u\rangle = \|\Delta u\|_{L^2(\Dom)}^2 \geq \theta\|u\|_{V}^2.
\end{align*}
In order to check Assumption \ref{ass:condFG}\eqref{it:condFG2} we start with $F$. We focus on the local Lipschitz estimate, since the growth condition can be proved in the same way. One has
\begin{align*}
& \|F(u) - F(v)\|_{V^*} \stackrel{(i)}{\lesssim} \|f(\cdot, u) - f(\cdot, v)\|_{L^r(\Dom)}
\\ & \lesssim  \|(1+|u|^{\rho}+|v|^{\rho}) (u-v)\|_{L^{r}(\Dom)} & \text{(by Assumption \ref{ass:SH})}
\\ & \lesssim (1+\|u\|^{\rho}_{L^{r(\rho+1)}(\Dom)}+\|v\|^{\rho}_{L^{r(\rho+1)}(\Dom)} )\|u-v\|_{L^{r(\rho+1)}(\Dom)} & \text{(by H\"older's inequality)}
\\ & \stackrel{(ii)}{\lesssim} (1+\|u\|^{\rho}_{H^{4\beta_1-2}(\Dom)}+\|v\|^{\rho}_{H^{4\beta_1-2}(\Dom)}) 
 \|u-v\|_{H^{4\beta_1-2}(\Dom)} & \text{(by Sobolev embedding).}
\end{align*}
In the Sobolev embedding in $(i)$ we need
\[-\frac{d}{r}\geq -2 - \frac{d}{2}  \ \ \text{and} \ \ r\in [1, 2], \]
where $r\in (1,2]$ if $-\frac{d}{r}= -2 - \frac{d}{2}$.
In the Sobolev embedding in $(ii)$ we need
\begin{align}\label{eq:condbetaSH}
4\beta_1 -2- \frac{d}{2} \geq -\frac{d}{r(\rho+1)}
\end{align}
By Assumption \ref{ass:condFG}\eqref{it:condFG2}, we also need $2\beta_1\leq 1+\frac{1}{\rho+1}$. In order to have as much flexibility as possible we take $r$ small and set $2\beta_1 = 1+\frac{1}{\rho+1}$.

For $d\geq 5$ taking $r$ such that $-\frac{d}{r}= -2 - \frac{d}{2}$, \eqref{eq:condbetaSH} leads to
\[\frac{2}{\rho+1} - \frac{d}{2} = 4\beta_1 -2 - \frac{d}{2} \geq -\frac{d}{r(\rho+1)} = -\frac{2}{\rho+1} - \frac{d}{2(\rho+1)},\]
which is equivalent to $\rho\leq \frac{8}{d}$. For $d\in \{1, 2, 3\}$ setting $r=1$, \eqref{eq:condbetaSH} leads to
\[\frac{2}{\rho+1} - \frac{d}{2} = 4\beta_1 -2 - \frac{d}{2} \geq -\frac{d}{\rho+1}.\]
which is equivalent to $\rho\leq \frac{d+4}{d}$. For $d=4$, we can take $r = 1+\varepsilon$. The same calculation leads to $\rho<2$ by taking $\varepsilon>0$ small enough.

For $G$ we have
\begin{align*}
\|G(t,u)& - G(t,v)\|_{L^2(\Dom;\ell^2)}    \lesssim
\|(1+|u|^{\rho/2}+|v|^{\rho/2}) (u-v)\|_{L^{2}(\Dom)} + \|\nabla u - \nabla v\|_{L^2(\Dom)}
\\ &\leq (1+\|u\|^{\rho}_{L^{\rho+2}(\Dom)}+\|v\|^{\rho}_{L^{\rho+2}(\Dom)} )\|u-v\|_{L^{\rho+2}(\Dom)} + \|u - v\|_{H^1}
\\ & \lesssim (1+\|u\|^{\rho/2}_{H^{4\beta_2-2}(\Dom)}+\|v\|^{\rho/2}_{H^{4\beta_2-2}(\Dom)})  \|u-v\|_{H^{4\beta_2-2}(\Dom)} + \|u - v\|_{H^1},
\end{align*}
where we used Sobolev embedding with $4\beta_2 - 2-\frac{d}{2} \geq -\frac{d}{\rho+2}$.  Setting $2\beta_2= 1+\frac{1}{\rho_2+1}$ with $\rho_2 = \rho/2$, the condition on $\beta_2$ is equivalent to $\rho\leq \frac{8}{d}$, which always holds. A similar growth condition can be checked for $\|G(t,u)\|_{L^2(\Dom;\ell^2)}$. Setting $\beta_3 = \frac{3}{4}$ and  $\rho_3 = 0$, it follows that Assumption \ref{ass:condFG}\eqref{it:condFG2} holds.

In order to check the coercivity condition \eqref{eq:pcoercivityclas1} it remains to note that the assumptions give
\begin{align*}
\langle A_0 u,u\rangle-\langle F(u),u\rangle - (\tfrac{1}{2}+\eta) \|G(t,u)\|_{\ell^2}^2& \geq \theta\|u\|_{V}^2 -C(1+\|u\|^2_{H}+\|u\|_{\frac34}^2)\\ & \geq \wt{\theta}\|u\|_{V}^2 -\wt{C}(1+\|u\|^2_{H}),
\end{align*}
where we used $\|u\|_{\frac34}\leq \varepsilon \|u\|_{V} + C_{\varepsilon} \|u\|_{H}$ for every $\varepsilon>0$.

Now the required result follows from Theorem \ref{thm:globalclas}.
\end{proof}

\begin{remark}
A version of Theorem \ref{thm:SH} also holds in the case $\eta=0$ in Assumption \ref{ass:SH} and follows from Theorem \ref{thm:globaleta0} instead. However, the estimate for the $L^2$-moment in the maximal inequality has to be replaced by the weaker estimates \eqref{eq:aprioriVH1}-\eqref{eq:aprioriVH2}.
\end{remark}

\appendix
\section{A stochastic Gronwall lemma}
\label{app:gronwall}
In this appendix we present a stochastic variant of the classical Gronwall's lemma. The following is a variant of \cite[Lemma 5.3]{GHZ09} with tail estimates probability. The lemma was already applied in the papers \cite{Primitive3, agresti2022stochastic}. Alternatively, one can often use the stochastic Gronwall inequalities of \cite[Theorem 4.1]{geiss2021sharp} and \cite[Theorem 2.1]{MeSch}.

\begin{lemma}[Stochastic Gronwall lemma]
\label{lem:gronwall}
Let $s\geq 0$, and let $\tau$ be a stopping time with values in $[s,\infty)$. Let $X,Y,f:[s,\tau)\times\O \to [0,\infty)$ are progressively measurable processes such that
a.s.\ $X$ has increasing and continuous paths, a.s.\ $Y\in L^1_{\loc}([s,\tau))$, and a.s.\ $f\in L^1(s,\tau)$. Suppose that there exist constants $\eta\geq 0$ and $C \geq 1$ such that for all stopping times $s\leq \lambda\leq \Lambda\leq \tau$
\begin{equation}
\label{eq:assumption_gronwall}
\E[X(\Lambda) ]+
\E\int_{\lambda}^{\Lambda} Y(t) \, \dd t \leq
C (\E[X(\lambda)]+\eta) + \E \Big[ (X(\Lambda)+\eta)\int_{\lambda}^{\Lambda} f(t)\, \dd t\Big],
\end{equation}
whenever the right-hand side is finite. Then one has
\begin{equation}
\label{eq:boundedness_as}
 X(\tau) + \int_s^{\tau} Y(t)\, \dd t<\infty \ \ \text{ a.s., }
 \end{equation}
where we set $X(\tau) = \lim_{t\uparrow \tau}X(t)$.
Moreover, for all $\gamma,R>0$
\begin{equation}
\label{eq:tail_probability_grownall}
\P\Big(X(\tau)+\int_s^{\tau} Y(t)\, \dd t \geq \gamma \Big)
\leq \frac{4C}{\gamma} e^{4C R} (\E[X(0)]+\eta)+ \P \Big(\int_s^{\tau}f(t)\, \dd t \geq R\Big).
\end{equation}
\end{lemma}
The proof below shows that for \eqref{eq:boundedness_as} to hold, it is enough to prove \eqref{eq:assumption_gronwall} for all $\Lambda$ such that $X(\Lambda)+\int_{s}^{\Lambda} Y(t)\,\dd t\leq K$, where $K$ is an arbitrary deterministic constant.

\begin{remark}\
\begin{itemize}
\item Choosing $R(\g) =\frac{1-\varepsilon}{4C} \ln  \g$ for $\varepsilon\in (0,1)$ and $\g$ large, \eqref{eq:tail_probability_grownall} shows that the tail probability of $X(\tau)+\int_{s}^{\tau} Y(t) \, \dd t$ converges to $0$ as $\g\to \infty$ in a quantitative way.
\item Usually Gronwall's inequality is formulated under the condition that
\begin{equation*}
\E[ X(\Lambda) ]+
\E\int_{\lambda}^{\Lambda} Y(t) \, \dd t \leq
C (\E[X(\lambda)]+\eta) + \E \Big[ \int_{\lambda}^{\Lambda} (X(t)+\eta) f(t) \, \dd t\Big],
\end{equation*}
that is stronger than \eqref{eq:assumption_gronwall}.
\item Lemma \ref{lem:gronwall} is very close to the deterministic result. Indeed, let $X$ and $f$ be deterministic and $Y\equiv 0$, $C=1$, $\eta=0$, and $s=0$. Taking $R:=  \int_0^{\tau}f(t)\, \dd t $ and \eqref{eq:tail_probability_grownall} gives
$$
 X(\tau) \leq 4C e^{(4C  \int_0^{\tau}f(t)\, \dd t) } X(0).
$$
The latter would also follows from the standard Gronwall lemma with a more precise bound on the constant.
\end{itemize}
\end{remark}

\begin{proof}[Proof of Lemma \ref{lem:gronwall}]
Without loss of generality we can assume $s=0$.
Since $\lim_{R\to \infty}\P( \int_0^{\tau}f(t)\, \dd t \geq R)=0$ by assumption,
one can check that \eqref{eq:boundedness_as} follows from \eqref{eq:tail_probability_grownall} by first letting
$\gamma\to \infty$ and then $R\to \infty$.

Hence, it remains to prove  \eqref{eq:tail_probability_grownall}. First suppose that $\eta  = 0$.
We will prove the following slightly stronger estimate with constant $2C$:
\begin{align}
\label{eq:tail_probability_sup_X_j}
\P\Big( X(\tau) +\int_0^{\tau} Y(t) \, \dd t\geq \gamma \Big)
&\leq \frac{2C}{\gamma} e^{2C R} \E[X(0)]+ \P \Big( \int_0^{\tau}f(t)\, \dd t \geq R\Big).
\end{align}
for all $R, \gamma >0$.
To this end, fix $\gamma, R>0$ and let
$$
\tau_R:=\inf \Big\{t\in [0,\tau)\,:\, \int_0^t f(s)\, \dd s\geq R \Big\}, \ \  \text{ where }\ \  \inf\emptyset:=\tau.
$$
Then $\tau_R$ is a stopping time since $f$ is progressive measurable. Note that
\begin{align*}
\P&\Big( X(\tau) +\int_0^{\tau} Y(t) \, \dd t\geq \gamma \Big)
\\ &\leq
\P\Big(X(\tau) +\int_0^{\tau} Y(t) \, \dd t\geq \gamma ,\, \int_0^{\tau} f(t)\, \dd t< R\Big)
+
\P\Big(  \int_0^{\tau} f(t)\, \dd t\geq  R\Big)\\
&\leq
\P\Big(X(\tau_R) +\int_0^{\tau_R} Y(t) \, \dd t\geq \gamma \Big)
+
\P\Big(  \int_0^{\tau} f(t)\, \dd t \geq  R\Big),
\end{align*}
where in the last inequality we used that $\tau=\tau_R$ on $\{\int_0^{\tau} f(t)\, \dd t <  R\}$. Hence, to prove \eqref{eq:tail_probability_grownall}, it remains to show that
\begin{equation}
\label{eq:intermediate_claim_grownall}
\P\Big( X(\tau_R) +\int_0^{\tau_R} Y(t) \, \dd t\geq \gamma \Big)\leq
\frac{2C}{\gamma} e^{2C R} \E[X(0)]
\end{equation}
To this end we use \eqref{eq:assumption_gronwall}.
For each $k\geq 1$ define the stopping time $\mu_k$ by
$$
\mu_k :=\inf\Big\{t\in [0,\tau_R)\,:\, X(t)+\int_{0}^t Y(s)\, \dd  s\geq k\Big\} \quad  \text{ where }\quad \inf\emptyset:=\tau_R.
$$
Then $\mu_k\leq \tau_R$ and $\lim_{k\to \infty} \mu_k = \tau_R$ a.s. By the monotone convergence theorem it suffices to prove \eqref{eq:intermediate_claim_grownall} with $\tau_R$ replaced by $\mu_k$.

Fix $k\geq 1$. We define a suitable random partition of the interval $[0,\mu_k]$ on which the integral of $f$ is small. For this we recursively define the stopping times $(\lambda_{m})_{m\geq 0}$ by $\lambda_{0}=0$ and for each $m\geq 1$,
\begin{equation*}
\lambda_{m} :=\inf\Big\{t\in [\lambda_{m-1},\mu_k]\,:\, \int_{\lambda_{m-1}}^t f(s)\, \dd s\geq \frac12\Big\} \quad  \text{ where }\quad \inf\emptyset:=\mu_k.
\end{equation*}
Since $\int_{0}^{\mu_k} f(s) \, \dd s \leq \int_{0}^{\tau_R} f(s) \, \dd s  \leq R$ a.s., we have $\lambda_{M} =\mu_k$ a.s.\ with $M:=\lceil 2 R\rceil$ independent of $k$ (where we set $\lceil n \rceil :=n+1$ for $n\in\N$).

By the assumption \eqref{eq:assumption_gronwall} and the fact that $\int_{\lambda_{m-1}}^{\lambda_{m}} f(s)\, \dd s\leq \frac12$, we find that
\begin{align*}
\E[ X(\lambda_m) ]+
\E\int_{\lambda_{m-1}}^{\lambda_{m}} Y(s) \, \dd s
\leq
C \E[X(\lambda_{m-1})] + \frac12\E[ X(\lambda_m)],
\end{align*}
for every $m\in \{1, \ldots, M\}$. Since $\E[ X(\lambda_m)] \leq \E[ X(\mu_k)] \leq k$, the above gives
\begin{equation*}
\E[X(\lambda_m) ]+
\E\int_{\lambda_{m-1}}^{\lambda_{m}} Y(s) \\, \dd s
\leq 2C\E[X(\lambda_{m-1})]
\end{equation*}
Iterating the latter (and using that $C\geq 1$) we find that
\begin{align*}
\E[ X(\lambda_{M})] + \E\int_{0}^{\lambda_{M}} Y(s) \, \dd s
&\leq 2C\E[ X(\lambda_{M-1})] + \E\int_{0}^{\lambda_{M-1}} Y(s) \, \dd  s
\\ & \leq (2C)^2\E[X(\lambda_{M-2})]  + \E\int_{0}^{\lambda_{M-2}} Y(s) \, \dd  s
\\& \leq \ldots \leq (2C)^M\E[X(0)]
\\ & \leq 2C e^{2R\ln(2C)}\E[X(0)] \leq 2C e^{2RC}\E[X(0)],
\end{align*}
where we used $M\leq 2R+1$ and $\ln(2C)\leq C$. Since  $\lambda_M = \mu_k$, Chebychev's inequality implies \eqref{eq:intermediate_claim_grownall} with $\tau_R$ replaced by $\mu_k$. This completes the proof of \eqref{eq:tail_probability_sup_X_j}.

If $\eta>0$, then we can add $\eta$ on both sides of \eqref{eq:assumption_gronwall} and replace $C$ by $2C$. It remains to apply \eqref{eq:tail_probability_sup_X_j} to the pair $(X+\eta,Y)$.
\end{proof}

\bibliographystyle{plain}
\bibliography{literature}

\end{document}